\RequirePackage{fix-cm}

\documentclass[smallextended]{svjour3}       
\smartqed  
\usepackage{graphicx}
\usepackage[dvipsnames]{xcolor}
\usepackage{subfigure}
\usepackage{psfrag}
\usepackage{enumitem}
\usepackage{amsmath,amssymb,amscd,amsxtra,amsfonts}
\usepackage{epsf,graphicx,epsfig,latexsym,cite,multirow}
\usepackage{dsfont}

\numberwithin{equation}{section}
\numberwithin{lemma}{section}
\numberwithin{theorem}{section}
\numberwithin{figure}{section}
\numberwithin{definition}{section}
\numberwithin{remark}{section}
\numberwithin{table}{section}

\spnewtheorem{algorithm}{Algorithm}{\bf}{\rm}

\newcommand{\pa}{\partial}
\newcommand{\la}{\langle}
\newcommand{\ra}{\rangle}

\newcommand{\al}{\alpha}

\newcommand{\si}{\sigma}

\newcommand{\na}{\nabla}

\newcommand{\Ga}{\Gamma}
\newcommand{\Om}{\Omega}
\newcommand{\de}{\delta}

\newcommand{\De}{\Delta}
\newcommand{\Lam}{\Lambda}
\newcommand{\lam}{\lambda}

\newcommand{\lj}{[{\hskip -1.5pt} [}
\newcommand{\rj}{]{\hskip -1.5pt} ]}
\newcommand{\bks}{\backslash}
\newcommand{\diam}{\mathrm{diam}}

\newcommand{\dist}{\mathrm{dist}}

\newcommand{\cT}{{\cal T}}
\newcommand{\cM}{{\cal M}}
\newcommand{\cE}{{\cal E}}
\newcommand{\cN}{{\cal N}}

\renewcommand{\i}{\mathbf{i}}

\renewcommand{\div}{\mathrm{div}}

\newcommand{\D}{\mathcal{D}}

\newcommand{\R}{\mathbb{R}}

\newcommand{\V}{\mathbb{V}}

\newcommand{\X}{\mathbb{X}}

\newcommand{\eps}{\epsilon}

\newcommand{\dhat}[1]{\widehat{\vphantom{\rule{3pt}{6pt}}\smash{\widehat{#1}}}}

\newcommand{\be}{\begin{eqnarray}}
\newcommand{\ee}{\end{eqnarray}}
\newcommand{\ben}{\begin{eqnarray*}}
\newcommand{\een}{\end{eqnarray*}}
\newcommand{\nn}{\nonumber}
\begin{document}

\title{An adaptive high-order unfitted finite element method
for elliptic interface problems
\thanks{This work is supported in part by China National Key Technologies R\&D Program under the grant 2019YFA0709602 and China Natural Science Foundation under the grant 118311061.}}


\author{Zhiming Chen \and Ke Li \and Xueshuang Xiang 
}


\institute{Zhiming Chen \at
              LSEC, Academy of
		Mathematics and Systems Science, Chinese Academy of Sciences, Beijing 100190, China
		and School of Mathematical Science, University of
		Chinese Academy of Sciences, Beijing 100049, China \\
              \email{zmchen@lsec.cc.ac.cn}           
           \and
           Ke Li \at
              School of Mathematical Science, University of Chinese Academy of Sciences, Beijing 100049, China, Current address:    
                  Department of Basic, Information Engineering University, Zhengzhou 450001, China \\
		 \email{like@lsec.cc.ac.cn}
    \and
    Xueshaung Xiang \at
    Qian Xuesen Laboratory of Space Technology,
    China Academy of Space Technology, Beijing 100194, China \\
    \email{xiangxueshuang@qxslab.cn}
}

\date{Received: date / Accepted: date}

\titlerunning{An adaptive unfitted FEM for elliptic interface problems}

\maketitle

\begin{abstract}
We design an adaptive unfitted finite element method on the Cartesian mesh with hanging nodes. We derive an $hp$-reliable and efficient residual type a posteriori error estimate on $K$-meshes.
A key ingredient is a novel $hp$-domain inverse estimate which allows us to prove the stability of the
finite element method under practical interface resolving mesh conditions and also prove the lower bound of the $hp$ a posteriori error estimate. Numerical examples are included.

\end{abstract}

\section{Introduction}\label{sec:1}

We consider the following model elliptic interface problem
 \be
 &&-\mbox{div}(a\na u)=f\ \ \mbox{in }\Om,\label{p1}\\
 &&\lj u \rj_\Ga=0,\ \ \lj a\na u\cdot \nu \rj_\Ga=0\ \ \mbox{on }\Ga,\label{p2}\\
 && u=g\ \ \mbox{on }\pa\Om,\label{p3}
 \ee
where $\Omega\subset\R^2$ is a bounded Lipschitz domain, $f\in L^2(\Om)$, $g\in H^{1/2}(\pa\Om)$, $\Gamma$ is a Lipschitz and piecewise
$C^2$-smooth interface which divides $\Om$ into two nonintersecting subdomains
 \ben
 \Om_1\subset\bar\Om_1\subset\Omega,\quad
 \Om_2=\Om\backslash\bar\Om_1,\quad\Gamma =\partial\Om_1\cap\pa\Om_2.
 \een
For simplicity, we assume that the coefficient $a(x)$ is positive and piecewise constant, namely,
 \ben
 a = a_1\chi_{\Omega_1} +a_2\chi_{\Omega_2},\ \ a_1,a_2>0,
 \een
where $\chi_{\Om_i}$ denotes the characteristic function of $\Om_i$,
$i=1,2$. Here $\nu$ is the unit outer normal to $\Om_{1}$, and $\lj
v\rj_\Ga:=v|_{\Om_1}-v|_{\Om_2}$ stands for the jump of a function $v$ across the
interface $\Ga$. In this paper we will assume $\Om$ is a union of bounded rectangles so that it can be partitioned by Cartesian meshes. For general Lipschitz domains we can extend the ideas developed in this paper in the framework of fictitious domain finite element methods, which will be studied in a future work.

There are extensive studies in the literature for immersed or unfitted mesh
methods which allow the interface intersecting elements in an arbitrary manner and
thus are able to avoid expensive work in the mesh generation when using body-fitted methods \cite{Babuska70, Chen98, Xu}. For low order approximations, we refer to
the immersed boundary method \cite{Peskin}, the immersed
interface method \cite{Li06}, the ghost fluid method
\cite{Liu}, the immersed finite element method \cite{Li03, Chen09}, and the extended Nitsche's method or
the cut finite element method \cite{Hansbo, Burman10}. The seminal idea of ``doubling of unknowns" in the interface element in \cite{Hansbo} has motivated studies of unfitted high order $h$-methods in \cite{Johansson, Wang, Huang, Burman18} and $hp$-methods in \cite{Massjung, Wu}. We also refer to \cite{Lehrenfeld} for the unfitted isoparametric finite element method and the recent review paper \cite{Burman15} for further references on the theory and application of unfitted finite element methods. We remark that a crucial ingredient in the design and analysis of unfitted high order finite element methods is the inverse trace inequality on curved domains for which various interface resolving mesh conditions are
introduced.

A posteriori error estimates are computable quantities in terms of the discrete solution and the input data, which provide the estimation of the discrete error and are decisive in designing efficient adaptive methods \cite{Babuska87a}. There exists an extensive literature on $hp$-residual type a posteriori finite element error estimates, see \cite{Melenk01, Melenk05} for conforming finite element methods and \cite{Houston} for discontinuous Galerkin methods. The recent work \cite{Ern} proves that the equilibrated flux a posteriori error estimate on conforming meshes is also polynomial degree robust. The convergence and quasi-optimality of $h$-adaptive methods based on a posteriori error estimates for discontinuous Galerkin methods have been studied in \cite{Karakashian}, \cite{Bonito} and the references therein. For the reliable and efficient residual type a posteriori error estimation for other unfitted finite element methods we refer to the recent work \cite{He} for immersed finite element methods and \cite{Burman20} for the cut finite element method.

The purpose of this paper is two folds. We first introduce the concept of interface deviation and prove the domain inverse estimate, which allows us to show the $hp$-stability of an unfitted finite element method under new interface resolving mesh conditions that can be easily implemented in practical computations. The unfitted finite element method is based on the idea of doubling of unknowns in \cite{Hansbo} and the idea of merging small elements with neighboring large elements in \cite{Johansson} in the framework of the local discontinuous Galerkin (LDG) method \cite{Cockburn}. Secondly, we derive a residual type
$hp$-a posteriori error estimate for the unfitted finite element method on the so called $K$-meshes with possible hanging nodes \cite{Babuska87a}. Here we extend the $hp$-quasi-interpolation operator in \cite{Melenk05} and the $hp$-local smoothing operator in \cite{Houston, Zhu} to $K$-meshes. We also show the $hp$ approximation error of unfitted finite element functions by $H^1$ functions by using the $H^{1/2}$-norm localization lemma in \cite{Faermann}. The local lower bound of our a posteriori error estimate is established by using the domain inverse estimate. This argument is different from the classical argument in \cite{Melenk01} to derive the lower bound and the result is slightly better (see the remark below Theorem \ref{thm:4.1}). We remark that for simplicity, a uniform polynomial degree is used in this paper, but the change to a variable polynomial degree over the mesh can also be considered by the method in this paper.

The paper is organized as follows. In section 2 we introduce the unfitted finite element method and prove the domain inverse estimate. In section 3 we show the upper bound of the residual type a posteriori error estimate. In section 4 we prove the efficiency of our a posteriori error estimator. In section 5 we report several numerical examples to show the effectiveness of our adaptive unfitted finite element method.

\section{The unfitted finite element method}\label{sec:2}

We first introduce the notation and the unfitted finite element method in the first subsection. Then we prove the domain inverse estimate which plays a key role in this paper. In the third subsection we prove
 the stability of our finite element method.

\subsection{Notation and the finite element method}

Let $\cT$ be a Cartesian finite element mesh with possible local refinements and hanging nodes. The elements of the mesh are (open) rectangles whose sides are parallel to the coordinate axes. For any $K\in\cT$, let $h_K$ stand for its diameter. Denote
$\cT^\Ga=\{K\in\cT:K\cap\Ga\not=\emptyset\}$ the set of interface elements. We assume
the interface $\Ga$ intersects each element $K\in\cT^\Ga$ at most twice at different (open) sides and each element $K\in\cT^{\Ga}$ includes at most one singular point of $\Ga$ where $\Ga$ is not $C^2$-smooth.

\begin{figure}[t]
	\centering
	\includegraphics[width=0.9\textwidth]{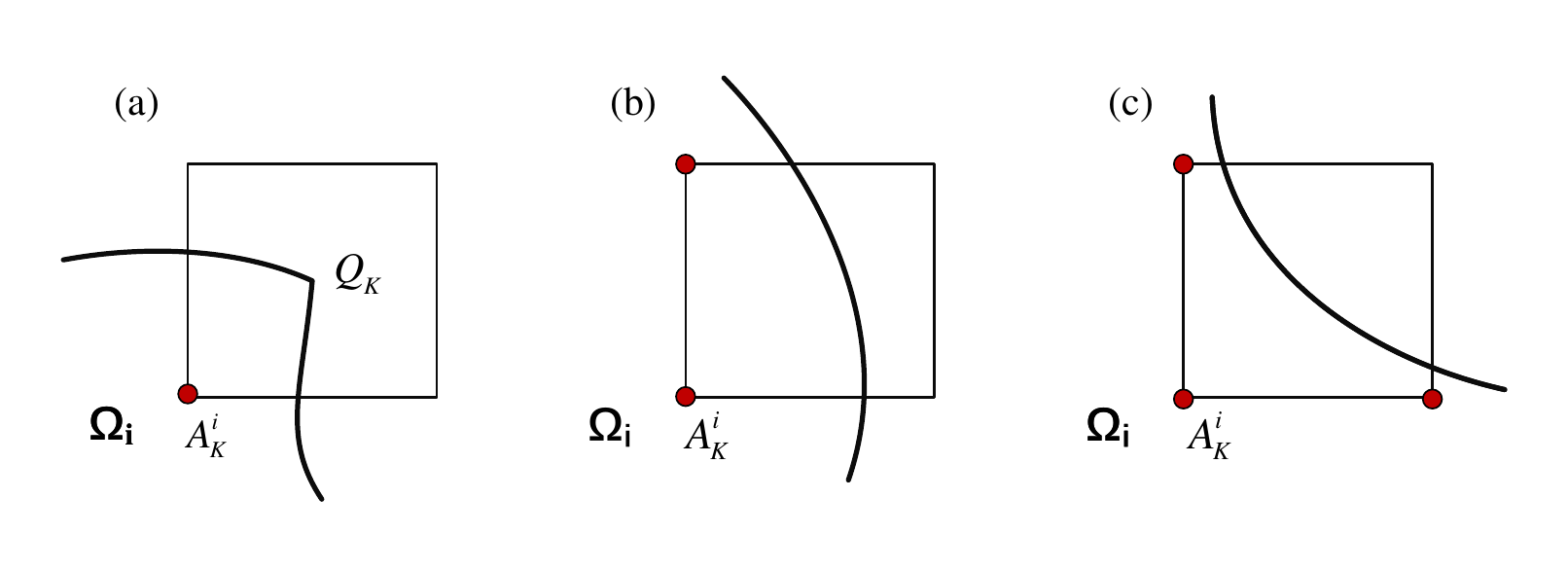}
	\caption{Examples of a large element $K$ with respect to $\Om_i$ with (a) one, (b) two, and (c) three vertices in $\Om_i$. The element in (a) is an irregular large element with respect to $\Om_i$.}
	\label{fig:2.1}
\end{figure}

\begin{definition}\label{def:2.1} (Large element)
For $i=1,2$, an element $K\in\cT$ is called a large element with respect to $\Om_i$ if $K\subset\Om_i$
or $K\in\cT^\Ga$ for which there exists a constant $\de_0\in (0,1/2)$ such that $|e\cap\Om_i|\ge\de_0|e|$ for each side $e$ of $K$ having nonempty intersection with $\Om_i$ and, if $K$ has only one vertex $A_K^i$ in $\Om_i$ and includes a singular point $Q_K$ of $\Ga$, ${\rm dist}(Q_K,e_j)\ge\frac 12\de_0\min(|e_1|, |e_2|)$, where $e_j$ is the side of $K$ having $A_K^i$ as one of its end points and ${\rm dist}(Q_K,e_j)$ is the distance of $Q_K$ to the side $e_j$, $j=1,2$, see Figure \ref{fig:2.1}. 
\end{definition}

The large elements with respect to $\Om_i$ which have only one vertex in $\Om_i$ and include a singular point of $\Ga$ will be called {\it irregular} large elements with respect to $\Om_i$. The other kinds of large elements with respect to $\Om_i$ will be called {\it regular} large elements with respect to $\Om_i$, $i=1,2$. We notice that if $K$ is an irregular large element, then the triangle with vertices $A_K^i$, $Q_K$, and one of the intersection points of $\Ga\cap\pa K$ is shape regular with the ratio of the radius of the maximal inscribed circle to the diameter of the triangle depending on $\de_0$.

One difficulty in the study of unfitted finite element methods is the possibility that $K$ may not be large with respect to both $\Om_1$ and $\Om_2$. We make the following assumption on the finite element mesh which is inspired by Johansson and Larson \cite{Johansson} in which a fictitious boundary discontinuous Galerkin method for elliptic equations is developed.

\begin{figure}[t]
	\centering
	\includegraphics[width=0.95\textwidth]{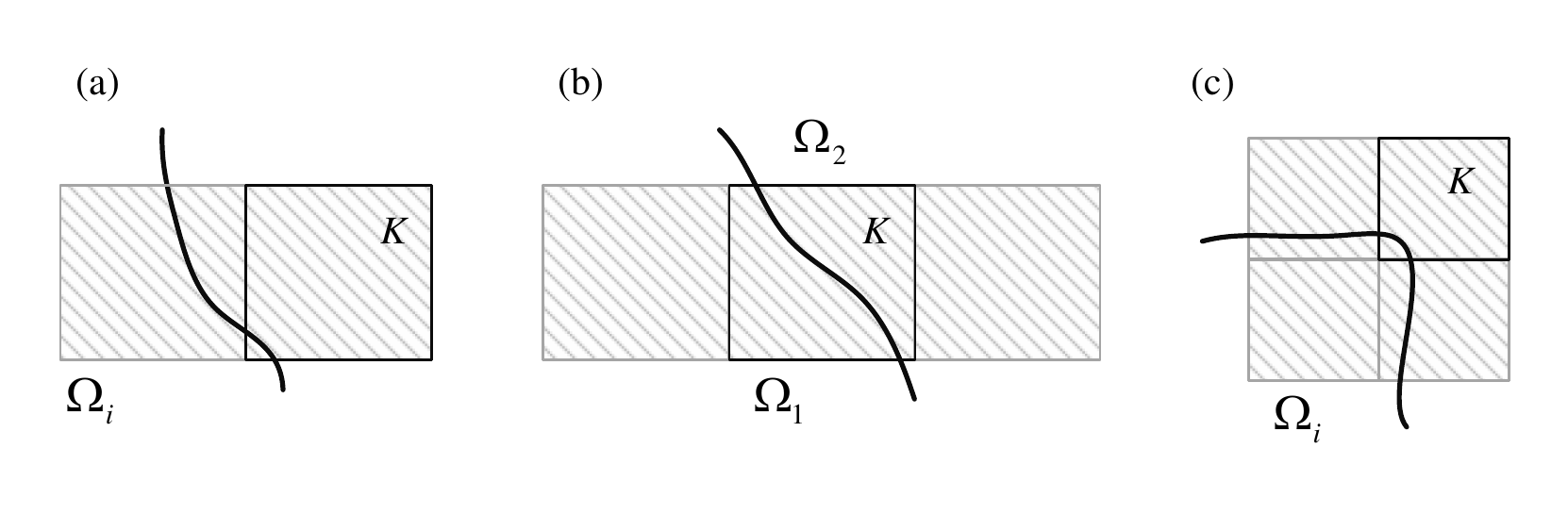}
	\caption{The small element $K$ and its macro-element $N(K)$ (shadow region).}
	\label{fig:2.2}
\end{figure}

\noindent\\
{\bf Assumption (H1)}: For each $K\in\cT^\Ga$, there exists a rectangular macro-element $N(K)$ which is a union of $K$ and its neighboring element (or elements) such that $N(K)$ is large with respect to both $\Om_1$ and $\Om_2$, see Figure \ref{fig:2.2}. We assume $h_{N(K)}\le C_0h_K$ for some fixed constant $C_0$. \\

One way to satisfy the assumption (H1) is to locally refine the neighboring elements $K'$ of $K\in\cT^\Ga$ which is not large with respect to both $\Om_1,\Om_2$ so that the elements $K'$ are of the same size as $K$ and $K'$
are completely included in $\Om_1$ or $\Om_2$. In this case, we can define $N(K)$ as the union of $K$
and those neighboring elements $K'$ (see Figure \ref{fig:2.2}).

In the following, we will always set $N(K)=K$ if $K\in\cT^\Ga$ and $K$ is large with respect to both $\Om_1,\Om_2$. Thus $\widetilde\cT=\{N(K):K\in\cT^\Ga\}\cup\{K\in\cT: K\subset\Om_i, i=1,2, K\not\subset N(K') \mbox{ for some $K'\in\cT^\Ga$}\}$ is also a Cartesian mesh of $\Om$. The
elements in $\widetilde\cT$ are large with respect to both domains $\Om_1,\Om_2$ and the interface intersects
the boundary of each element $K\in\widetilde\cT$ also twice at different sides. We will call $\widetilde\cT$ the induced mesh of $\cT$ and write $\widetilde\cT={\rm Induced }\,(\cT)$.

For any rectangular element $K$, $K\cap\Ga\not=\emptyset$, we denote $\Ga_K=\Ga\cap K$ and $\Ga_K^h$ the (open) straight segment
connecting the two intersection points of $\Ga$ and $\pa K$. If $K$ includes a singular point $Q_K$,
then $\Ga_K$ is the union of two $C^2$-smooth curves $\Ga_{1K}\cup\Ga_{2K}$. We denote $\Ga_{jK}^h$ the (open) straight segment
connecting $Q_K$ and the intersecting point of $\Ga_{jK}\cap\pa K$, $j=1,2$. 

The concept of interface deviation which measures how far $\Ga_K$ deviates from $\Ga_K^h$ or $\Ga_{1K}^h, \Ga_{2K}^h$ plays an important role in our subsequent analysis.

\begin{definition}\label{def:2.2}
For any rectangular element $K$, $K\cap\Ga\not=\emptyset$, the interface deviation $\eta_K$ is defined as $\eta_K=\max(\eta_K^1,\eta_K^2)$, where for $i=1,2$, if $K$ is a regular large element with respect to $\Om_i$ with $A^i_K\in\Om_i$ being the vertex of $K$ which has the maximum distance to $\Ga_K^h$, 
\ben
\eta_K^i=\frac{\dist_{\rm H}(\Ga_K,\Ga_K^h)}{\dist(A^i_K,\Ga_K^h)}\,,
\een
and if $K$ is an irregular large element with respect to $\Om_i$ with vertex $A_K^i\in\Om_i$, 
\ben
\eta_K^i=\max\left(\frac{\dist_{\rm H}(\Ga_{1K},\Ga_{1K}^h)}{\dist(A^i_K,\Ga_{1K}^h)},\frac{\dist_{\rm H}(\Ga_{2K},\Ga_{2K}^h)}{\dist(A^i_K,\Ga_{2K}^h)}\right).
\een
Here $\dist_{\rm H}(\Ga_1,\Ga_2)=\max_{x\in \Ga_1}(\min_{y\in \Ga_2}|x-y|)$ is the Hausdorff distance between two sets $\Ga_1,\Ga_2$ and $\dist(A,\Ga_1)$ is the distance of a point $A$ to the set $\Ga_1$.
\end{definition}

\begin{lemma}\label{lem:new}
Let $K\in\widetilde\cT^\Ga$ which is large with respect to both $\Om_1$, $\Om_2$ and $N(K)$ be the macro-element which is the union of $K$ and its two or three neighboring elements included in $\Om_i$ depending on $K$ having two or three vertices in $\Om_i$, $i=1,2$, see Figure \ref{fig:new}. The neighboring elements are assumed to be of the same size as $K$. Then $\eta_{N(K)}^i\le \max(1/2,(1-\de_0)/(1+\de_0))$.
\end{lemma}

\begin{proof} We first prove the case when $K$ has three vertices in $\Om_i$. Let $\Ga_K^h$ be the segment $CD$, $A_K^i\in\Om_i$ be the vertex of $K$ having the maximal distance to $CD$, and $A_{N(K)}^i\in\Om_i$ be the vertex of $N(K)$ having the maximal distance to $CD$. We 
extend $DC$ to intersect the extended segment $A_K^iB_K$ at $B_K'$ and $A^i_{N(K)}B_{N(K)}$ at $B'_{N(K)}$, see 
Figure \ref{fig:new}(b). Denote $h_j$ the length of the side of $K$ parallel to the $j$th coordinate axis, $j=1,2$. By elementary 
geometry, $\frac{|B_{N(K)}B'_{N(K)}|}{|B_KB'_K|}=\frac{|B_KC|+h_2}{|B_KC|}\ge 2$.
Thus, since $\dist_{\rm H}(\Ga_K,\Ga_K^h)\le\dist(A_K^i,\Ga_K^h)$, 
\ben
\eta^i_{N(K)}\le\frac{|B_KB_K'|+h_1}{|B_{N(K)}B_{N(K)}'|+2h_1}\le\frac{|B_KB_K'|+h_1}{2|B_KB_K'|+2h_1}=\frac 12.
\een
When $K$ has two vertices in $\Om_i$, we use the notation in Figure \ref{fig:new}(a). Since $K$ is large with respect to both $\Om_1$, $\Om_2$, we have $\de_0h_1\le |A_K^iC|\le (1-\de_0)h_1, \de_0h_1\le |B_K^iD|\le (1-\de_0)h_1$. Thus it follows from $\dist_{\rm H}(\Ga_K,\Ga_K^h)\le\max(\dist(A_K^i,\Ga_K^h),\dist(B_K^i,\Ga_K^h))$ that
\ben
\eta^i_{N(K)}\le\max\left(\frac{|A_K^iC|}{|A_{N(K)}^iC|},\frac{|B_K^iD|}{|A_{N(K)}^iC|}\right)\le\frac{(1-\de_0)h_1}{h_1+\de_0h_1}=\frac{1-\de_0}{1+\de_0}.
\een
This completes the proof. $\Box$
\end{proof}

\begin{figure}[t]
	\centering
	\includegraphics[width=0.95\textwidth]{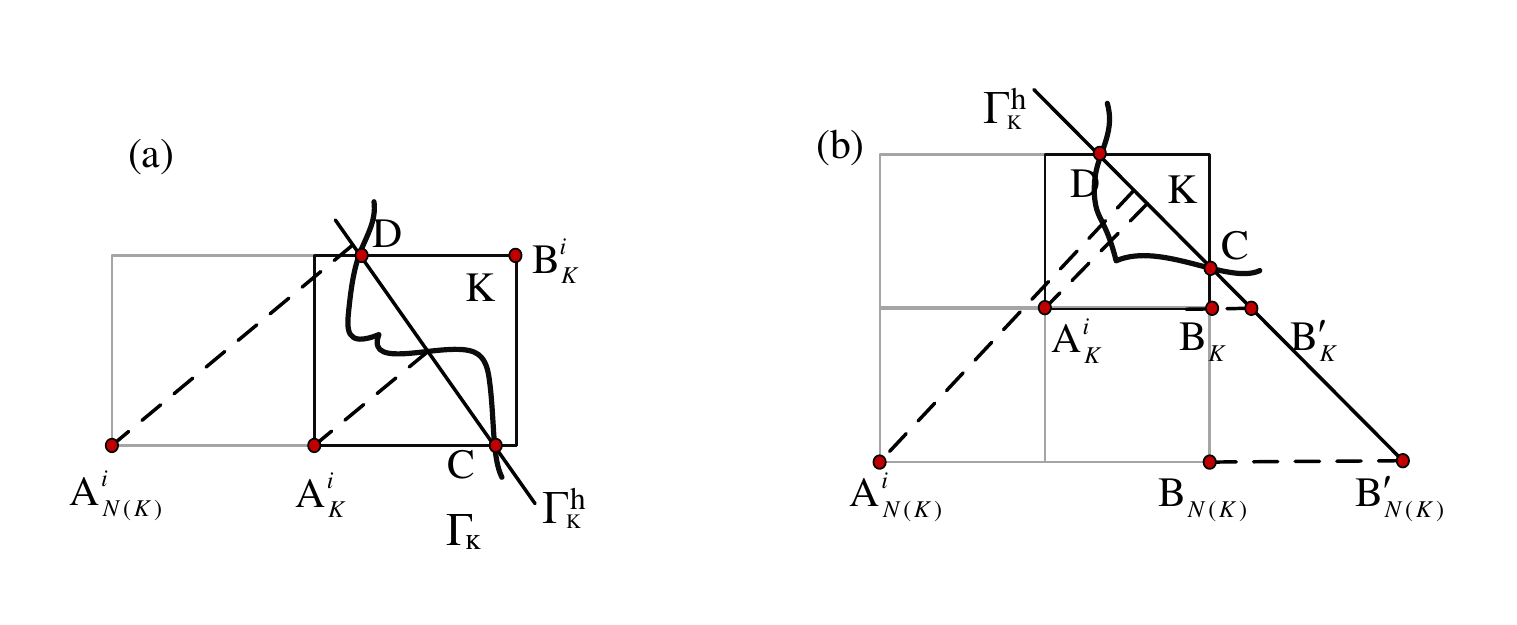}
	\caption{The element $K$ and its macro-element $N(K)$ when $K$ includes a singular point of $\Ga$. (a) $K$ has two vertices in $\Om_i$. (b) $K$ has three vertices in $\Om_i$.}
	\label{fig:new}
\end{figure}

We make the following assumption which can be viewed as a variant of interface resolving mesh conditions.

\noindent\\
{\bf Assumption (H2)}: For any $K\in\widetilde\cT^\Ga$, there exists a rectangular macro-element $N(K)$ which is a union of $K$ and its neighboring element (or elements) such that $\eta_{N(K)}\le \max(1/2,(1-\de_0)/(1+\de_0))$. \\

If $\Ga_K$ is $C^2$-smooth in $K$, it is easy to see that $\dist_{\rm H}(\Ga_K,\Ga_K^h)\le Ch_K^2$ (see, e.g., Feistauer \cite[\S 3.3.2]{Feistauer}) and thus $\eta_K\le Ch_K$ for some constant $C$ independent of $h_K$. When $K$ is an irregular large element with respect to $\Om_i$, we still have $\dist_{\rm H}(\Ga_{jK},\Ga_{jK}^h)\le Ch_K^2$, $j=1,2$, and thus $\eta_K^i\le Ch_K$. Therefore, in these cases, Assumption (H2) can be satisfied with $N(K)=K$ if $h_K$ is sufficiently small. When $K\in\widetilde\cT^\Ga$ includes a singular point of $\Ga$ and has two or three vertices in $\Om_i$, by Lemma \ref{lem:new},  if $h_K$ is sufficiently small, we may merge $K$ with its neighboring elements in $\Om_i$ to obtain a macro-element $N(K)$ so that $\eta_{N(K)}^i\le \max(1/2,(1-\de_0)/(1+\de_0))$. Therefore, when the interface elements are sufficiently refined, Assumption (H2) can always be satisfied.  

In the following, we denote $\cM$ the induced mesh from $\widetilde\cT$ by possibly merging elements in $\widetilde\cT^\Ga$ with their neighboring elements such that 
\be\label{HH}
\eta_K\le \max(1/2,(1-\de_0)/(1+\de_0))\ \ \forall K\in\cM.
\ee
Obviously, each element in $\cM$ is large with respect to both $\Om_1,\Om_2$. 

Now we introduce the finite element space using the idea of ``doubling of unknowns" in Hansbo and Hansbo \cite{Hansbo}. For any integer $p\ge 1$ and $K\in\cM$, denote $Q_p(K)$ the set of all polynomials in $K$ which is of order $p$ in each variable. We define the unfitted finite element space as
\ben
\mathbb{X}_p(\cM)=\{v_1\chi_{\Om_1}+v_2\chi_{\Om_2}:v_i|_K\in Q_p(K), i=1,2\}.
\een
We also define the broken Sobolev space
\ben
H^1(\cM)=\{v_1\chi_{\Om_1}+v_2\chi_{\Om_2}:v_i|_K\in H^1(K),i=1,2\}.
\een
For any $v\in H^1(\cM)$, $v|_K=v_1\chi_{K_1}+v_2\chi_{K_2}\ \forall K\in\cM$, we denote $\na_hv|_K:=\na v_1\chi_{K_1}+\na v_2\chi_{K_2}$, where $K_i=K\cap\Om_i$, $\chi_{K_i}$ is the characteristic function of $K_i$, $i=1,2$.

Let $\cE=\cE^{\rm side}\cup\cE^\Ga\cup\cE^{\rm bdy}$, where $\cE^{\rm side}=\{e=\pa K\cap\pa K':K,K'\in\cM\}$, $\cE^\Ga=\{\Ga_K:K\in\cM\}$ and $\cE^{\rm bdy}=\{e=\pa K\cap\pa\Om:K\in\cM\}$. Since hanging nodes are allowed, $e\in\cE^{\rm side}$ can be part of a side of an adjacent element. For $i=1,2$, denote by $\cM_i=\{K\in\cM:K\cap\Om_i\not=\emptyset\}$. Then $\Om_i\subset\Om_i^h=\cup\{K:K\in\cM_i\}$. We denote $\cE_i^{\rm side}$ the set of all sides of $\cM_i$ interior to $\Om_i^h$, that is, not on the boundary $\pa\Om_i^h$. Finally, we set $\bar{\cal{E}}= \cE^{\rm side}_{1} \cup\cE^{\rm side}_{2}\cup\cE^\Ga\cup\cE^{\rm bdy}$.

For any subset $\hat\cM\subset\cM$ and $\hat\cE\subset\bar\cE$, we use the notation
\ben
(u,v)_{\hat\cM}:=\sum_{K\in\hat\cM}(u,v)_K,\ \ \la u,v\ra_{\hat\cE}:=\sum_{e\subset\hat\cE}\la u,v\ra_e,
\een
where $(u,v)_K$ is the inner product of $L^2(K)$ and $\la u,v\ra_e$ is the inner product of $L^2(e)$.

For any $e\in\cE$, we fix a unit normal vector $n_e$ of $e$ with the convention that $n_e$ is the unit outer normal to $\pa\Om$ if $e\in\cE^{\rm bdy}$ and $n_e$ is the unit outer normal to $\pa\Om_1$ if $e\in\cE^\Ga$. For any $v\in H^1(\cM)$, we define the jump of $v$ across $e$ as 
\ben
\lj v\rj_e:=v_--v_+\ \ \forall e\in\cE^{\rm side}\cup\cE^\Ga,\ \ \ \
\lj v\rj_e:=v_-\ \ \forall e\in\cE^{\rm bdy},
\een
where $v_\pm$ is the trace of $v$ on $e$ in the $\pm n_e$ direction. We define the piecewise constant normal vector function $n\in L^\infty(\cE)=\Pi_{e\in\cE}L^\infty(e)$ by
$n|_e=n_e\ \ \forall e\in\cE$.

Now we introduce our unfitted finite element method in the framework of LDG method. We focus on the primal formulation by following Arnold, Brezzi, Cockburn and Marini \cite{Arnold}, Perugia and Sch\"otzau \cite{Perugia}. For any $v\in H^1(\cM),g\in L^2(\pa\Om)$, we define the liftings $\mathsf{L}(v)\in [\X_p(\cM)]^2$, $\mathsf{L}_1(g)\in [\X_p(\cM)]^2$ such that for any $r\in [\X_p(\cM)]^2$,
\be
(r,\mathsf{L}(v))_\cM=\la\hat{r}\cdot n,\lj v\rj\ra_{\cE},\ \ \ \
(r,\mathsf{L}_1(g))_\cM=\la r\cdot n,g\ra_{\cE^{\rm bdy}},\label{ll1}
\ee
where the numerical flux $\hat r|_e=\beta_er_-+(1-\beta_e)r_+\ \ \forall e\in\cE$. Here $\beta_e=0$ or $\beta_e=1$ for $e\in\cE^{\rm side}\cup\cE^\Ga$ and $\beta_e=1$ for $e\in\cE^{\rm bdy}$ as suggested in \cite{Cockburn} to enhance the sparsity of the stiffness matrix.

Our unfitted finite element method is to find $U\in\X_p(\cM)$ such that
\be\label{a2}
a_h(U,v)=F_h(v)\ \ \ \ \forall v\in\X_p(\cM),
\ee
where the bilinear form $a_h: H^1(\cM)\times H^1(\cM)\to \R$ and the functional $F_h:H^1(\cM)\to\R$ are
given by
\ben
& &a_h(v,w)=(a(\na_h v-\mathsf{L}(v)),\na_h w-\mathsf{L}(w))_\cM+\la\al\lj v\rj,\lj w\rj\ra_{\bar{\cal{E}}},\\
& &F_h(v)=(f,v)_\cM-(a\mathsf{L}_1(g),\na_h v-\mathsf{L}(v))_\cM+\la\al g,v\ra_{\cE^{\rm bdy}}.
\een
Here for any $v=v_1\chi_{\Om_1}+v_2\chi_{\Om_2}, w=w_1\chi_{\Om_1}+w_2\chi_{\Om_2}\in H^1(\cM)$,
\be\label{xx}
\la\al\lj v\rj,\lj w\rj\ra_{\bar{\cal{E}}}:=\sum^2_{i=1}\la\al\lj v_i\rj,\lj w_i\rj\ra_{\cE_i^{\rm side}}+\la\al \lj v\rj,\lj w\rj\ra_{\cE^\Ga\cup\cE^{\rm bdy}}.
\ee
We notice that the penalty is added on $\bar{\cal{E}}= \cE^{\rm side}_{1} \cup\cE^{\rm side}_{2}\cup\cE^\Ga\cup\cE^{\rm bdy}$ instead of $\cE=\cE^{\rm side}\cup\cE^\Ga\cup\cE^{\rm bdy}$. The interface penalty function $\al\in L^\infty(\cE)$ will be specified in \S2.3 after we prove the inverse trace inequality on the curved domain in the next subsection. We remark that the stabilization term $\la\al\lj v\rj,\lj w\rj\ra_{\cE^\Ga}$ plays the key role in weakly capturing the jump behavior of the finite element solution at the interface in the weak formulation \eqref{a2}.

To conclude this section, we remark that the unfitted finite element methods in the literature are mostly based on the interior penalty discontinuous Galerkin (IPDG) method. The LDG formulation allows us to prove the stability of the method without assuming the interface penalty constant $\al_0$ being sufficiently large (see \S 2.3 below).

\subsection{Domain inverse estimates}

Let $I=(-1,1)$ and $\{L_n\}_{n\ge 0}$ be the Legendre polynomials which are orthogonal in $L^2(I)$ and satisfy $L_n(1)=1$, $n\ge 0$. We start by recalling the first integral of Laplace for the Legendre polynomials (see, e.g., Szeg\"o \cite[P.97]{Szego}).

\begin{lemma}\label{lem:2.1}
For $n\ge 0$, we have
\ben
L_n(t)=\frac 1\pi\int^\pi_0\left[t+(t^2-1)^{1/2}\cos\phi\right]^nd\phi\ \ \ \ \forall t\in\R.
\een
\end{lemma}

We remark that the integral on the right hand side of above identity is actually real if $|t|<1$ since $\int^\pi_0(\cos\phi)^{2k+1} d\phi=0$ for 
any integer $k\ge 0$.

\begin{proof} For the sake of completeness, we sketch the proof here. By Rodrigues' formula (cf., e.g., Bernardi and Maday \cite{Bernardi}), we know that
\ben
L_n(t)=\frac{(-1)^n}{2^n n!}\left(\frac d{dt}\right)^n\left[(1-t^2)^n\right]\ \ \ \ \forall t\in\R.
\een
By Cauchy's integration formula,
\ben
L_n(t)=\frac 1{2\pi\i}\int_{\Sigma}\frac{L_n(z)}{z-t}dz=\frac 1{2\pi\i}\frac{(-1)^n}{2^n n!}\int_{\Sigma}
\left(\frac d{dz}\right)^n\left[(1-z^2)^n\right]\frac 1{z-t}dz
\een
for any closed contour enclosing the point $z=t$. Integrating by parts we obtain
\ben
L_n(t)=\frac 1{2\pi\i}\int_{\Sigma}\left(\frac 12\frac{z^2-1}{z-t}\right)^n\frac{dz}{z-t}.
\een
The lemma is obvious if $t=\pm 1$. For $t\not=\pm 1$, we choose the circle $|z-t|=|t^2-1|^{1/2}$ as the
contour of the integration. By writing $z=t+(t^2-1)^{1/2}e^{\i\phi}$, we obtain easily the formula of Laplace.
$\Box$
\end{proof}

It follows from Lemma \ref{lem:2.1} that $|L_n(t)|\le 1$ $\forall t\in [-1,1]$, and
\be\label{a3}
|L_n(t)|\le \left(|t|+\sqrt{t^2-1}\right)^n\ \ \ \ \forall |t|>1,\ \ n\ge 0.
\ee

We now prove the one dimensional domain inverse estimate.

\begin{lemma}\label{lem:2.2}
Let $I_\lam=(-\lam,\lam), \lam>1$, we have
\ben
\|g\|_{L^2(I_\lam\backslash\bar I)}^2\le\frac 12\left[(\lam+\sqrt{\lam^2-1}\,)^{2p+1}-1\right]\|g\|_{L^2(I)}^2\ \ \ \ \forall g\in Q_p(I_\lam),
\een
where $Q_p(I_\lam)$ is the set of polynomials of order $p$ in $I_\lam$.
\end{lemma}

\begin{proof} It is well known that $\|L_n\|_{L^2(I)}=(n+1/2)^{-1/2}$ for $n\ge 0$. Thus, for any $g\in Q_p(I_\lam)$, $g(t)=\sum^p_{n=0}a_nL_n(t)$ and $\|g\|^2_{L^2(I)}=\sum^p_{n=0}a_n^2(n+1/2)^{-1}$. Now by Cauchy-Schwarz's inequality
\be\label{a4}
\|g\|^2_{L^2(I_\lam\backslash\bar I)}\le\|g\|^2_{L^2(I)}\cdot\sum^p_{n=0}(n+1/2)\|L_n\|^2_{L^2(I_\lam\backslash\bar I)}.
\ee
By using \eqref{a3} and taking the transform $s=t+\sqrt{t^2-1}$,
\ben
\sum^p_{n=0}(n+1/2)\|L_n\|^2_{L^2(I_\lam\backslash\bar I)}&\le&2\sum^p_{n=0}(n+1/2)\int^\lam_1(t+\sqrt{t^2-1}\,)^{2n}dt\\
&=&\sum^p_{n=0}(n+1/2)\int^{\lam+\sqrt{\lam^2-1}}_1(s^{2n}-s^{2n-2})ds.
\een
Now by using the summation by parts, we have
\ben
\sum^p_{n=0}(n+1/2)\|L_n\|^2_{L^2(I_\lam\backslash\bar I)}&\le&(p+1/2)\int^{\lam+\sqrt{\lam^2-1}}_1 s^{2p}ds\\
&=&\frac 12\left[(\lam+\sqrt{\lam^2-1}\,)^{2p+1}-1\right].
\een
This completes the proof by using \eqref{a4}. $\Box$
\end{proof}


It follows from Lemma \ref{lem:2.2} that for any $(a,b)\subset (a,c)$, we have
\be\label{ll}
\int^c_b|g|^2dt\le\frac12\left[(\lam+\sqrt{\lam^2-1})^{2p+1}-1\right]\int^b_a|g|^2dt\ \ \ \ \forall g\in  Q_p(a,c),
\ee
where $\lam=(c-t_0)/(b-t_0)$, $t_0=(a+b)/2$ is the midpoint of the interval $(a,b)$.

The following two dimensional domain inverse estimate plays a key role in the next subsection to study the stability of our unfitted finite element method.

\begin{lemma}\label{lem:2.3}
Let $\De$ be a triangle with vertices $A=(a_1,a_2)^T, B=(0,0)^T, C=(c_1,0)^T$, where $a_2,c_1>0$. Let $\de\in (0,a_2)$ and $\De_\de=\{x\in\De:\dist(x,BC)>\de\}$, where $\dist(x,BC)=\min\{|x-y|:y\in BC\}$. Then, we have
\ben
\|v\|_{L^2(\De)}\le\mathsf{T}\left(\frac{1+\delta a_2^{-1}}{1-\de a_2^{-1}}\right)^{2p+3/2}\|v\|_{L^2(\Delta_\de)}\ \ \ \ \forall v\in Q_p(\De).
\een
where $\mathsf{T}(t)=t+\sqrt{t^2-1}\ \ \forall t\ge 1$.
\end{lemma}

\begin{proof} The triangle $\De$ can be parametrized as $x=t(s,0)^T+(1-t)(a_1,a_2)^T$, $s\in(0,c_1),t\in (0,1)$. The Jacobi determinant of the parametrization is $a_2t$. Obviously,
\ben
\int_{\De_\de}|v|^2dx=\int^{c_1}_0\int^{1-\de a_2^{-1}}_0|v(ts+(1-t)a_1,(1-t)a_2)|^2a_2tdtds.
\een
Since for a fixed $s$, $\tilde v(t)=v(ts+(1-t)a_1,(1-t)a_2)t\in Q_{2p+1}(0,1)$, we use \eqref{ll} to obtain
\ben
& &\int^1_{1-\de a_2^{-1}}|v(ts+(1-t)a_1,(1-t)a_2)|^2tdt\\
&\le&\frac 1{1-\de a_2^{-1}}\int^1_{1-\de a_2^{-1}}|tv(ts+(1-t)a_1,(1-t)a_2)|^2dt\\
&\le&\frac 12\left[\mathsf{T}\left(\frac{1+\de a_2^{-1}}{1-\de a_2^{-1}}\right)^{2(2p+1)+1}-1\right]\int^{1-\de a_2^{-1}}_0|v(ts+(1-t)a_1,(1-t)a_2)|^2tdt.
\een
This completes the proof. $\Box$
\end{proof}

The following lemma will be used in section 4 to prove the efficiency of the a posteriori error estimators.

\begin{lemma}\label{lem:2.4}
Let $\De\subset\R^2$ be a triangle and $\rho_\De$ the radius of its maximal inscribed circle. For any $\de\in (0,\rho_\De/2)$, denote $\De_\de=\{x\in \De:\dist(x,\pa \De)>\de\}$. Then for any $v\in Q_p(\De)$, we have
\ben
\|v\|_{L^2(\De)}\le (1+7\sqrt{\de/\rho_\De})^{2p+3/2}\|v\|_{L^2(\De_\de)}.
\een
\end{lemma}

\begin{proof} Let $O$ be the center of the maximal inscribed circle of $\De$. The triangle $\De$ is divided into three sub-triangles by connecting $O$ and three vertices of $\De$. We use Lemma \ref{lem:2.3} in each of the three triangles to obtain
\ben
\|v\|_{L^2(\De)}\le \mathsf{T}(\lam)^{2p+3/2}\|v\|_{L^2(\De_\de)},\ \ \lam=\frac{1+\de/\rho_\De}{1-\de/\rho_\De}.
\een
Since $\mathsf{T}(\lam)= 1+\sqrt{\lam-1}(\sqrt{\lam-1}+\sqrt{\lam+1})$ and $\lam<3$ by the assumption $\de\in (0,\rho_\De/2)$, we have
\ben
\|v\|_{L^2(\De)}\le (1+2(2+\sqrt 2)\sqrt{\de/\rho_\De})^{2p+3/2}\|v\|_{L^2(\De_\de)}.
\een
This completes the proof. $\Box$
\end{proof}

\subsection{Stability and a priori error analysis}

We first recall the standard multiplicative trace inequality (cf., e.g., Burman and Ern \cite{Burman07}), for any $K\in\cM$ and $v\in H^1(K)$,
\be\label{m1}
\|v\|_{L^2(\pa K)}\le Ch_K^{-1/2}\|v\|_{L^2(K)}+C\|v\|_{L^2(K)}^{1/2}\|\na v\|_{L^2(K)}^{1/2}.
\ee

The following lemma is proved in Xiao, Xu and Wang \cite{Wang} when the interface $\Ga$ is $C^2$-smooth. It can be extended to cover the case when $\Ga$ is Lipschitz and piecewise $C^2$ as assumed in this paper.

\begin{lemma}\label{lem:2.5}
For any $K\in\cM$, denote $K_i=K\cap\Om_i$, $i=1,2$. Then there exists a constant $C$ independent of $h_K$ such that for $i=1,2$,
\ben
\|v\|_{L^2(\Ga_K)}\le C\|v\|_{L^2(K_i)}^{1/2}\|v\|_{H^1(K_i)}^{1/2}+\|v\|_{L^2(\pa K_i\backslash\bar\Ga_K)}\ \ \ \ \forall v\in H^1(K_i).
\een
\end{lemma}

\begin{proof} Since $\Ga$ is Lipschitz continuous and piecewise $C^2$, there is a set of subdomains $\{U_j\}^r_{j=1}$ that covers $\Ga$ and a partition of unity $\{\phi_j\}^r_{j=1}$ subordinated to $\{U_j\}^r_{j=1}$, that is, $\phi_j\in C^\infty_0(U_j), 0\le\phi_j\le 1, \sum^r_{j=1}\phi_j=1$ in $\cup^r_{j=1}U_j$. Moreover, let $\nu=(\nu_1,\nu_2)^T$ be the unit outer normal vector to $\pa\Om_1$, we may assume in each $U_j$, there exists an index $k(j)=1$ or $2$, such that $|\nu_{k(j)}|\ge 1/2$ in $U_j$, $j=1,\cdots,r$. Here for the points on $\Ga$ where $\nu$ is discontinuous, we define $\nu=(1/\sqrt 2,1/\sqrt 2)^T$. Since $\nu_{k(j)}$ does not change sign in each $U_j$, we have
\ben
\frac 12\int_{\Ga_K}|v|^2ds=\frac 12\sum^r_{j=1}\int_{\Ga_K}|v|^2\phi_jds&\le&\sum^r_{j=1}\int_{\Ga_K}|v|^2\phi_j|\nu_{k(j)}|ds\\
&\le&\sum^r_{j=1}\left|\int_{\Ga_K}|v|^2\phi_j \nu_{k(j)}ds\right|.
\een
Now by integration by parts, we obtain
\ben
\int_{\Ga_K}|v|^2\phi_j \nu_{k(j)}ds&=&\int_{\pa K_i}|v|^2\phi_j \nu_{k(j)}ds-\int_{\pa K_i\backslash\bar\Ga_K}|v|^2\phi_j
\nu_{k(j)}ds\\
&=&\int_{K_i}\frac{\pa}{\pa x_{k(j)}}\left[\phi_j|v|^2\right]dx-\int_{\pa K_i\backslash\bar\Ga_K}|v|^2\phi_j
\nu_{k(j)}ds\\
&\le&C\|v\|_{L^2(K_i)}^2+2\|v\|_{L^2(K_i)}\|\na v\|_{L^2(K_i)}+\|v\|_{L^2(\pa K_i\backslash\bar\Ga_K)}^2,
\een
where $C=\max_{1\le j\le r}\|\na\phi_j\|_{L^\infty(U_j)}$. This completes the proof. $\Box$
\end{proof}

We will use the following inverse trace inequality in Warburton and Hesthaven \cite{Warburton}.

\begin{lemma}\label{lem:2.6}
Let $\De$ be a triangle. For any $v\in P_p(\De)$, the set of all polynomials of order $p$ in $\De$, we have
\ben
\|v\|_{L^2(\pa \De)}\le\sqrt{\frac{(p+1)(p+2)}2\frac{|\pa \De|}{|\De|}}\,\|v\|_{L^2(\De)}.
\een
\end{lemma}

\begin{figure}[t]
	\centering
	\includegraphics[width=0.35\textwidth]{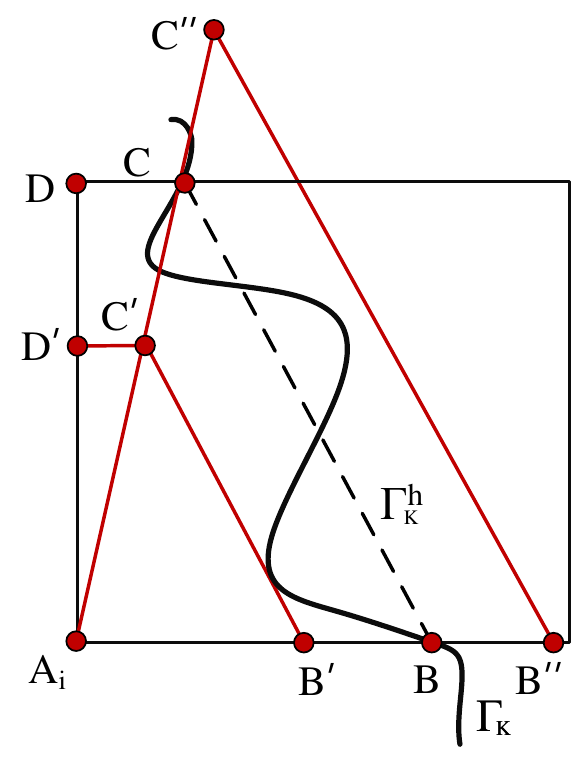}
	\caption{The figure used in the proof of Lemma \ref{lem:2.7} and Lemma \ref{lem:4.1}.}
	\label{fig:2.3}
\end{figure}

The following inverse trace inequality on curved domains plays a key role in our analysis.

\begin{lemma}\label{lem:2.7}
Let $K\in\cM^\Ga:=\{K\in\cM:K\cap\Ga\not=\emptyset\}$. Then for $i=1,2$,
\ben
\|v\|_{L^2(\pa K_i)}\le Cph_K^{-1/2}\mathsf{T}\left(\frac{1+3\eta_K}{1-\eta_K}\right)^{2p}\|v\|_{L^2(K_i)}\
 \ \ \ \forall v\in Q_p(K),
\een
where the constant $C$ is independent of $h_K,p$, and $\eta_K$.
\end{lemma}

\begin{proof} We only prove the case when $K_i=K\cap\Om_i$ is a curved trapezoid (see Figure \ref{fig:2.3}). The other cases can be proved similarly. Let $K_i^h$ be the trapezoid $A_iBCD$ which replaces $\Ga_K$ by the straight segment $\Ga_K^h$, where 
$A_i$ is the vertex of $K$ in $\Om_i$ having the maximum distance to $\Ga_K^h$, $B,C$ are the end points of $\Ga^h_K$ with $C$ on the side of $K$ opposite to $A_i$, and $D$ the other vertex of $K$ in $\Om_i$ (see Figure \ref{fig:2.3}). As $K$ is large with respect to $\Om_i$, the triangles $\De A_iBC, \De A_iCD$ are shape regular with the shape regular constant depending possibly on $\de_0$ in Definition \ref{def:2.1}. By Lemma \ref{lem:2.5} and using Lemma \ref{lem:2.6} in each triangle $\De A_iBC,\De A_iCD$ we obtain
\be\label{b1}
\|v\|_{L^2(\pa K_i)}&\le&C\|v\|_{L^2(K_i)}^{1/2}\|v\|_{H^1(K_i)}^{1/2}+\|v\|_{L^2(\pa K_i^h)}\nonumber\\
&\le&C\|v\|_{L^2(K_i)}^{1/2}\|v\|_{H^1(K_i)}^{1/2}+Cph_K^{-1/2}\|v\|_{L^2(K_i^h)}.
\ee
Let $\de=\dist(\Ga_K,\Ga_K^h)$ and $d_i=\dist(A_i,\Ga_K^h)$. Then the interface deviation $\eta_K\ge\de/d_i$ by Definition \ref{def:2.2}. Let $\De A_iB'C'\subset\De ABC\subset\De A_iB''C''$ such that $B'C', B''C''$ are parallel to $\Ga_K^h$ and the distances of $B'C',B''C''$ to $\Ga_K^h$ are $\de$. $B',C'$ are respectively on the segments $A_iB,A_iC$ and $B'',C''$ are respectively on the extended lines of $A_iB, A_iC$. Let $D'$ on $AD$ such that $D'C'$ is parallel to $DC$, see Figure \ref{fig:2.3}. It is clear that $\De A_iC'D'\subset K_i$ and
$\frac{|C'D'|}{|CD|}=\frac{|A_iC'|}{|A_iC|}=\frac{d_i-\de}{d_i}$. Thus $\frac{|DD'|}{|A_iD|}=\frac\de{d_i}\le\eta_K$.

Since $K^h_i=(\De A_iCD)\cup(\De A_iBC)$ and $\De A_iC'D', \De A_iB'C'\subset K_i$, we obtain by using Lemma \ref{lem:2.3} that
\be\label{b2}
||v\|_{L^2(K_i^h)}&\le&\|v\|_{L^2(\De A_iCD)}+\|v\|_{L^2(\De A_iBC)}\nn\\
&\le&\mathsf{T}\left(\frac{1+\eta_K}{1-\eta_K}\right)^{2p+3/2}(\|v\|_{L^2(\De A_iC'D')}+\|v\|_{L^2(\De A_iB'C')})\nn\\
&\le&C\mathsf{T}\left(\frac{1+\eta_K}{1-\eta_K}\right)^{2p+3/2}\|v\|_{L^2(K_i)}.
\ee
Since $K_i\subset (\De A_iCD)\cup (\De A_iB''C'')$, by the inverse estimate for $hp$ finite element method (cf., e.g., Schwab \cite[Theorem 4.76]{Schwab}), we have
\be
\|\na v\|_{L^2(K_i)}
&\le&\|\na v\|_{L^2(\De A_iCD)}+\|\na v\|_{L^2(\De A_iB''C'')}\nn\\
&\le&Cp^2h_K^{-1}\|v\|_{L^2(\De A_iCD)}+Cp^2h_K^{-1}\|v\|_{L^2(\De A_iB''C'')}.\label{ll2}
\ee
On the other hand, by using Lemma \ref{lem:2.3} again,
\ben
\|v\|_{L^2(\De A_iCD)}&\le&\mathsf{T}\left(\frac{1+\eta_K}{1-\eta_K}\right)^{2p+3/2}\|v\|_{L^2(\De A_iC'D')},\\
\|v\|_{L^2(\De A_iB''C'')}&\le&\mathsf{T}\left(\frac{1+2\de(d_i+\de)^{-1}}{1-2\de(d_i+\de)^{-1}}\right)^{2p+3/2}\|v\|_{L^2(\De A_iBC)}\\
&\le&\mathsf{T}\left(\frac{1+3\eta_K}{1-\eta_K}\right)^{2p+3/2}\|v\|_{L^2(K_i)}.
\een
Inserting these two estimates to \eqref{ll2}, we obtain
\be
\|\na v\|_{L^2(K_i)}\le Cp^2h_K^{-1}\mathsf{T}\left(\frac{1+3\eta_K}{1-\eta_K}\right)^{2p+3/2}\|v\|_{L^2(K_i)}.
\ee
This, together with \eqref{b1}-\eqref{b2}, completes the proof. $\Box$
\end{proof}

We remark that various interface resolving mesh conditions have been made in the literature to obtain the inverse trace inequality in Lemma \ref{lem:2.7}, which is crucial in establishing the stability of unfitted finite element methods. For example, it is assumed in Massjung \cite{Massjung}, Wu and Xiao \cite{Wu} that each local interface $\Ga_K$, $K\in\cM$, is star shaped with respect to some point in $\Om_i$, which allows for the use of a local polar coordinate system.

To proceed, we define the interface penalty function $\al\in L^\infty(\cE)$:
\be\label{xx0}
\al|_e=\al_0\hat a_e\hat\Theta_eh_e^{-1}p^2\ \ \ \ \forall e\in\cE,
\ee
where $\al_0>0$ is some fixed constant
which is taken to be $1$ in all our numerical examples, and
\ben
\hat a_e=\max\{a_K:e\cap\bar K\not=\emptyset\},\ \ \hat\Theta_e=\max\{\Theta_K:e\cap\bar K\not=\emptyset\},\een
with
\be\label{e1}
a_K=\left\{\begin{array}{ll}
\frac{a_1+a_2}2 &\mbox{if }K\in\cM^\Ga,\\
a_i & \mbox{if }K\subset\Om_i.
\end{array}\right.,\ \ \Theta_K=\left\{\begin{array}{ll}
\mathsf{T}\left(\frac{1+3\eta_K}{1-\eta_K}\right)^{4p} &\mbox{if }K\in\cM^\Ga,\\
1 & {\rm otherwise}.
\end{array}\right.
\ee
Here $\mathsf{T}(t)=t+\sqrt{t^2-1}$, $\forall t\ge 1$. We remark that $\eta_K$ is the interface deviation of the interface in $K\in\cM$ defined in Definition \ref{def:2.2}, which is the only place that the geometry of the interface comes into our method. The mesh function $h|_e=(h_{K}+h_{K'})/2$ if $e=\pa K\cap \pa K'\in\cE^{\rm side}$ and $h|_e=h_K$ if $e=K\cap\Ga\in\cE^\Ga$ or $e=\pa K\cap\pa\Om\in\cE^{\rm bdy}$ for some $K\in\cM$.

\begin{lemma}\label{lem:2.8}
We have $\|a^{1/2}\mathsf{L}(v)\|_\cM\le c_{\mathsf{L}}\|\al^{1/2}\lj v\rj \|_\cE\ \ \forall v\in\X_p(\cM)$ for some constant $c_{\mathsf{L}}>0$ independent of $p$, the mesh $\cM$, and the coefficient $a$.
\end{lemma}

\begin{proof} By taking $r=a\mathsf{L}(v)$ in \eqref{ll1}, we have
\ben
\|a^{1/2}\mathsf{L}(v)\|_\cM^2&\le&\|\al^{-1/2}\widehat{a\mathsf{L}(v)}\|_\cE\|\al^{1/2}\lj v\rj\|_\cE\\
&\le&C\left(\sum_{e\in\cE}\|\hat\Theta_e^{-1/2}h_e^{1/2}p^{-1}\widehat{a^{1/2}\mathsf{L}(v)}\|_{L^2(e)}^2\right)^{1/2}
\|\al^{1/2}\lj v\rj\|_\cE\\
&\le&C\left(\sum_{K\in\cM}\sum_{i=1}^2\|\Theta_K^{-1/2}h_K^{1/2}p^{-1}(a^{1/2}\mathsf{L}(v))\|_{L^2(\pa K_i)}^2\right)^{1/2}\|\al^{1/2}\lj v\rj\|_\cE\\
&\le&C\|a^{1/2}\mathsf{L}(v)\|_\cM\|\al^{1/2}\lj v\rj\|_\cE,
\een
where we have used Lemma \ref{lem:2.7} in the interface elements and a scaled version of Lemma \ref{lem:2.6} for the elements not intersecting the interface. This completes the proof. $\Box$
\end{proof}

For any $v\in H^1(\cM)$, we define the DG norm
\ben
\|v\|_{\rm DG}^2=\|a^{1/2}\na_h v\|_{\cM}^2+\|\al^{1/2}\lj v\rj\|_{\bar{\cE}}^2,
\een
where $\|\al^{1/2}\lj v\rj\|_{\bar\cE}^2:=\la\al\lj v\rj,\lj v\rj\ra_{\bar\cE}$. By \eqref{xx}, we know that
\be
\|\al^{1/2}\lj v\rj\|_{\bar\cE}^2
&=&\sum^2_{i=1}\sum_{e\in\cE_i^{\rm side}}\|\al^{1/2}\lj v_i\rj\|^2_{L^2(e)}+\|\al^{1/2}\lj v\rj\|_{\cE^{\Ga}\cup\cE^{\rm bdy}}^2\nn\\
&\ge&\|\al^{1/2}\lj v\rj \|_{\cE}^2.\label{xx1}
\ee

\begin{theorem}\label{thm:2.1}
We have $a_h(v,v)\ge (4+c_{\mathsf{L}}^2)^{-1}\|v\|_{\rm DG}^2\ \ \forall v\in\mathbb{X}_p(\cM)$, where $c_{\mathsf{L}}>0$ is the constant in Lemma \ref{lem:2.8}.

\end{theorem}

\begin{proof} The argument is standard. For any $\de_1\in (0,1)$,  by Lemma \ref{lem:2.8} and \eqref{xx1} we have
\ben
a_h(v,v)&=&\|a^{1/2}\na_h v\|_\cM^2+\|a^{1/2}\mathsf{L}(v)\|_\cM^2-2(a\na_h v,\mathsf{L}(v))_\cM+\|\al^{1/2}\lj v\rj \|_{\bar{\cal{E}}}^2\\
&\ge&\|a^{1/2}\na_h v\|_\cM^2+(1+(1-\de_1)c_{\mathsf{L}}^{-2})\|a^{1/2}\mathsf{L}(v)\|_\cM^2-2(a\na_h v,\mathsf{L}(v))_\cM\\
&+&\de_1\|\al^{1/2}\lj v\rj \|_{\bar\cE}^2.
\een
By the elementary inequality $a^2-2ab+(1+\eps)b^2\ge\frac \eps{1+\eps}a^2$ $\ \ \forall a,b>0,\eps>0$,
we obtain
\ben
a_h(v,v)\ge\frac{(1-\de_1)c_{\mathsf{L}}^{-2}}{1+(1-\de_1)c_{\mathsf{L}}^{-2}}\|a^{1/2}\na_h v\|_\cM^2+
\de_1\|\al^{1/2}\lj v\rj \|_{\bar\cE}^2.
\een
This completes the proof by choosing $\de_1=\frac{\sqrt{1+4c_{\mathsf{L}}^{-2}}-1}{\sqrt{1+4c_{\mathsf{L}}^{-2}}+1}$ to make the coefficients in the above inequality equal and noticing that $\de_1\ge (4+c_{\mathsf{L}}^{2})^{-1}$. $\Box$
\end{proof}



The following a priori error estimate can be proved by using Theorem \ref{thm:2.1}, the classical $hp$-interpolation error estimate in Babu$\check{\rm s}$ka and Suri \cite[Lemma 4.5]{Babuska87b}, and the argument in \cite{Perugia}, \cite{Wu}. Here we omit the details.

\begin{theorem}\label{thm:2.2}
Let the solution of the problem \eqref{p1}-\eqref{p3} $u\in H^{k}(\Om_1\cup\Om_2)$, $k\ge 2$. Let $U\in\X_p(\cM)$ be the solution of \eqref{a2}. Then there exists a constant $C$ independent of $p$, the mesh $\cM$, and the coefficient $a$ such that
\ben
\|u-U\|_{\rm DG}\le C\max_{e\in\cE}|\al|_e|^{1/2}\,\frac{h^{\min(p+1,k)-1}}{p^{k-3/2}}\sum^{2}_{i=1}\|a_i^{1/2}\tilde u_i\|_{H^k(\Om)}.
\een
Here $h=\max_{K\in\cM}h_K$ and $\tilde u_i\in H^{k}(\Om)$ is the Stein extension \cite[P.154]{Adams} of $u_i\in H^{k}(\Om_i)$ for Lipschitz domains satisfying $\|\tilde u_i\|_{H^{k}(\Om)}\le C\|u_i\|_{H^{k}(\Om_i)}$, $i=1,2$.
\end{theorem}

We remark that the error estimate is slightly sub-optimal in $p$ which is typical for discontinuous Galerkin methods (see e.g., Georgoulis, Hall and Melenk \cite{Georgoulis}). However, $hp$-optimal error estimates can be proved in some special cases for discontinuous Galerkin methods for Possion problem on $1$-irregular meshes (each side containing at most 1 hanging node), see Stamm and Wihler \cite{Stamm}.

\section{A posteriori error estimation: reliability}\label{sec:3}

We start by introducing some further notation. We assume the elements in $\cT$ are obtained by local successive quad-refinements of some conforming initial mesh $\cT_0$. A quad-refinement of an element consists of subdividing the element into four congruent rectangles.

Let $\cN^0$ be the set of conforming nodes of the induced mesh $\cM$ from $\cT$ such that each element $K\in\cM$ is large with respect to
both $\Om_1,\Om_2$ and satisfies \eqref{HH}.
A node is called conforming if it either locates on the boundary or is shared by the four elements to which it belongs. For each conforming node $P$, we define $\psi_P\in\mathbb{X}_1(\cM)\cap H^1(\Om)$, which is bilinear in each element and satisfies $\psi_P(Q)=\delta_{PQ}$ for any $Q\in\cN^0$. Here $\delta_{PQ}$ is the Kronecker delta. It is proved in Babu\v{s}ka and Miller \cite{Babuska87a} that $\{\psi_P:P\in\cN^0\}$ consists of a basis of $\X_1(\cM)\cap H^1(\Om)$ and satisfies the property of the partition of unity
\ben
\sum_{P\in\cN^0}\psi_P=1.
\een
We impose the following assumption on the finite element mesh which is first introduced in Babu\v{s}ka and Miller \cite{Babuska87a} as the $K$-mesh (see Figure \ref{fig:3.0}).

\noindent\\
{\bf Assumption (H3)} There exists a constant $C>0$ uniform on the level of discretization of $\cM$ such that for any conforming node $P\in\cN^0$,
\be\label{m2}
\diam({\rm supp}(\psi_P))\le C\min_{K\in\cM_P}h_K,
\ee
where $\cM_P:=\{K\in\cM,\,\,K\subset{\rm supp}(\psi_P)\}$.

\bigskip
\begin{figure}[t]
	\centering
	\includegraphics[width=0.35\textwidth]{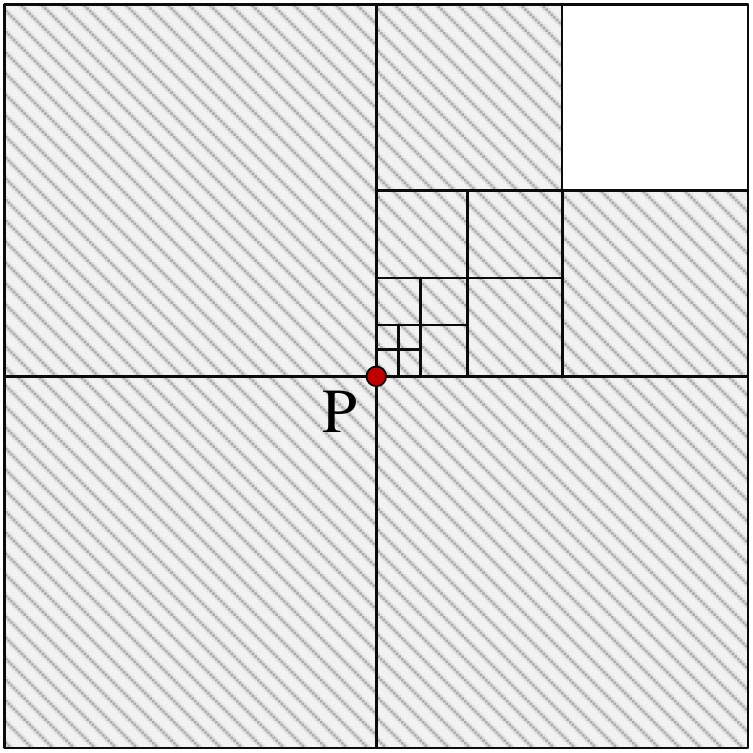}\quad
	\includegraphics[width=0.35\textwidth]{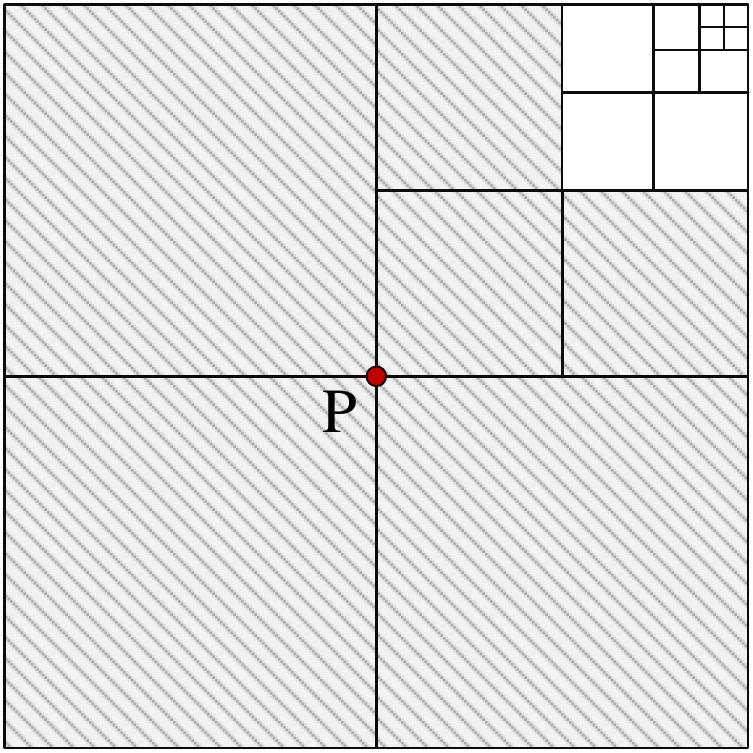}
	\caption{The left mesh is not a $K$-mesh if it refines close to $P$. The right mesh is a $K$-mesh if it
	refines close to upper-right corner. The shadow region is the support of $\psi_P$.}
	\label{fig:3.0}
\end{figure}
We refer to \cite[\S 1.4]{Babuska87a} for further properties of $K$-meshes and Bonito and Nochetto \cite[\S 6]{Bonito} for a refinement algorithm to enforce the assumption (H3) in practical computations.

The a posteriori error analysis depends on a suitable quasi-interpolation operator. In Melenk \cite{Melenk05},  a Cl\'ement type $hp$-quasi-interpolation is constructed for conforming meshes. The following lemma shows that a similar construction leads to a $hp$-quasi-interpolation operator on $K$-meshes.

\begin{lemma}\label{lem:3.1}
Let $\V_p(\cM)=\Pi_{K\in\cM}Q_p(K)$. There exists a quasi-interpolation operator $\Pi_h:H^1_0(\Om)\to\mathbb{V}_p(\cM)\cap H^1_0(\Om)$ such that for any $v\in H^1_0(\Om)$,
\ben
& &\|D^m(v- \Pi_hv)\|_{L^2(K)}\le C(h_K/p)^{1-m}\|\na v\|_{L^2(\omega(K))},\ \ m=0,1,\\
& &\|v-\Pi_h v\|_{L^2(\pa K)}\le C(h_K/p)^{1/2}\|\na v\|_{L^2(\omega(K))}.
\een
Here for any $K\in\cM$, $\omega(K)$ is a union of a discrete set of elements including $K$ such that $\diam(\omega (K))\le Ch_K$. The constant $C$ is independent of $h_K,p$.
\end{lemma}

\begin{figure}[t]
	\centering
	\includegraphics[width=0.8\textwidth]{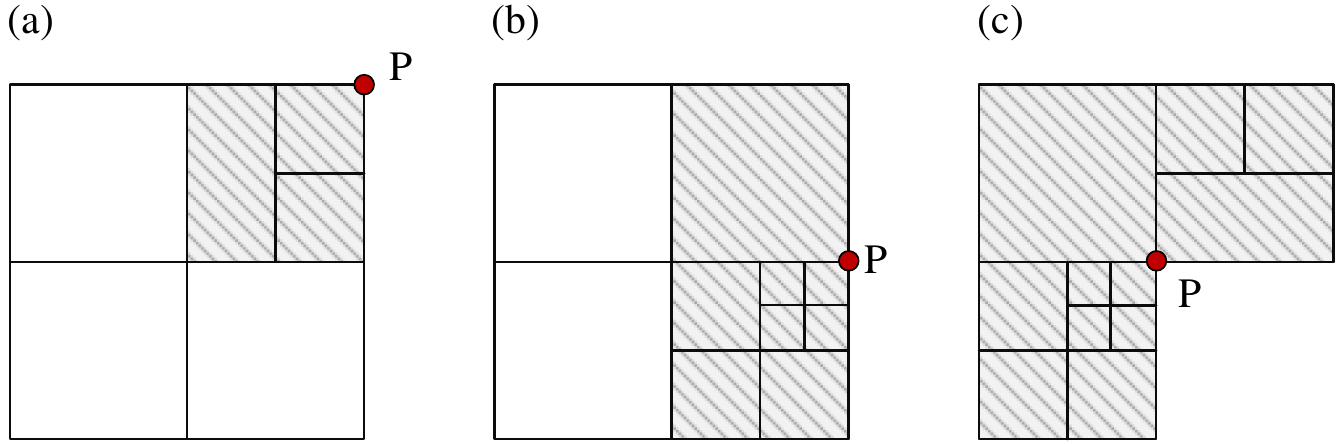}
	\caption{An example of $S_P$ with $P$ is the vertex of (a) one element, (b) two elements, and (c) three elements.}
	\label{fig:3.1}
\end{figure}

\begin{proof} The second estimate follows from the first one by the multiplicative trace inequality \eqref{m1}. We now describe how to construct the operator which satisfies the first estimate by the method in \cite{Melenk05}. For any $P\in\cN^0$, denote $\Om_P=({\rm supp}(\psi_P))^\circ$, the interior of ${\rm supp}(\psi_P)$, and $h_P=\diam(\Om_P)$. For any $v\in H^1_0(\Om)$, which is extended to be zero outside $\Om$, we  define
\be\label{m6}
I_hv=\sum_{P\in\cN^0}(I_Pv)\psi_P,
\ee
where $I_P:H^1_0(\Om)\to\V_{p-1}(\cM_{P})$, is defined by using local projection and polynomial lifting. More precisely, denote $S_P$ the rectangle centered at $P$ which includes $\Om_P$ and has minimum size. Let $J_P:H^1(S_P)\to Q_{p-1}(S_P)$ be the polynomial approximation operator on rectangles
in \cite[Theorem 5.1]{Melenk05} which satisfies
\be\label{m5}
\|D^m(v-J_P v)\|_{L^2(S_P)}\le C(h_P/p)^{1-m}\|\na v\|_{L^2(S_P)},\ \ m=0,1.
\ee
Notice that $J_Pv$ does not vanish on the boundary. Let $P\in\pa\Om\cap\cN^0$ and $\Ga_P=\pa\Om\cap \bar S_P$. Since $v=0$ on $\pa\Om$, we obtain from \eqref{m5} that
\ben
\|(h/p)^{-1/2}J_Pv\|_{L^2(\Ga_P)}+\|J_Pv\|_{H^{1/2}(\Ga_P)}\le C\|\na v\|_{L^2(S_P)}.
\een
We observe that if $P\in\pa\Om$ is the vertex of only one element or two elements, $S_P$ can be chosen to be inside $\Om$ (see Figure \ref{fig:3.1}).  Thus one can use the polynomial lifting theorem in \cite[Proposition 5.3]{Melenk05} to obtain a $v_P\in Q_{4(p-1)}(S_P)$ such that
\be\label{m7}
& &(h_P/p)^{-1}\|v_P\|_{L^2(S_P)}+\|\na v_P\|_{L^2(S_P)}\nn\\
&\le&C\|(h/p)^{-1/2}J_P v\|_{L^2(\Ga_P)}+C\|J_Pv\|_{H^{1/2}(\Ga_P)}\nn\\
&\le&C\|\na v\|_{L^2(S_P)}.
\ee
If $P\in\pa\Om$ is the vertex of three elements, then $S_P\cap\Om$ is the union of three rectangles $S_{P}^j$, $j=1,2,3$, such that each element in $\cM_P$ is included in one of these three elements
(see Figure \ref{fig:3.1}).
In this case, one can use the argument in \cite[Lemma 5.8]{Melenk05} to conclude that there exists a $v_P\in \left[\Pi_{j=1,2,3}Q_{4(p-1)}(S_P^j)\right]\cap H^1(S_P)$ such that \eqref{m7} is valid.

Now we define $I_Pv=J_Pv$ if $P\in\cN^0$ is an interior node and $I_Pv=J_Pv-v_P$ if $P\in\cN^0$ is a node on the boundary. By using the partition of unity \eqref{m2}, \eqref{m5} and \eqref{m7}, we obtain easily
\ben
\|D^m(v-I_hv)\|_{L^2(K)}\le C(h_K/p)^{1-m}\|\na v\|_{L^2(\omega(K))},\ \ m=0,1.
\een
Finally, since $I_hv\in \V_{4(p-1)+1}(\cM)\cap H^1_0(\Om)$, we define $\Pi_h$ by replacing $p$ in \eqref{m6} by $\lfloor (p-1)/4\rfloor+1$. This proves the lemma. $\Box$
\end{proof}

\begin{remark}\label{rem:3.1}
We know from the proof of Lemma \ref{lem:3.1} that for any $K\in\cM$,
\ben
\omega(K)=\{K'\in\cM:K'\subset S_P,\forall P\in\cN^0\ \mbox{such that }\psi_P|_K\not=0\}.
\een
\end{remark}

The following local smoothing operator on $K$-meshes extends the construction in Burman and Ern \cite{Burman07}, Houston, Sch\"otzau and Wihler \cite{Houston} for conforming meshes and Zhu and Sch\"otzau \cite{Zhu} for $1$-irregular meshes.

\begin{lemma}\label{lem:3.2}
There exists an interpolation operator $\pi_h:\V_p(\cM)\to\mathbb{V}_p(\cM)\cap H^1(\Om)$ such that for any $v\in\V_p(\cM)$,
\ben
& &\|v-\pi_h v\|_{L^2(K)}\le C\|p^{-1}h^{1/2}\lj v\rj\|_{L^2(\sigma(K))},\\
&&\|\na(v-\pi_h v)\|_{L^2(K)}\le C\|ph^{-1/2}\lj v\rj\|_{L^2(\sigma (K))},
\een
where $\si(K)=\{e\in\cE^{\rm side}:e\subset\widetilde\omega(K)\}$, $\widetilde\omega(K)$ is a set of elements including $K$ such that $\diam(\widetilde\omega (K))\le Ch_K$. The constant $C$ is independent of $h_K,p$. Moreover, $\pi_hv\in H^1_0(\Om)$ if $v=0$ on $\pa\Om$.
\end{lemma}

\begin{proof} Let $\hat K=I\times I$, $I=(-1,1)$, be the reference element. Let $\widehat\cN_p$ be the Gauss-Legendre-Lobatto grid of $\hat K$, that is, $\widehat\cN_p=\{(\xi_i,\xi_j)^T\in\hat K: 0\le i\le p\}$, where $\xi_i,0\le i\le p$, are the zeros of the polynomial $(1-\xi^2)L'_p(\xi)$. Here $\{L_n\}_{n\ge 0}$ is the set of Legendre polynomials. Let $\{\hat\phi_i\}_{i=0}^p$ be the set of Lagrange interpolation functions in $Q_p(\Lam)$ corresponding to the Gauss-Legendre-Lobatto nodes, that is, $\hat\phi_i\in Q_p(\Lam)$, $\hat\phi_i(\xi_j)=\de_{ij}$, $0\le i,j\le p$. Here $\de_{ij}$ is the Kronecker delta.

It is known by the differential equation satisfied by the Legendre polynomials that
\ben
\hat\phi_i(\xi)=\frac{-1}{p(p+1)}\frac{(1-\xi^2)L'_p(\xi)}{(\xi-\xi_i)L_p(\xi_i)},\ \ \ \ 0\le i\le p.
\een
Notice that $\|L'_p\|_{L^2(\Lam)}=\sqrt{p(p+1)}$, $L_p(\pm 1)=(\pm 1)^p$, we have
\be\label{d4}
\|\hat\phi_0\|_{L^2(\Lam)}\le [p(p+1)]^{-1}\|(1-\xi)L_p'\|_{L^2(\Lam)}\le 2/\sqrt{p(p+1)}.
\ee
Similarly, $\|\hat\phi_p\|_{L^2(\Lam)}\le 2/\sqrt{p(p+1)}$.

For any $K\in\cM$, let $F_K:\hat K\to K$ be the affine mapping. Denote $\cN_p(K)=F_K(\widehat\cN_p)$  the set of Gauss-Legendre-Lobatto nodes on $K$. The degrees of freedom of a function in $Q_p(K)$ are its nodal values at $\cN_p(K)$. The set of basis functions of $Q_p(K)$ is $\{\phi_P=\hat\phi_{\hat P}\circ F_K^{-1}:P=F_K(\hat P)\}$.
Here $\hat\phi_{\hat P}$ is the nodal basis of $\hat Q_p(\hat K)$ corresponding to $\hat P\in\widehat\cN_p$.

To construct the interpolation operator, we classify the set of nodes and sides of the mesh $\cM$. Let $\cN^0$ be the set of conforming nodes. For $k\ge 1$, let $\cN^k$ be the subset of nodes that are located on some side $e\in\cE^{\rm side}$ whose end points are in $\cN^m, 0\le m\le k-1$, and with at least one end point in $\cN^{k-1}$. By the assumption (H3), the maximum number of levels $L$ of the classification of the nodes is uniformly bounded.

For $1\le k\le L+1$, we denote $\cE^k\subset\cE^{\rm side}$ the collection of sides whose end points are in $\cN^m, 0\le m\le k-1$, and with at least one end point in $\cN^{k-1}$. Clearly, $\cE^k\cap\cE^l=\emptyset$ if $k\not= l$ and $\cE^1$ is the set of sides whose end points are conforming nodes. For any $v\in\V_p(\cM)$, we define $\pi_h^k v\in P_p(\cE^k)$, the set of polynomials of order $p$ in each side of $\cE^k$, successively as follows.
\begin{enumerate}
\item If $e\in\cE^1$ whose end points $P_1,P_2\in\cN^0$, $e=\pa K\cap\pa K'$, $K,K'\in\cM$, and $K'$ is the element such that the length of its side including $e$ is larger or equal to $|e|$, we define
\be\label{d1}
\pi^1_hv=v|_{K'}+\sum_{i=1}^2\left[(\pi_h^0v)(P_i)-(v|_{K'})(P_i)\right]\phi_{P_i}\ \ \mbox{on }e,
\ee
where for $P\in\cN^0$, $
(\pi^0_hv)(P)=\frac 1{\#\{K\in\cM:P\in\bar K\}}\sum_{K\in\cM,P\in \bar K}(v|_K)(P)$, the local average of $v$
sharing $P$ as the common vertex.
Here the boundary value of $v|_K$ is understood as its trace.
\item For $k\ge 2$, $e\in\cE^k$ whose end points $P_i\in\cN^{m_i} (i=1,2)$, $e=\pa K\cap\pa K'$, $K,K'\in\cM$, and $K'$ is the element such that the length of its side including $e$ is larger or equal to $|e|$, we define
\be\label{d2}
\pi^k_hv=v|_{K'}+\sum_{i=1}^2\left[(\pi_h^{m_i}v)(P_i)-(v|_{K'})(P_i)\right]\phi_{P_i}\ \ \mbox{on }e.
\ee
Since for $e\in\cE^k$, $0\le m_i\le k-1$, $i=1,2$, \eqref{d2} is well defined. Obviously, $(\pi^k_hv)(P_i)=(\pi^{m_i}_hv)(P_i)$, $i=1,2$.
\end{enumerate}

We define $(\pi_h v)|_e=(\pi_h^k v)|_e$ if $e\in\cE^k, 1\le k \le L+1$. Then $\pi_h v$ is piecewise polynomial of order $p$ and continuous on $\cE^{\rm side}$. Moreover, $\pi_hv=0$ on $\pa\Om$ if $v=0$ on $\pa\Om$. Having defined the $\pi_h v$ on $\cE^{\rm side}$ we now define $\pi_h v$ on each element $K\in\cM$ as
\ben
\pi_h v=\sum_{P\in\cN_p(K),P\not\in\pa K}v(P)\phi_P+\sum_{P\in\cN_p(K),P\in\pa K}(\pi_hv)(P)\phi_P.
\een
Then $v-\pi_hv\in Q_p(K)$ and vanishes in all interior Gauss-Legrendre-Lobatto nodes, by the inverse trace inequality in Burman and Ern \cite[Lemma 3.1]{Burman07}, we have
\be\label{d5}
\|v-\pi_h v\|_{L^2(K)}\le Cp^{-1}h_K^{1/2}\sum_{e\subset\pa K}\|v|_K-\pi_h v\|_{L^2(e)}.
\ee
Let $e\subset\pa K$ and $e\in\cE^k$ for some $1\le k\le L+1$. There exists a conforming node $P$ such that $e\in\cE_P=\{e\in\cE^{\rm side}:e\subset{\rm supp}(\psi_P)\}$. By definition, $e$ has the end points $P_i\in\cN^{m_i}, m_i\le k-1$, $i=1,2$, and one of $m_1,m_2$ is $k-1$. If $P_i\not\in\cN^0$, then $\psi_P(P_i)\not=0$ and it is a hanging node of some $e'_i\in\cE^{m_i}$. The crucial observation is that $e'_i\in\cE_P$. Thus by \eqref{d2} and using \eqref{d4} we have
\ben
\|v|_K-\pi_h v\|_{L^2(e)}&=&\|v|_K-\pi_h^k v\|_{L^2(e)}\\
&\le&\|\lj v\rj\|_{L^2(e)}+Cp^{-1}h_K^{1/2}\sum^2_{i=1}|(v|_{e_i'}-\pi^{m_i}_hv)(P_i)|.
\een
By the inverse estimate
\ben
|(v|_{e'_i}-\pi^{m_i}_hv)(P_i)|\le\|v-\pi^{m_i}_hv\|_{L^\infty(e_i')}\le Cph_K^{-1/2}\|v-\pi^{m_i}_hv\|_{L^2(e'_i)}.
\een
Combining above two inequalities we obtain
\ben
\|v|_K-\pi_h v\|_{L^2(e)}&\le&\|\lj v\rj\|_{L^2(e)}+C\max_{\stackrel{e'\in\cE_P}{e'\in\cE^m,\,1\le m\le k-1}}\|v-\pi^m_hv\|_{L^2(e')}\\
& &+Cp^{-1}h_K^{1/2}\max_{Q\in\cN^0,Q\in{\rm supp}(\psi_P)}|(v-\pi^0_hv)(Q)|.
\een
By the mathematical induction, since $k\le L+1$ and $L$ is uniformly bounded according to (H3), we obtain
\ben
\|v|_K-\pi_h v\|_{L^2(e)}&\le&\|\lj v\rj\|_{\cE_P}+C\max_{e'\in\cE_P,e'\in\cE^1}\|v-\pi^0_hv\|_{L^2(e')}\\
& &+Cp^{-1}h_K^{1/2}\max_{Q\in\cN^0,Q\in{\rm supp}(\psi_P)}|(v-\pi^0_hv)(Q)|\\
&\le&\|\lj v\rj\|_{\cE_P}+Cp^{-1}h_K^{1/2}\max_{Q\in\cN^0,Q\in{\rm supp}(\psi_P)}|(v-\pi^0_hv)(Q)|,
\een
where we have used \eqref{d1} in the second estimate. Since $(\pi_h^0 v)(Q)$ is the local average of $v$
sharing $Q$ as the common vertex, we have
\ben
|(v-\pi^0_hv)(Q)|\le\sum_{Q\in \bar e', e'\in\cE^{\rm side}}\|\lj v\rj\|_{L^\infty(e')}\le C\sum_{Q\in \bar e', e'\in\cE^{\rm side}}\|ph^{-1/2}\lj v\rj \|_{L^2(e')}.
\een
By using the assumption (H3), we conclude that
\ben
\|v|_K-\pi_h v\|_{L^2(e)}&\le& C\|\lj v\rj\|_{L^2(\sigma (K))},
\een
where $\sigma(K)$ is set of sides included in some $\widetilde\omega(K)$ which is a union of elements surrounding $K$ whose diameter is bounded by $Ch_K$. This shows the first estimate of the lemma by \eqref{d5}. The
second estimate can be proved by the standard inverse estimate
\ben
\|\na(v-\pi_hv)\|_{L^2(K)}\le Cp^2h_K^{-1}\|v-\pi_h v\|_{L^2(K)}\le C\|ph^{-1/2}\lj v\rj\|_{L^2(\si(K))}.
\een
This completes the proof. $\Box$
\end{proof}

Let $\Sigma$ be a Lipschitz curve in $\R^2$, we recall the definition of the Aronszaja-Slobodeckij norm $\|v\|_{H^{1/2}(\Sigma)}=(\|v\|_{L^2(\Sigma)}^2+|v|_{H^{1/2}(\Sigma)}^2)^{1/2}$, where
\ben
|v|_{H^{1/2}(\Sigma)}^2=\int_\Sigma\int_\Sigma\frac{|v(x)-v(y)|^2}{|x-y|^2}ds(x)ds(y).
\een
The following Gagliardo-Nirenberg type estimate for $H^{1/2}$-seminorm is well known (see e.g., Triebel \cite{Triebel}).

\begin{lemma}\label{lem:3.0}
Let the interval $(a,b)\subset\R$ and $v\in H^1(a,b)$. Then $|v|_{H^{1/2}(a,b)}\le C\|v\|_{L^2(a,b)}^{1/2}\|v'\|_{L^2(a,b)}^{1/2}$ for some constant $C$ independent of $(a,b)$.
\end{lemma}

By definition, any function $v\in\X_p(\cM)$ can be written as $v=v_1\chi_{\Om_1}+v_2\chi_{\Om_2}$ for some $v_i\in\V_p(\cM_i)$. In the following, we still denote by $v_i$ the function in $\V_p(\cM)$ which is obtained by zero extension of $v_i$ outside $\Om_i^h$, $i=1,2$.

\begin{lemma}\label{lem:3.3}
There exists a linear operator $\pi^c_h:\X_p(\cM)\to H^1(\Om)$ such that
\ben
\|a^{1/2}\na_h(v-\pi^c_h v)\|_\cM&\le&C\left(\sum^2_{i=1}\|\hat a^{1/2}ph^{-1/2}\lj v_i\rj\|_{\cE_i^{\rm side}}+\|\hat a^{1/2}ph^{-1/2}\lj v\rj\|_{\cE^\Ga}\right)\\
& &+C\|\hat a^{1/2}p^{-1}h^{1/2}\na_\Ga\lj v\rj \|_{\cE^\Ga}.
\een
Here $\na_\Ga$ is the tangential gradient on $\Ga$. Moreover, $\pi_h^cv=\pi_hv_i$ on $\pa\Om$ if $\pa\Om_i\cap\pa\Om\not=\emptyset$, $i=1,2$.
\end{lemma}

\begin{proof} Without loss of generality, we assume $a_1\le a_2$. By Lemma \ref{lem:3.2}, for $v_i\in\V_p(\cM_i)$, $i=1,2$, there exists $\pi_h v_i\in\V_p(\cM_i)\cap H^1(\Om_i^h)$ such that for any $K\in\cM_i$,
\be
& &\|v_i-\pi_h v_i\|_{L^2(K)}\le C\|p^{-1}h^{1/2}\lj v_i\rj\|_{L^2(\si(K))},\label{n1}\\
& &\|\na(v_i-\pi_h v_i)\|_{L^2(K)}\le C\|ph^{-1/2}\lj v_i\rj\|_{L^2(\si(K))}.\label{n2}
\ee
Let $w_1\in H^1(\Om_1)$ satisfy
\ben
-\De w_1=0\ \ \mbox{in }\Om_1,\ \ w_1=\lj\pi_h v\rj_\Ga \ \ \mbox{on }\Ga,\ \ w_1=0\ \ \mbox{on }\pa\Om_1\bks\Ga.
\een
We define $\pi_h^cv:=(\pi_h v_1-w_1)\chi_{\Om_1}+(\pi_h v_2)\chi_{\Om_2}$.
Obviously, $\pi^c_hv\in H^1(\Om)$. By \eqref{n2},
\be\label{n5}
& &\|a^{1/2}\na_h(v-\pi^c_hv)\|_\cM\nn\\
&\le&C\left(a^{1/2}_1\|\lj\pi_h v\rj\|_{H^{1/2}(\Ga)}+\sum^2_{i=1}\|ph^{-1/2}\lj a_i^{1/2}v_i\rj\|_{\cE_i^{\rm side}}\right).
\ee
We now estimate $\|\lj\pi_h v\rj\|_{H^{1/2}(\Ga)}$. We know from the construction of the finite element space that $\Ga=\cup_{K\in\cM}\Ga_{K}$. Since $K$ is large with respect to both $\Om_1,\Om_2$, the partition $\{\Ga_{K}, K\in\cM\}$ of $\Ga$ is shape regular in the sense that
\be\label{n3}
|\Ga_{K}|/|\Ga_{K'}|\le C_0,\ \ \forall K,K'\in\cM^\Ga, \ \ K,K'\ \mbox{are adjacent}.
\ee
Let
\ben
\omega(\Ga_{K })=\cup\{\Ga_{K'}:\bar K'\cap \bar K\not=\emptyset\}
\een
be the set of neighboring curve segment of $\Ga_{K}$. By the localization lemma of the $H^{1/2}$ semi-norm in Faermann \cite[Lemma 2.3]{Faermann}, we know that
\ben
|\lj\pi_h v\rj |_{H^{1/2}(\Ga)}^2\le\sum_{K\in\cM}|\lj\pi_h v\rj|^2_{H^{1/2}(\omega(\Ga_{K}))}+C\sum_{K\in\cM}h_K^{-1}\|\lj\pi_h v\rj\|_{L^2(\Ga_K)}^2,
\een
where the constant $C$ depends on  the Lipschitz constant of the curve $\Ga$ and the shape regularity
constant $C_0$ in \eqref{n3}. Now by Lemma \ref{lem:3.0} we obtain easily
\ben
\sum_{K\in\cM}|\lj\pi_h v\rj |_{H^{1/2}(\omega(\Ga_{K}))}^2\le C\sum_{K\in\cM}\|\lj\pi_h v\rj\|_{L^2(\Ga_K)}\|\na_\Ga\lj\pi_h v\rj\|_{L^2(\Ga_K)}.
\een
Therefore,
\be\label{n4}
& &|\lj\pi_h v\rj|_{H^{1/2}(\Ga)}^2\nn\\
&\le&C\sum_{K\in\cM}\left(\|\lj\pi_h v\rj\|_{L^2(\Ga_K)}\|\na_\Ga\lj\pi_h v\rj\|_{L^2(\Ga_K)}+h_K^{-1}\|\lj\pi_h v\rj\|_{L^2(\Ga_K)}^2\right).
\ee
It is easy to see that
\ben
\|\na_\Ga\lj\pi_h v\rj\|_{L^2(\Ga_K)}\le\sum^2_{i=1}\|\na(v_i-\pi_h v_i)\|_{L^2(\Ga_K)}+\|\na_\Ga\lj v\rj\|_{L^2(\Ga_K)}.
\een
By Lemma \ref{lem:2.5}, the trace inequality \eqref{m1}, the inverse estimate, and Lemma \ref{lem:3.2} we have
\ben
& &\|\na(v_i-\pi_h v_i)\|_{L^2(\Ga_K)}\\
&\le&C\left(h_K^{-1/2}\|\na(v_i-\pi_hv_i)\|_{L^2(K)}+\|\na(v_i-\pi_h v_i)\|_{L^2(K)}^{1/2}
\|D^2(v_i-\pi_hv_i)\|_{L^2(K)}^{1/2}\right)\\
&\le&Cph_K^{-1/2}\|\na(v_i-\pi_hv_i)\|_{L^2(K)}\\
&\le&Cp^2h_K^{-1/2}\|h^{-1/2}\lj v_i\rj\|_{L^2(\si(K))}.
\een
Thus
\ben
\|\na_\Ga\lj\pi_h v\rj\|_{L^2(\Ga_K)}\le \|\na_\Ga\lj v\rj\|_{L^2(\Ga_K)}+C\sum^2_{i=1}p^2h_K^{-1/2}\|h^{-1/2}\lj v_i\rj\|_{L^2(\si(K))}.
\een
Similarly,
\ben
\|\lj\pi_h v\rj\|_{L^2(\Ga_K)}\le \|\lj v\rj\|_{L^2(\Ga_K)}+C\sum^2_{i=1}\|\lj v_i\rj\|_{L^2(\si(K))}.
\een
By substituting above two estimates into \eqref{n4} we have
\ben
\|\lj\pi_h v\rj_{H^{1/2}(\Ga)}&\le&C\|ph^{-1/2}\lj v\rj\|_{L^2(\Ga_K)}+C\|p^{-1}h^{1/2}\na_\Ga\lj v\rj\|_{L^2(\Ga_K)}\\
&+&C\sum^2_{i=1}\|ph^{-1/2}\lj v_i\rj\|_{\cE^{\rm side}}.
\een
This completes the proof by \eqref{n5} and the fact that $a_1\le\hat a_e\ \ \forall e\in\cE^\Ga\cup\cE_1^{\rm side}$, and $a_2\le 2\hat a_e\ \ \forall e\in\cE_2^{\rm side}$. $\Box$
\end{proof}

Let $U\in\X_p(\cM)$ be the solution of the problem \eqref{a2}, we define the element and jump residuals
\ben
& &R(U)|_K=f+{\rm div}_h(a\na_h U)\ \ \ \ \forall K\in\cM,\\
& &J(U)|_e=\lj a\na_h U\cdot n\rj_e\ \ \ \ \forall e\in\cE^{\rm side}\cup\cE^\Ga.
\een
We also define the functions $\Lam:\Pi_{K\in\cM}L^2(K)\to\R$ and $\hat\Lam:\Pi_{e\in\cE}L^2(e)\to\R$ as
\ben
& &\Lam|_K=\|a^{1/2}\|_{L^\infty(K)}\|a^{-1/2}\|_{L^\infty(\omega(K))}\ \ \ \ \forall K\in\cM,\\
& &\hat\Lam|_e=\max\{\Lam_K:e\cap\bar K\not=\emptyset\}\ \ \ \ \forall e\in\cE.
\een
Here $\omega(K)$ is defined in Remark \ref{rem:3.1}. We remark that $\Lam,\hat\Lam$ are one on the elements or sides away from the interface.

The following theorem is the main result of this section.

\begin{theorem}\label{thm:3.1}
Let $u\in H^1(\Om)$ be the weak solution of \eqref{p1}-\eqref{p3} with $g\in H^1(\pa\Om)$ and $U\in\mathbb{X}_p(\cM)$ be the solution of \eqref{a2}. Then there exists a constant $C$ independent of the coefficient $a$, the mesh $\cM$, the interface $\Ga$, and the ratio ${\max(a_1,a_2)}/{\min(a_1,a_2)}$ such that
\ben
\|u-U\|_{\rm DG}\le C\left(\sum_{K\in\cM}\xi_K^2\right)^{1/2},
\een
where for each $K\in\cM$, the local a posteriori error estimator
\ben
\xi_K^2&=&\Big(\|a^{-1/2}(h/p)\Lam R(U)\|_{K}^2
+\|\hat a^{-1/2}(h/p)^{1/2}\hat\Lam J(U)\|_{\cE_K\cup\Ga_K}^2\Big)\\
&&+\Big(\,\sum^2_{i=1}\|\al^{1/2}\hat\Lam \lj U_i\rj\|_{\cE_K^i}^2+\|\al^{1/2}\hat\Lam \lj U\rj\|_{\Ga_K}^2+\|\al^{1/2}(U-g)\|_{\pa K\cap\pa\Om}^2\Big)\\
&&+\Big(\,\|\hat a^{1/2}p^{-1}h^{1/2}\na_\Ga\lj U\rj\|_{\Ga_K}^2
+\|\hat a^{1/2}p^{-1}h^{1/2}\na_{\pa\Om}(U-g)\|_{\pa K\cap\pa\Om}^2\Big),
\een
$\na_{\pa\Om}$ is the tangential derivative on the boundary $\pa\Om$, $\cE_K=\{e\in\cE^{\rm side}:e\subset\pa K\}$, and $\cE_K^i=\{e\in\cE_i^{\rm side}:e\subset\pa K\}$. 
\end{theorem}

We remark that by \eqref{xx}, the sum of the second term in $\xi_K^2$ over $K\in\cM$ is equivalent to $\|\al^{1/2}\lj u-U\rj\|_{\bar\cE}^2$ up to the factor $\hat\Lam^2$. The sum of the third term in $\xi_K^2$ over $K\in\cM$ is roughly of the same order as the sum of the second term. The local lower bounds of the first term in $\xi_K^2$ will be studied in the next section.

We also remark that the factors $\Lam,\hat\Lam$ in the theorem are absent in the a posteriori error estimate in Cai, Ye and Zhang \cite{Cai} under the assumption that the mesh fits the interface and the coefficient is quasi-monotone with respect to each node of the mesh. The quasi-monotone property of the diffusion coefficient was first introduced in Petzoldt \cite{Petzoldt} and it also played an important role in Chen and Dai \cite{Chen02} for the study of coefficient robust a posteriori error estimates for conforming finite element methods.

\begin{proof}
Let $\tilde U\in H^1(\Om)$ satisfy $\tilde U=g$ on $\pa\Om$, and
\be\label{c1}
\int_\Om a\na\tilde U\cdot\na vdx=\int_\Om a\na_h U\cdot \na v dx\ \ \ \ \forall v\in H^1_0(\Om).
\ee
By the Lax-Milgram lemma, $\tilde U\in H^1(\Om)$ is well defined. By the triangle inequality, we
have
\be\label{c2}
& &\|u-U\|_{\rm DG}\nn\\
&\le&\|u-\tilde U\|_{\rm DG}+\|U-\tilde U\|_{\rm DG}\nn\\
&\le&\|a^{1/2}\na (u-\tilde U)\|_\cM+\|a^{1/2}\na_h(U-\tilde U)\|_\cM+\|\al^{1/2}\lj U-\tilde U\rj\|_{\bar\cE}.
\ee
By the definition in \eqref{xx}
\ben
\|\al^{1/2}\lj U-\tilde U\rj\|_{\bar\cE}^2&=&\sum^{2}_{i=1}\|\al^{1/2}\lj U_i\rj\|_{\cE_i^{\rm side}}^2
+\|\al^{1/2}\lj U\rj \|_{\cE^\Ga}^2+\|\al^{1/2}(U-g)\|_{\cE^{\rm bdy}}^2.
\een
Thus we are left to bound the first two terms in \eqref{c2} since $\hat\Lam\ge 1$ on $\cE$.

$1^\circ$ We first estimate the conforming component $\|a^{1/2}\na(u-\tilde U)\|_{\cM}$ of the error. For any $w\in H^1_0(\Om)$, we take $w_h=\Pi_h w\in\V_p(\cM)\cap H^1_0(\Om)\subset\X_p(\cM)$. Since $\mathsf{L}(w_h)=0$ we obtain from the discrete equation \eqref{a2} that
\ben
(a\na_h U,\na_h w_h)_\cM-(a\mathsf{L}(U),\na_h w_h)_\cM=(f,w_h)_\cM-(a\mathsf{L}_1(g),\na_h w_h)_\cM.
\een
This yields by \eqref{c1} that
\ben
(a\na (u-\tilde U),\na w)_\cM&=&(f,w)_\cM-(a\na_h U,\na w)_\cM\\
&=&(f,w-w_h)_\cM-(a\na_h U,\na(w-w_h))_\cM\\
& &-\,(a\mathsf{L}(U),\na_hw_h)_\cM+(a\mathsf{L}_1(g),\na_h w_h)_\cM.
\een
Since $w-w_h\in H^1_0(\Om)$, by doing integration by parts we have
\ben
(a\na(u-\tilde U),\na w)_\cM
&=&(R(U),w-w_h)_\cM-\la J(U),w-w_h\ra_\cE\\
& &-\,(a\mathsf{L}(U),\na_hw_h)_\cM+(a\mathsf{L}_1(g),\na_h w_h)_\cM\\
&:=&{\rm I}_1+{\rm I}_2+{\rm I}_3+{\rm I}_4.
\een
By Lemma \ref{lem:3.1} we have
\ben
|{\rm I_1}+{\rm I}_2|&\le&C\|a^{-1/2}(h/p)\Lam R(U)\|_\cM\|a^{1/2}\na w\|_\cM\\
&+&C\|\hat a^{-1/2}(h/p)^{1/2}\hat\Lam J(U)\|_{\cE^{\rm side}\cup\cE^\Ga}\|a^{1/2}\na w\|_\cM.
\een
Moreover, by Lemma \ref{lem:3.1},
\ben
|{\rm I}_3+{\rm I}_4|&=&|-\la\lj U\rj,\widehat{a\na_hw_h}\cdot n\ra_\cE+\la g,\widehat{a\na_hw_h}\cdot n\ra_{\cE^{\rm bdy}}|\\
&\le&C\left(\|\al^{1/2}\hat\Lam\lj U\rj\|_{\cE^{\rm side}\cup\cE^\Ga}+\|\al^{1/2}(U-g)\|_{\cE^{\rm bdy}}\right)\|a^{1/2}\Lam^{-1}\na_h w_h\|_\cM\\
&\le&C\left(\|\al^{1/2}\hat\Lam\lj U\rj\|_{\cE^{\rm side}\cup\cE^\Ga}+\|\al^{1/2}(U-g)\|_{\cE^{\rm bdy}}\right)\|a^{1/2}\na w\|_\cM.
\een
This shows
\be\label{nn1}
& &\|a^{1/2}\na(u-\tilde U)\|_{L^2(\Om)}\nn\\
&\le&C\|a^{-1/2}(h/p)\Lam R(U)\|_{\cM}+C\|\hat a^{-1/2}(h/p)^{1/2}\hat\Lam J(U)\|_{{\cE^{\rm side}\cup\cE^\Ga}}\nn\\
&&+C\|\al^{1/2}\hat\Lam\lj U\rj\|_{\cE^{\rm side}\cup\cE^\Ga}+C\|\al^{1/2}(U-g)\|_{\cE^{\rm bdy}}.
\ee

$2^\circ$ We next estimate the nonconforming component $\|a^{1/2}\na_h(U-\tilde U)\|_\cM$ of the error
in \eqref{c2}. By \eqref{c1} we know that
\ben
\|a^{1/2}\na_h(U-\tilde U)\|_\cM&\le&\inf_{\stackrel{w\in H^1(\Om)}{w=g\,\,\mbox{\scriptsize on }\pa\Om}}\|a^{1/2}\na_h(U-w)\|_\cM\\
&\le&\|a^{1/2}\na_h(U-\pi^c_hU)\|_\cM+\inf_{\stackrel{w\in H^1(\Om)}{w=g\,\,\mbox{\scriptsize on }\pa\Om}}\|a^{1/2}\na(\pi_h^cU-w)\|_\cM.
\een
Let $\psi\in H^1(\Om)$ satisfy $-\De\psi=0$ in $\Om$, $\psi=\pi_h^cU-g\in H^{1/2}(\pa\Om)$. Then $\|\psi\|_{H^1(\Om)}\le C\|\pi_h^cU-g\|_{H^{1/2}(\pa\Om)}$. Thus $w=\pi_h^cU-\psi\in H^1(\Om)$ satisfies $w=g$ on $\pa\Om$, which yields
\ben
\inf_{\stackrel{w\in H^1(\Om)}{w=g\,\,\mbox{\scriptsize on }\pa\Om}}\|a^{1/2}\na(\pi_h^cU-w)\|_\cM\le Ca_j^{1/2}\|\pi^c_hU-g\|_{H^{1/2}(\pa\Om)},
\een
where $j=1,2$ such that $\pa\Om_j\cap\pa\Om\not=\emptyset$. Similar to the argument in the proof of Lemma \ref{lem:3.3}, we can use the localization lemma of the $H^{1/2}$ semi-norm in Faermann \cite[Lemma 2.3]{Faermann} and Lemma \ref{lem:3.0} to obtain
\ben
& &\|\pi^c_hU-g\|_{H^{1/2}(\pa\Om)}\\
&\le&C(\|ph^{-1/2}(\pi^c_hU-g)\|_{\cE^{\rm bdy}}+\|p^{-1}h^{1/2}\na_{\pa\Om}(\pi^c_hU-g)\|_{\cE^{\rm bdy}}).
\een
Since by Lemma \ref{lem:3.3}, $\pi_h^cU=\pi_h U_j$ on $\pa\Om$ for $\pa\Om_j\cap\pa\Om\not=\emptyset$, we have by the triangle inequality that
\ben
& &\|\pi^c_hU-g\|_{H^{1/2}(\pa\Om)}\\
&\le&C(\| ph^{-1/2}(\pi_hU_j-U_j)\|_{\cE^{\rm bdy}}+\| p^{-1}h^{1/2}\na_{\pa\Om}(\pi_hU_j-U_j)\|_{\cE^{\rm bdy}})\\
&+&C(\| ph^{-1/2}(U-g)\|_{\cE^{\rm bdy}}+\| p^{-1}h^{1/2}\na_{\pa\Om}(U-g)\|_{\cE^{\rm bdy}}).
\een
By inverse trace inequality in Lemma \ref{lem:2.6} and Lemma \ref{lem:3.2},
\ben
& &\| ph^{-1/2}(\pi_hU_j-U_j)\|_{\cE^{\rm bdy}}+\| p^{-1}h^{1/2}\na_{\pa\Om}(\pi_hU_j-U_j)\|_{\cE^{\rm bdy}}\\
&\le&C(\|p^2h^{-1}(\pi_hU_j-U_j)\|_{\cM_j}+\|\na_h(\pi_hU_j-U_j)\|_{\cM_j}\\
&\le&C\|ph^{-1/2}\lj U_j\rj\|_{\cE_j^{\rm side}}.
\een
Combining above estimates and using Lemma \ref{lem:3.3}, we conclude
\ben
& &\|a^{1/2}\na_h(U-\tilde U)\|_\cM\\
&\le&C\left(\sum^2_{i=1}\|\al^{1/2}\lj U_i\rj\|_{\cE_i^{\rm side}}+\|\al^{1/2}\lj U\rj\|_{\cE^\Ga}+\|\al^{1/2}(U-g)\|_{\cE^{\rm bdy}}\right)\\
&&+C\left(\|\hat a^{1/2}p^{-1}h^{1/2}\na_\Ga\lj U\rj\|_{\cE^\Ga}+\|\hat a^{1/2} p^{-1}h^{1/2}\na_{\pa\Om}(U-g)\|_{\cE^{\rm bdy}}\right).
\een
This completes the proof by \eqref{c2} and \eqref{nn1}. $\Box$
\end{proof}

To conclude this section we refer to Sacchi and Veeser \cite{Sacchi} for a different approach to deal with the non-homogeneous Dirichlet boundary condition in the finite element a posteriori error analysis where the localization of the $H^{1/2}$ semi-norm also plays a crucial role.

\section{A posteriori error estimation: efficiency}\label{sec:4}

In this section we derive the lower bound of the a posteriori error estimate proved in Theorem \ref{thm:3.1}
by using the domain inverse estimate in Lemma \ref{lem:2.4}. We start with the residual $R(U)$.

\begin{lemma}\label{lem:4.1}
For any $K\in\cM$, there exists a constant $C$ independent of $p$ and $K$ such that
\ben
& &(h_K/p)\|a^{-1/2}R(U)\|_{L^2(K)}\\
&\le&C\Theta_K^{1/2} \left(p\|a^{1/2}\na_h (u-U)\|_{L^2(K)}+(h_K/p)\|a^{-1/2}(f-f_K)\|_{L^2(K)}\right),
\een
where $f_K=P_K(f|_K)$, $P_K: L^2(K)\to Q_{p-1}(K)$ is the $L^2$ projection operator and $\Theta_K$ is defined in \eqref{e1}.
\end{lemma}

\begin{proof} Without loss of generality, we only consider the case when $\Ga$ intersects with $\pa K$ at two opposite sides. We also use the notation in Lemma \ref{lem:2.7}, see Figure \ref{fig:2.3}. Denote $V=f_K+\div_h(a\na_h U)$ in $K$. Since $K_i\subset\De A_iCD\cup\De A_iB''C''$, by Lemma \ref{lem:2.3},
\be\label{ll4}
\|V\|_{L^2(K_i)}&\le&\|V\|_{L^2(\De A_iCD)}+\|V\|_{L^2(\De A_iB''C'')}\nn\\
&\le&\|V\|_{L^2(\De A_iCD)}+\mathsf{T}\left(\frac{1+3\eta_K}{1-\eta_K}\right)^{2p+3/2}\|V\|_{L^2(\De)},
\ee
where $\De=\De A_iB'C'$ which is shape regular and $h_\De\ge Ch_K$. For any $\eps>0$ sufficiently small, denote $\De_\eps=\{x\in\De:\dist(x,\pa\De)>\eps\}$ and $\chi_\eps\in C^\infty_0(\De)$ the cut-off function such that $\chi_\eps=1$ in $\De_\eps$, $0\le\chi_\eps\le 1$, and $|\na\chi_\eps|\le C\eps^{-1}$ in $\De$.

Let $v=V\chi_\eps\in H^1_0(\De)\subset H^1_0(K_i)$. Since $\De\subset K_i$ in which $a=a_i$, by the domain inverse estimate in Lemma \ref{lem:2.4}
\be\label{g1}
\|V\|_{L^2(\De)}
\le C(1+C\sqrt{\eps/h_K})^{2p}\|V\|_{L^2(\De_\eps)}.
\ee
On the other hand, since the solution $u$ satisfies \eqref{p1}-\eqref{p3},
\ben
\|V\|_{L^2(\De_\eps)}^2
&\le&\int_{\De}V^2\chi_\eps dx\\
&=&\int_{\De}(f_K+\div(a\na U))vdx\\
&=&\int_{\De}(f_K-f)vdx+\int_{\De}a\na(u-U)\cdot\na vdx.
\een
Since $\na V\in Q_{p-2}(\De)$, by the inverse estimate,
\ben
\|\na v\|_{L^2(\De)}&\le&\|\na V\|_{L^2(\De)}+C\eps^{-1}\|V\|_{L^2(\De)}\\
&\le&C(p^2h_K^{-1}+\eps^{-1})\|V\|_{L^2(\De)}.
\een
Thus if we choose $\eps=c_0h_K(p+1)^{-2}$ for some constant $c_0>0$ depending possibly on $\de_0\in (0,1/2)$ in Definition \ref{def:2.1} so that $\eps<\rho_\De/2$, where $\rho_\De$ is the radius of the maximal inscribed circle of $\De$, we obtain
\ben
\|V\|_{L^2(\De_\eps)}^2&\le&C\|a^{-1/2}(f-f_K)\|_{L^2(\De)}\|a_i^{1/2}V\|_{L^2(\De)}\\
&+&Cp^2h_K^{-1}\|a^{1/2}\na (u-U)\|_{L^2(\De)}\|a_i^{1/2}V\|_{L^2(\De)}.
\een
Noticing that $(1+C\sqrt{\eps/h_K})^{2p}\le (1+Cp^{-1})^{2p}\le C$, by \eqref{g1} we have
\ben
& &(h_K/p)\|V\|_{L^2(\De)}\\&\le&C\Theta_K^{1/2}\left((h_K/p)\|a^{-1/2}(f-f_K)\|_{L^2(\De)}+p\|a^{1/2}\na(u-U)\|_{L^2(\De)}\right).
\een
A similar argument shows the same estimate holds when $\De$ is replaced by $\De A_iCD$. This completes the proof by \eqref{ll4}. $\Box$
\end{proof}

To derive a lower bound for the jump residual, we need the following extension lemma.

\begin{lemma}\label{lem:4.2}
Let $\D$ be a bounded Lipschitz domain in $\R^d$ $(d\ge 2)$. For any $g\in H^1(\pa\D)$ and any $\eps>0$, there exists a function $v\in H^1(\D)$ such that $v=g$ on $\pa\D$, and
\ben
\|v\|_{L^2(\D)}\le C\eps\|g\|_{L^2(\pa\D)},\ \ \|\na v\|_{L^2(\D)}\le C\eps^{-1}\|g\|_{L^2(\pa\D)}+C\eps |g|_{H^1(\pa\D)},
\een
where the constant $C$ depends on the Lipschitz constant of $\pa \D$ and is independent of $v$ and $\eps$.
\end{lemma}

\begin{proof} The proof depends on the classical argument of flattening the boundary. Since $\pa\D$ is Lipschitz continuous, there is a set of sub-domains $\{U_j\}^r_{j=1}$ that covers $\pa\D$ and a partition of unity $\{\phi_j\}^r_{j=1}$ subordinated to $\{U_j\}^r_{j=1}$, that is, $\phi_j\in C^\infty_0(U_j)$, $0\le\phi_j\le 1$, $\sum^r_{j=1}\phi_j=1$ in $\cup^r_{j=1}U_j$. Moreover, there exist bi-jective Lipschitz mappings $\Phi_j:U_j\to V_j$, $V_j\subset\R^d$, such that $\Phi_j(\D\cap U_j)=\R^d_+\cap V_j$ and $\Phi_j(U_j\cap\pa\D)=\pa\R^d_+\cap V_j$, $j=1,\cdots, r$, see e.g., Evans \cite[\S C.1]{Evans}. Here $\R^d_+=\{x\in\R^d:x_d>0\}$.

For any $y=(y',0)^T\in\pa\R^d_+\cap V_j$, $j=1,\cdots,r$, let $\hat g_j(y')=g(\Phi_j^{-1}(y'))$. We define the extension of $\hat g_j$ by
\ben
\hat v_j(y',y_d)=\hat g_j(y')e^{-\frac{y_d}{\eps^2}},\ \ \forall y=(y',y_d)^T\in V_j.
\een
It is easy to see that
\ben
& &\|\hat v_j\|_{L^2(V_j\cap\R^d_+)}\le\eps\|\hat g_j\|_{L^2(\pa\R^d_+\cap V_j)},\\
& &\|\na_y\hat v_j\|_{L^2(V_j\cap\R^d_+)}\le \eps |\hat g_j|_{H^1(\pa\R^d_+\cap V_j)}+\eps^{-1}\|\hat g_j\|_{L^2(\pa\R^d_+\cap V_j)}.
\een
This completes the proof by letting $v(x)=\sum ^r_{j=1}\hat v_j(\Phi_j(x))\phi_j(x), \forall x\in \D\cap\left(\cup_{j=1}^rU_j\right)$ and $v(x)=0, \forall x\in \D\bks\left(\cup^r_{j=1}\bar U_J\right)$. $\Box$
\end{proof}

For any $K\in\cM$, let $L_K=|\Ga_K|$ and $\Phi_K:(0,L_K)\to\Ga_K$ be the arc length parametrization of $\Ga_K$. We define the $L^2$ projection $P_{\Ga_K}:L^2(\Ga_K)\to Q_p(\Ga_K)=Q_p(0,L_K)\circ\Phi_K^{-1}$ as follows: For any $g\in L^2(\Ga_K)$, $P_{\Ga_K}g\circ\Phi_K\in Q_p(0,L_K)$ such that
\ben
\int_0^{L_K}(P_{\Ga_K}g\circ\Phi_K)vds=\int_0^{L_K}(g\circ\Phi_K)vds,\ \ \forall v\in Q_p(0,L_K).
\een

\begin{lemma}\label{lem:4.3}
For any $K\in\cM^\Ga$, there exists a constant $C$ independent of $p$ and $K$ such that
\ben
& &(h_K/p)^{1/2}\|\hat a^{-1/2}J(U)\|_{L^2(\Ga_K)}\nn\\
&\le&C(h_K/p)^{1/2}\|\hat a^{-1/2}(J(U)-J_{\Ga_K}(U))\|_{L^2(\Ga_K)}\\
&&+\,C\left(p\|a^{1/2}\na_h (u-U)\|_{L^2(K)}+\|a^{-1/2}(h/p)R(U)\|_{L^2(K)}\right),
\een
where $J_{\Ga_K}(U)=P_{\Ga_K}(J(U)|_{\Ga_K})$.
\end{lemma}

\begin{proof} Let $\si=c_0(p+1)^{-2}h_K$ for some constant $c_0>0$ such that $\si$ is less than half of the minimum length of the sides of $K$, and denote $K_\si=\{x\in K:\dist(x,\pa K)>\si\}$. Let $(t_1,t_2)\subset (0,L_K)$ such that $\Phi_K$ maps $(t_1,t_2)$ to $\Ga_K\cap K_\si$. Obviously, $t_1\le C_1\si, L_K-t_2\le C_2\si$ for some constants $C_1,C_2>0$. Since $J_{\Ga_K}\circ\Phi_K^{-1}\in Q_p(0,L_K)$, we use the domain inverse Lemma \ref{lem:2.2} to obtain
\be\label{f1}
\|J_{\Ga_K}(U)\|_{L^2(\Ga_K)}&\le&C\|J_{\Ga_K}(U)\circ\Phi_K^{-1}\|_{L^2(0,L_K)}\nn\\
&\le&C\mathsf{T}(1+C\si/L_K)^{2p+1}\|J_{\Ga_K}(U)\circ\Phi_K^{-1}\|_{L^2(t_1,t_2)}\nn\\
&\le&C\|J_{\Ga_K}(U)\|_{L^2(\Ga_K\cap K_\si)},
\ee
where we have used the fact that $\mathsf{T}(\lam)=1+\sqrt{\lam-1}(\sqrt{\lam-1}+\sqrt{1+\lam})$ and $(1+C\sqrt{\si/h_K})^{2p+1}=(1+Cp^{-1})^{2p+1}\le C$ for some constant $C$ independent of $p$.

Since $J_{\Ga_K}(U)\circ\Phi_K^{-1}\in Q_p(0,L_K)$, by the inverse estimate we have
\be\label{f2}
\|\na_\Ga J_{\Ga_K}(U)\|_{L^2(\Ga_K)}
&\le&C|J_{\Ga_K}(U)\circ\Phi_K^{-1}|_{H^1(0,L_K)}\nn\\
&\le&Cp^2L_K^{-1}\|J_{\Ga_K}(U)\circ\Phi_K^{-1}\|_{L^2(0,L_K)}\nn\\
&\le&Cp^2h_K^{-1}\|J_{\Ga_K}(U)\|_{L^2(\Ga_K)}.
\ee
Let $\chi_\si\in C^\infty_0(K_\si)$ be the cut-off function satisfying $\chi_\si=1$ in $K_{\si}$, $0\le\chi\le 1$, $|\na \chi_\si|\le C\si^{-1}$ in $K$. Let $v_\si\in H^1(\Om)$ be such that $v_\si|_{\Om_i}\in H^1(\Om_i), i=1,2,$ is the extension of $J_{\Ga_K}(U)\chi_\si\in H^1(\Ga)$ defined in Lemma \ref{lem:4.2} with $\eps=\sqrt{h_K/p}$, then
\ben
\|v_\si\|_{L^2(\Om_i)}\le C(h_K/p)^{1/2}\|J_{\Ga_K}(U)\chi_\si\|_{L^2(\Ga)}\le C(h_K/p)^{1/2}\|J_{\Ga_K}(U)\|_{L^2(\Ga_K)},
\een
and
\ben
& &\|\na v_\si\|_{L^2(\Om_i)}\\
&\le&C(h_K/p)^{-1/2}\|J_{\Ga_K}(U)\chi_\si\|_{L^2(\Ga_K)}+C(h_K/p)^{1/2}\|\na_\Ga(J_{\Ga_K}(U)\chi_\si)\|_{L^2(\Ga_K)}\\
&\le&Cp(h_K/p)^{-1/2}\|J_{\Ga_K}(U)\|_{L^2(\Ga_K)},
\een
where we have used \eqref{f2} in the last inequality. Let $w_\si=v_\si\chi_\si$. Then $w_\si\in H^1_0(K)$ and satisfies
\be
& &\|w_\si\|_{L^2(K)}\le C(h_K/p)^{1/2}\|J_{\Ga_K}(U)\|_{L^2(\Ga_K)},\label{f3} \\
& &\|\na w_\si\|_{L^2(K)}\le Cp(h_K/p)^{-1/2}\|J_{\Ga_K}(U)\|_{L^2(\Ga_K)}.\label{f4}
\ee
Now by \eqref{f1}
\ben
\|J_{\Ga_K}(U)\|^2_{L^2(\Ga_K)}&\le&C\int_{\Ga_K}|J_{\Ga_K}(U)|^2\chi_\si^2=C
\int_{\Ga_K}J_{\Ga_K}(U)\cdot w_\si.
\een
By using the equation \eqref{p1}-\eqref{p3} and integration by parts
\ben
\|J_{\Ga_K}(U)\|^2_{L^2(\Ga_K)}&\le&C\|J(U)-J_{\Ga_K}(U)\|_{L^2(\Ga_K)}\|w_\si\|_{L^2(\Ga_K)}+\int_{\Ga_K}J(U)w_\si \\
&\le&C\|(h/p)^{1/2}(J(U)-J_{\Ga_K}(U))\|_{L^2(\Ga_K)}\|J_{\Ga_K}(U)\|_{L^2(\Ga_K)}\\
&&+\,\int_Ka\na_h(U-u)\cdot\na w_\si
-\int_KR(U)w_\si.
\een
This completes the proof by using \eqref{f3}-\eqref{f4}. $\Box$
\end{proof}

The following lemma can be proved by the method in Lemma \ref{lem:4.3}. We omit the details.

\begin{lemma}\label{lem:4.4}
For any $e\in\cE^{\rm side}$, $e=\pa K\cap\pa K'$, $K,K'\in\cM$, we have
\ben
& &\|\hat a^{-1/2}(h/p)^{1/2}J(U)\|_{L^2(e)}\\
&\le&C\|\hat a^{-1/2}(h/p)^{1/2}(J(U)-P_eJ(U))\|_{L^2(e)}+Cp\|a^{1/2}\na_h(u-U)\|_{L^2(K\cup K')}\\
&&+\,C\|a^{-1/2}(h/p)R(U)\|_{L^2(K\cup K')}+C\|\hat a^{-1/2}(h/p)^{1/2}J(U)\|_{L^2(\Ga_K\cup\Ga_{K'})},
\een
where $P_e:L^2(e)\to Q_p(e)$ is the $L^2$ projection operator.
\end{lemma}

Let $\mathbb{P}:\Pi_{K\in\cM}L^2(K)\to\V_{p-1}(\cM)$ be defined elementwise as $\mathbb{P}|_K=P_K$
and $\mathbb{Q}:\Pi_{e\in\cE}L^2(e)\to Q_p(e)$ be defined as $\mathbb{Q}|_e=P_e$.

The following theorem which is the main result of this section can be proved by combining
Lemma \ref{lem:4.1}, Lemma \ref{lem:4.3} and Lemma \ref{lem:4.4}.

\begin{theorem}\label{thm:4.1}
Let $u\in H^1(\Om)$ be the weak solution of \eqref{p1}-\eqref{p3} and $U\in\mathbb{X}_p(\cM)$ be the solution of \eqref{a2}. We have
\ben
& &\|a^{-1/2}\Theta^{-1}(h/p)R(U)\|_\cM+\|\al^{-1/2}(h/p)^{1/2}J(U)\|_\cE\\
&\le&C\left(p\|a^{1/2}\na_h(u-U)\|_\cM+{\rm osc}(f,U,\Ga)\right).
\een
where ${\rm osc}(f,U,\Ga)$ is the data oscillation defined as
\ben
{\rm osc}(f,U,\Ga)=\|a^{-1/2}(h/p)(f-\mathbb{P} f)\|_{\cM}+\|\hat a^{-1/2}(h/p)^{1/2}(J(U)-\mathbb{Q}J(U))\|_\cE.
\een
\end{theorem}

We remark that the factor $p$ in the front of $\|a^{1/2}\na_h(u-U)\|_\cM$ is well-known for residual type $hp$ a posteriori error estimates, see Melenk and Wohlmuth \cite{Melenk01}, in which $hp$ a posteriori error estimation was first studied for elliptic equations on conforming meshes based on polynomial inverse estimates. Our argument is different by using the domain inverse estimate and is slightly better in the sense that the additional factor $p^\eps$ in the local lower bound in \cite{Melenk01} is removed in our analysis.

\section{Numerical examples}\label{sec:5}

In this section, we present several numerical examples to illustrate the performance of the proposed adaptive unfitted finite element method.
The computations are carried out using MATLAB on a workstation with Intel(R) i9-9900 CPU 2.70GHz and 64GB memory. The basis functions of $Q_p(K)$ are the Lagrange interpolation polynomials
through the local Gauss-Lobatto-Legendre (GLL) integration points in each element $K$.

For each $K\in\cM$, we compute the local a posteriori error estimator $\xi_K$ as in Theorem \ref{thm:3.1} and define
the global a posteriori error estimate $\eta=\left ( \sum_{K\in \cM} \xi_K^2 \right )^{1/2}$.

We first describe the adaptive unfitted finite element algorithm.

\bigskip
{\bf Algorithm 5.1} Given a tolerance TOL $>0$, $N_0\ge 1$ a fixed number, and an initial conforming Cartesian mesh $\cT$.
\begin{enumerate}
	\item
	Construct the induced mesh $\cM$ by Algorithm 5.2 so that each element $K$ in $\cM$ is large with respect to both $\Om_1$, $\Om_2$ and satisfies \eqref{HH}.
	\item
	Solve the discrete problem \eqref{a2} on $\mathcal{M}$.
	\item
	Compute the local error estimator $\xi_K$ on each $K\in \mathcal{M}$ and the global error estimate $\eta$.
	\item
	While  $\  \eta >  \mbox{TOL } $   do
	\begin{itemize}
		\item
		Mark the elements in $\hat{\mathcal{M}} \subset \cM$ such that:
		\begin{equation*}
		\left ( \sum_{K\in \hat{\mathcal{M}}} \xi_K^2 \right )^{1/2} \geq \frac{1}{2}\,\eta.
		\end{equation*}
		\item
		Refine the elements in $\hat\cT=\{K\in\cT:K\subset K', K'\in\hat\cM\}$ by quad refinement to obtain a new mesh $\dhat\cT$.
		\item
		Refine $\dhat\cT$ to obtain a new mesh $\cT$ such that each side of $\cT$ includes at most $N_0$ hanging nodes, which makes $\cT$ a $K$-mesh satisfying the Assumption (H3).
		\item
		Construct the induced mesh $\cM$ by Algorithm 5.2 so that each element $K\in\cM$ is large with respect to both $\Om_1$, $\Om_2$, and $\cM$ satisfies \eqref{HH}.
		\item
		Solve the discrete problem \eqref{a2} on $\cM$.
		\item
		Compute the local error estimator $\xi_K$ on each $K \in \cM$ and the global
		error estimate $\eta$.
	\end{itemize}
	end while
\end{enumerate}

The following algorithm is used to construct the induced mesh $\cM$ from a Cartesian mesh $\cT$ so that each element $K$ in $\cM$ is large with respect to both $\Om_1$, $\Om_2$ and $\cM$ satisfies \eqref{HH}.
We use the notation $\cT_i^{\rm large}:=\{K\in\cT: \text{$K$ is large with respect to $\Om_i$}\}, i=1,2$, according to Definition 2.1
with the parameter $\de_0\in (0,1/2)$.

\bigskip
{\bf Algorithm 5.2} Given $\delta_0\in(0,1/2)$, $N_0\ge 1$ a fixed number, and a Cartesian mesh $\cT$. 
\begin{enumerate}
	\item
	Mark all small elements in $\cT_{\rm small}\subset\cT$, where
	\begin{equation*}
	\cT_{\rm small}=\{K\in\cT: K\cap\Ga\not=\emptyset,K\not\in\cT_1^{\rm large}\cap\cT_2^{\rm large}\}.
	\end{equation*}
	\item
	If $\cT_{\rm small}\not=\emptyset$, for each $K\in\cT_{\rm small}$, $K\not\in\cT_i^{\rm large}$, $i=1,2$, do 
		\begin{itemize}
			\item If $K$ has a neighboring element $K'\in\cT^{\rm large}_i$ whose size is the same as that of $K$ and the minimum rectangle containing $K,K'$ is large with respect to $\Om_i$, then merge $K$ and $K'$.
			\item Else if $K$ has a neighboring element $K'\in\cT^{\rm large}_i$ whose size is larger than that of $K$, add $K'$ to $\cT_{\rm refine}$.
			\item Else if $K$ has a neighboring element $K'\in\cT^{\rm large}_i$ whose size is smaller than that of $K$, add $K$ to $\cT_{\rm refine}$.
			\item Otherwise, add $K$ and all its neighboring elements in $\cT_i^{\rm large}$ to $\cT_{\rm refine}$ .
		\end{itemize}
	\item 
	While $K\in\cT\backslash\cT_{\rm small}$, $\eta_K>\max(1/2,(1-\de_0)/(1+\de_0))$, do $i=1,2$
	     \begin{itemize}
	         \item If $K$ does not include singular points of $\Ga$ or $K$ is an irregular large element with respect to $\Om_i$, add $K$ to $\cT_{\rm refine}$.
		\item Else if $K$ has two vertices in $\Om_i$ and there exists a neighboring element $K'\subset\Om_i$ whose size is the same as that of $K$, then merge $K$ and $K'$.
		\item Else if $K$ has three vertices in $\Om_i$ and there exist three neighboring elements $K',K'',K'''\subset\Om_i$ whose sizes are the same as that of $K$, then merge $K$ and $K',K'',K'''$.
		\item Otherwise, add the elements with the largest size among $K$ and its neighboring elements to $\cT_{\rm refine}$. 
	     \end{itemize}
	end while
	\item
	If $\cT_{\rm refine}\neq\emptyset$, refine the elements in  $\cT_{\rm refine}$ and their neighboring elements to obtain a new mesh $\cT$ such that each side of $\cT$ includes at most $N_0$ hanging nodes, go to 1.
\end{enumerate} 

We remark that if each side of a mesh $\cT$ includes at most $N_0$ hanging nodes, the induced mesh $\cM$ from $\cT$ by Algorithm 5.2 is also a $K$-mesh with the constant $C$ in Assumption (H3) depending only on $N_0$. 

Now we present three examples to demonstrate the efficiency of our adaptive algorithm. We consider the case of high contrast coefficient $a(x)$ in Example 2 and the case of non-smooth interface in Example 3. 

In all examples we set the computational domain $\Omega = (-2,2)\times(-2,2)$. In our theory, the penalty parameter $\al_0$ can be any fixed positive constant and the constant $\de_0$ in Definition \ref{def:2.1} can be any constant in $(0,1/2)$. Clearly, a larger $\de_0$ will lead to more small elements to be merged with neighboring elements. Here we take the natural choice $\alpha_0=1$ and $\delta_0=1/4$. We always set the maximal number of hanging nodes in each side of the mesh $N_0=3$. 

\begin{figure}[htb]
	\begin{center}
		\begin{tabular}{ccc}
			\epsfxsize=0.3\textwidth\epsffile{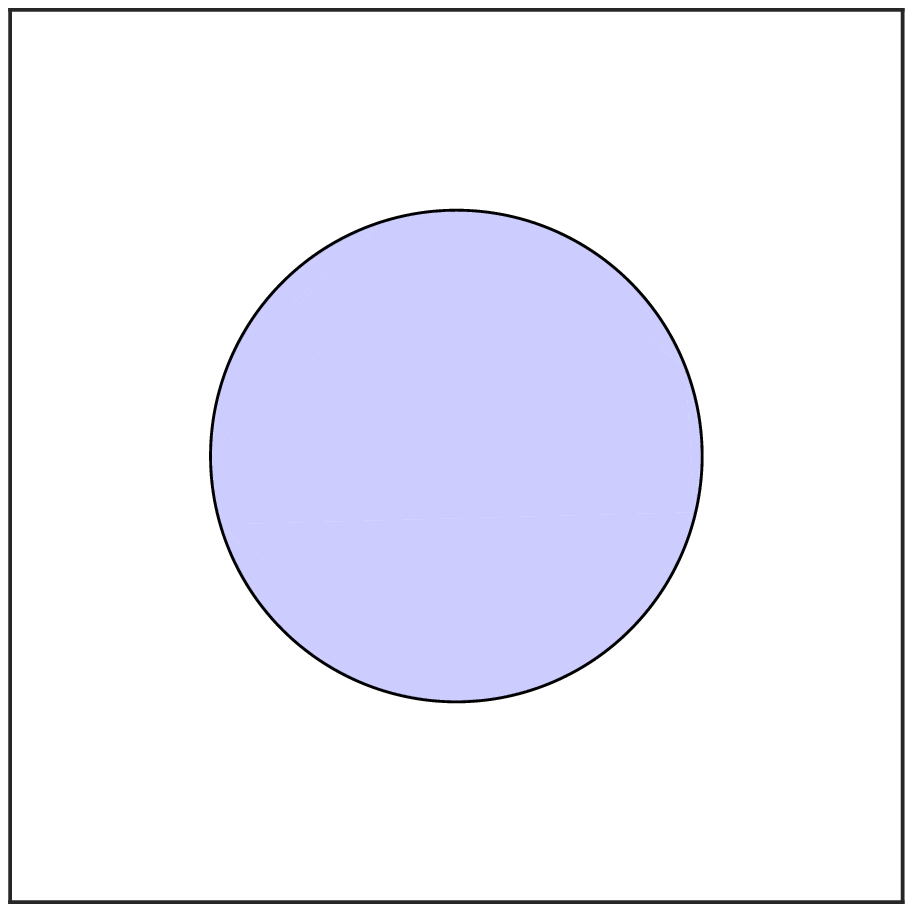} &
			\epsfxsize=0.3\textwidth\epsffile{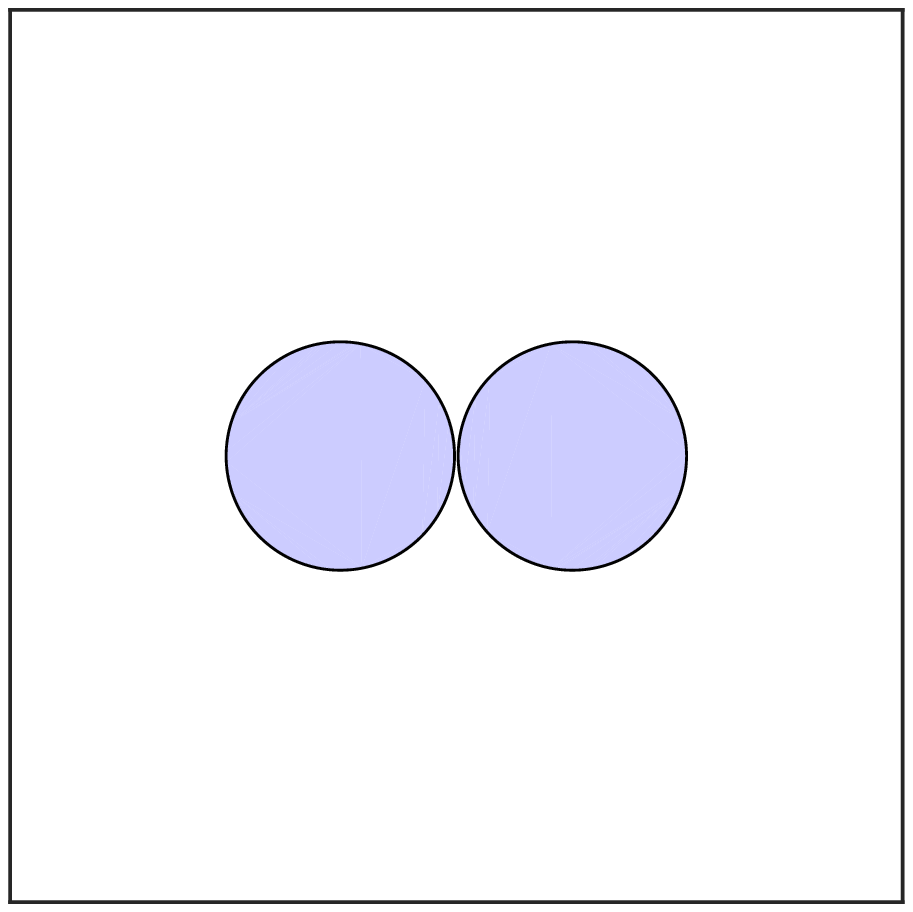} &
			\epsfxsize=0.3\textwidth\epsffile{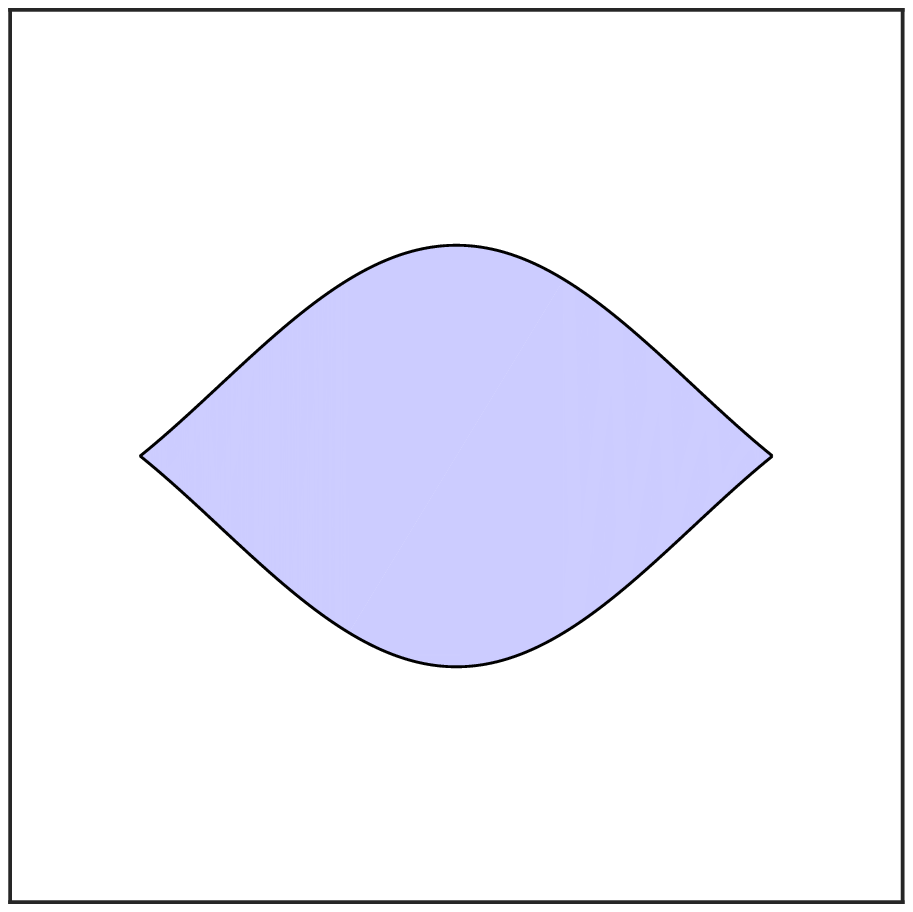}\\
			(a) & (b) & (c)
		\end{tabular}
	\end{center}
	\vspace{-0.5cm}
	\caption{The interface used in Example 1 (left), Example 2 (middle), and Example 3 (right).}\label{fig:5.0}
\end{figure}

\bigskip
{\bf Example 1.} We first consider a problem whose exact solution is known to illustrate the effectivity index of the a posteriori error estimate.
Let the interface $\Gamma$ be the circle centered at $(0,0)^T$ with radius $r=1.1$. We define $\Om_1=\{x\,:\,|x| < r\}$ and $\Om_2=\Om \backslash \bar{\Om}_1$, as shown in Figure \ref{fig:5.0} (a). Set $a_1= 10$ and $a_2 = 1$. The right-hand side $f$ and boundary condition $g$ are computed such that the exact solution is
\ben
u(x) =\begin{cases}
	e^{|x|^2-r^2} + 10r^2 - 1,\quad &\mbox{if }|x| \leq r, \\
	10|x|^2, \quad &\mbox{otherwise}.
\end{cases}
\een

Figure \ref{fig:5.3} depicts the surface plot of the exact solution and one discrete solution. Figure \ref{fig:5.1} shows the quasi-optimal decay $O(N^{-p/2})$ of both the error $\|u-U\|_{DG}$ and the a posterior error estimate $\eta$ for $p=1,2,3$, respectively. Effectivity indexes ${\rm eff}=\eta/\|u-U\|_{DG}$ for $p=1,2,3$ are evaluated in Figure \ref{fig:5.2}, which keep nearly constant as the number of degrees of freedom (\#DoFs) increases. 

In Table \ref{tab:5.0}, we display \#DoFs, $\eta$, and ${\rm eff}$ of uniform refinements and adaptive refinements. Figure \ref{fig:5.4} shows some examples of adaptive meshes and corresponding zoomed meshes.  It is clear that much less number of degrees of freedom are needed to reach nearly the same error when using higher order methods. We remark that using higher degree polynomials yields higher accuracy but requires more computational cost. Appropriate balance of these two factors in practical computations is an interesting question that requires further investigation. 

\begin{figure}[htb]
	\begin{center}
		\begin{tabular}{cc}
			\epsfxsize=0.45\textwidth\epsffile{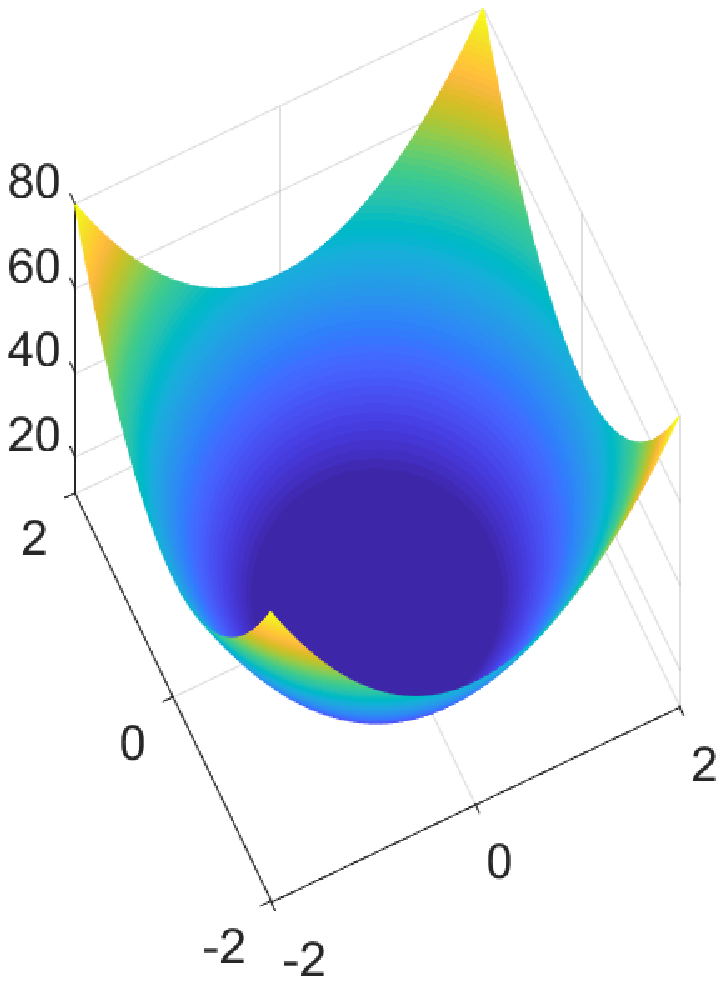} &
			\epsfxsize=0.45\textwidth\epsffile{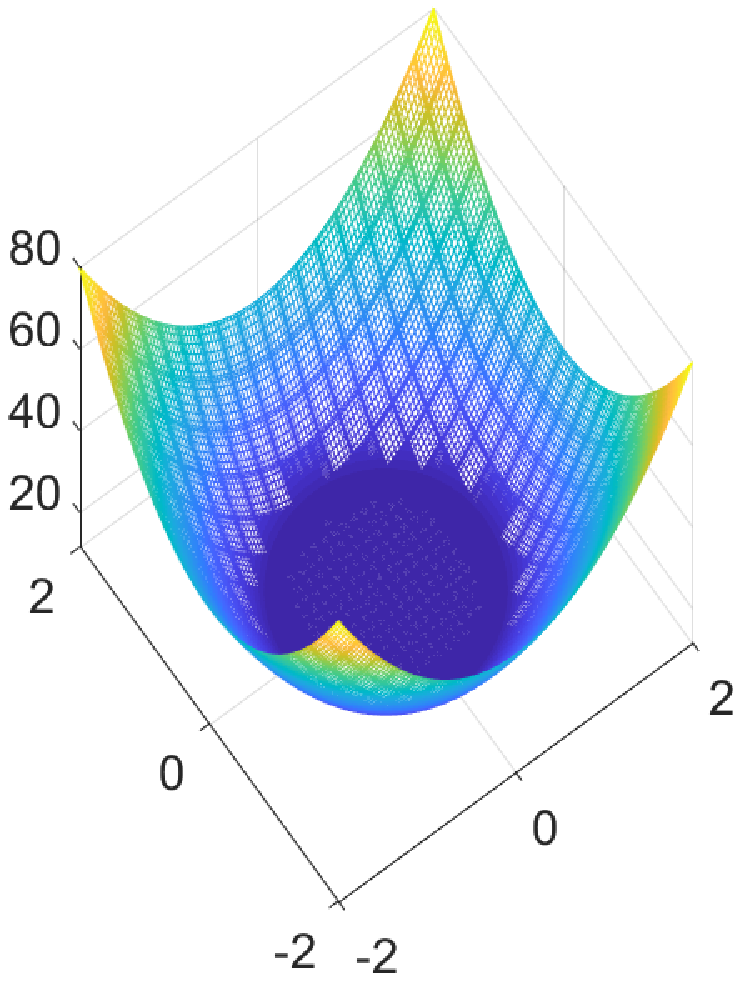}  \\
			(a) & (b) \\
		\end{tabular}
	\end{center}
	\vspace{-0.5cm}
	\caption{Example 1: (a) The exact solution. (b) The discrete solution on the mesh of $4184$ elements when $p=3$. }\label{fig:5.3}
\end{figure}

\begin{figure}[htb]
	\begin{center}
		\begin{tabular}{cc}
			\epsfxsize=0.45\textwidth\epsffile{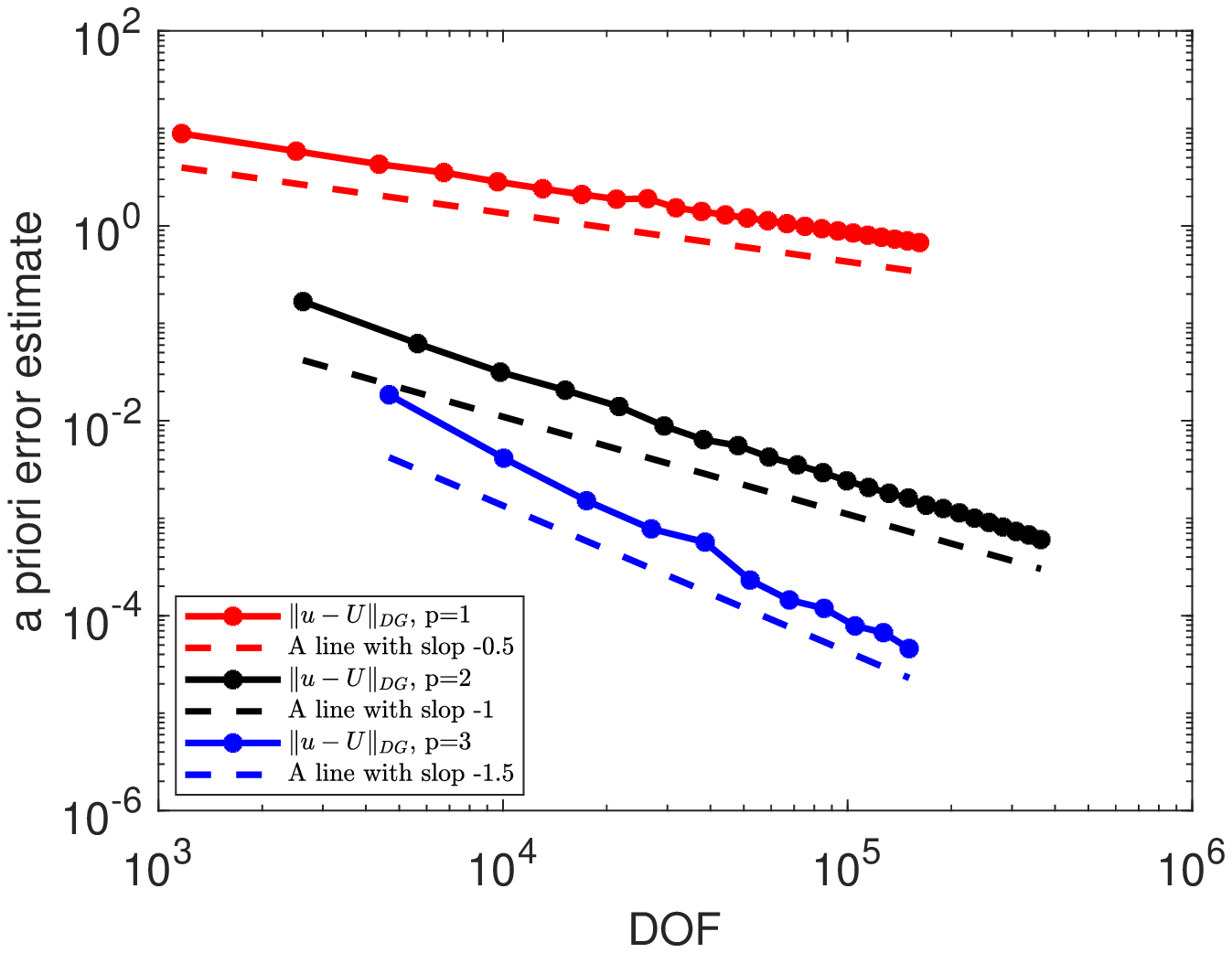} &
			\epsfxsize=0.45\textwidth\epsffile{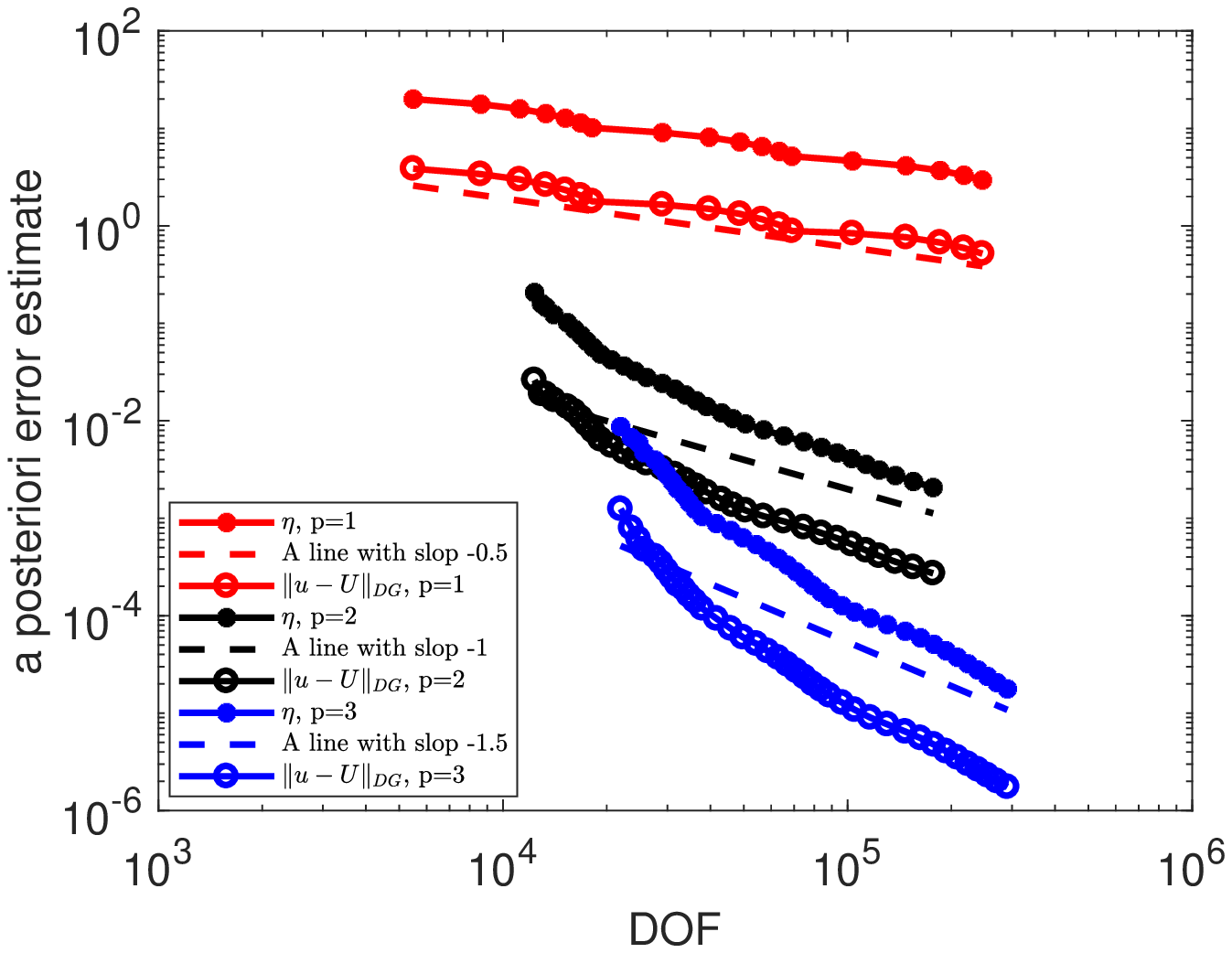}\\
			(a) & (b)
		\end{tabular}
	\end{center}
	\vspace{-0.5cm}
	\caption{Example 1: (a) The error $\|u-U\|_{\rm DG}$ for $p=1,2,3$ by uniform refinements. (b) A priori and a posterior error estimates $\eta$ for $p=1,2,3$ by adaptive refinements.}\label{fig:5.1}
\end{figure}
\begin{figure}[htb]
	\begin{center}
		\epsfxsize=0.5\textwidth\epsffile{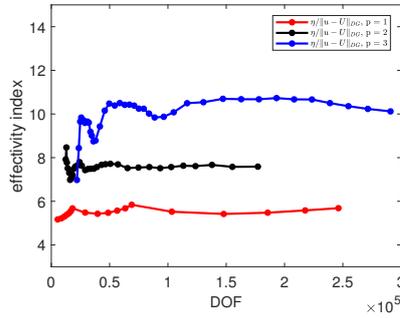}
	\end{center}
	\vspace{-0.5cm}
	\caption{Example 1: The effectivity index ${\rm eff}=\eta/\|u-U\|_{DG}$ against the degrees of freedom for $p=1,2,3$. }\label{fig:5.2}
\end{figure}
\begin{table}[htb]
	\caption{Comparison between uniform refinements and adaptive refinements.}\label{tab:5.0}
	\centering
	\begin{tabular}{cccccc}
		\hline
		\multicolumn{5}{c}{$p=1$} \\
		\hline
		Refinement Strategy &  \#DoFs & $\|u-U\|_{\rm DG}$ &  $\eta$ & eff \\
		\hline
		Uniform    &  103792  &  8.43e-1    & - &  - \\
		\hline
		Adaptive    & 103344  &  8.39e-1    & 4.63 &  5.52  \\
		\hline
	\end{tabular}
	\begin{tabular}{cccccc}
		\hline
		\multicolumn{5}{c}{$p=2$} \\
		\hline
		Refinement strategy &  \#DoFs &  $\|u-U\|_{\rm DG}$ &  $\eta$ & eff \\
		\hline
		Uniform    & 363852  &  6.04e-4    & - &  -  \\
		\hline
		Adaptive    & 93357  &  6.21e-4    & 4.67e-3 &  7.52  \\
		\hline
	\end{tabular}
	\begin{tabular}{cccccc}
		\hline
		\multicolumn{5}{c}{$p=3$} \\
		\hline
		Refinement strategy &  \#DoFs & $\|u-U\|_{\rm DG}$ &  $\eta$ & eff \\
		\hline
		Uniform    & 150848  &  4.60e-5    & - & -  \\
		\hline
		Adaptive    & 59704  &  4.32e-5    & 4.54e-4 &  10.50  \\
		\hline
	\end{tabular}	
\end{table}

\begin{figure}[htb]
	\begin{center}
		\begin{tabular}{ccc}
			\epsfxsize=0.5\textwidth\epsffile{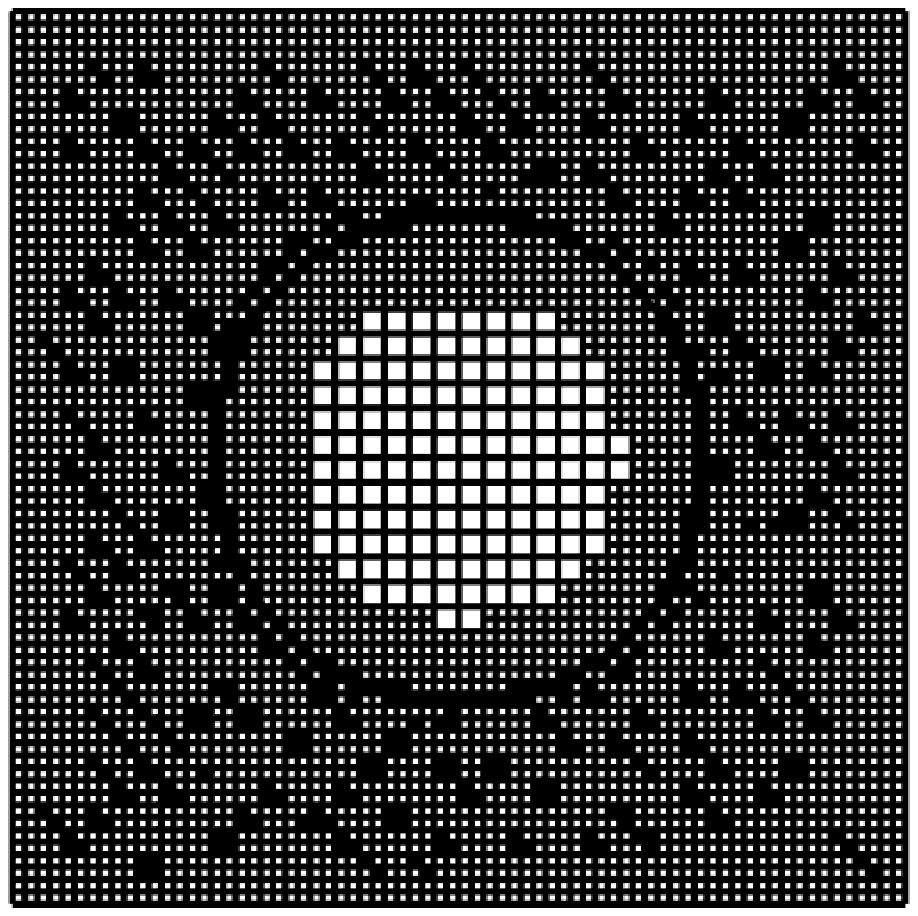} &
			\epsfxsize=0.5\textwidth\epsffile{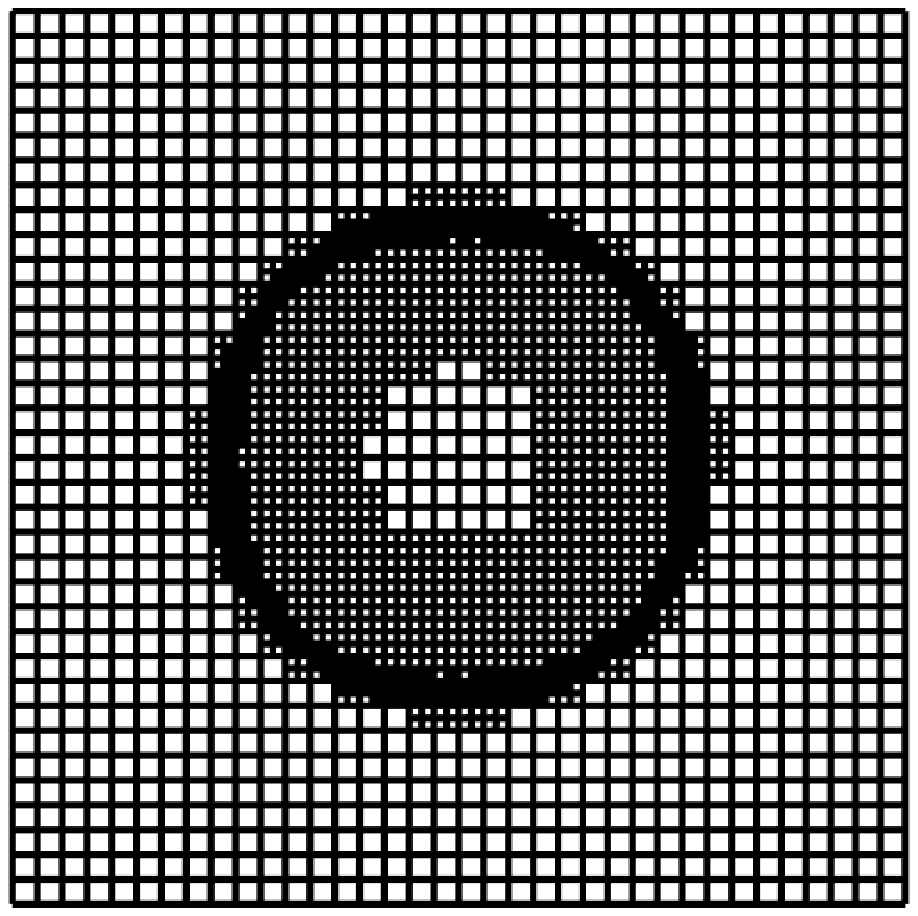} \\
			(a) & (b)\\
			\epsfxsize=0.5\textwidth\epsffile{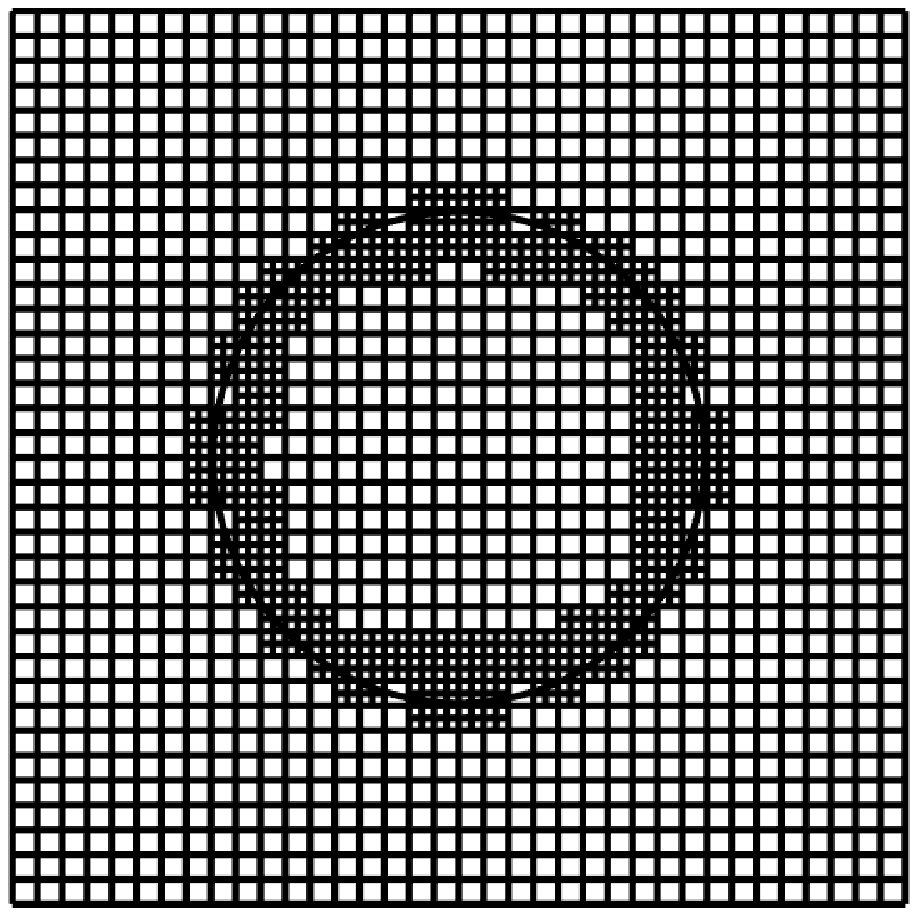} &
			\epsfxsize=0.5\textwidth\epsffile{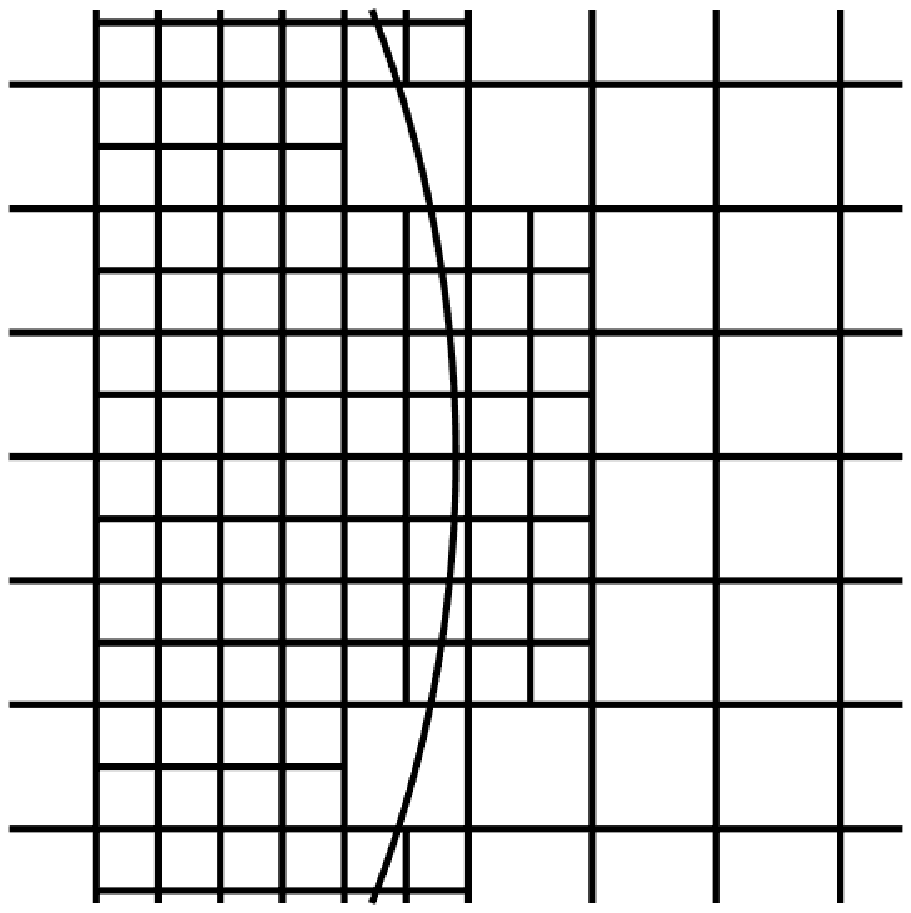}\\
			(c) & (d)\\
		\end{tabular}
	\end{center}
	\caption{Example 1: Adaptive meshes.
		(a) Mesh for $p=1$, \#DoFs=$29036$, $\|u-U\|_{\rm DG} = 1.6532$, and $\eta=9.0598$. (b) Mesh for $p=2$, \#DoFs=$31923$, $\|u-U\|_{\rm DG} = 2.7700e-3$, and  $\eta=2.0696e-2$.
		(c) Mesh for $p=3$, \#DoFs=$31088$, $\|u-U\|_{\rm DG} = 2.4554e-4$, and $\eta=2.3696e-3$. (d) The corresponding local mesh for $p=3$ within $(0.7,1.5)\times(-0.4,0.4)$.}\label{fig:5.4}
\end{figure}

{\bf Example 2.} In this example, we assume the interface $\Ga$ to be the union of two closely located circles of radius $r=0.51$. The distance between two circles is $d=0.02$. $\Om_{1}$ is the union of the interior of the two disks and $\Om_2=\Om \backslash \bar{\Om}_1$ (see Figure \ref{fig:5.0} (b)). To evaluate the effect of high contrast coefficients, we set $a_1= 100, a_2 = 1$. We set $f=1$ and $g=0$.

Although $a_1$ is fairly large, the quasi-optimal decay of the global a posteriori error estimate for $p=1,2,3$ is observed (Figure \ref{fig:5.5}). Figure \ref{fig:5.6} shows some examples of the adaptive meshes and the zoomed meshes. The discrete solution on the mesh of 2855 elements is shown in Figure \ref{fig:5.7} (a).

\begin{figure}[htb]
	\begin{center}
		\begin{tabular}{ccc}
			\epsfxsize=0.6\textwidth\epsffile{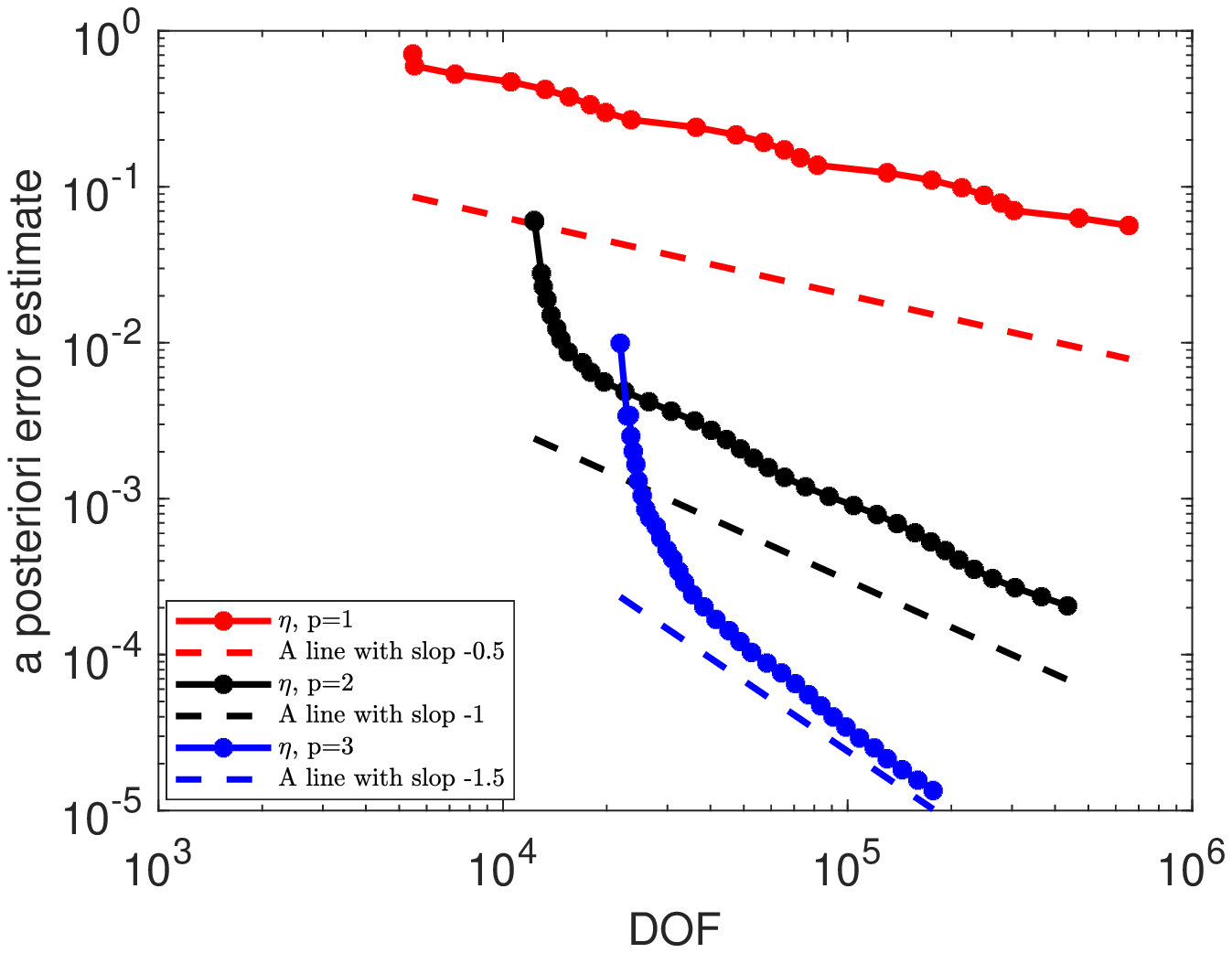}
		\end{tabular}
	\end{center}
	\vspace{-0.5cm}
	\caption{Example 2: The quasi-optimal decay of the a posteriori error estimate $\eta$ for $p=1,2,3$. }\label{fig:5.5}
\end{figure}

\begin{figure}[htb]
	\begin{center}
		\begin{tabular}{ccc}
			\epsfxsize=0.5\textwidth\epsffile{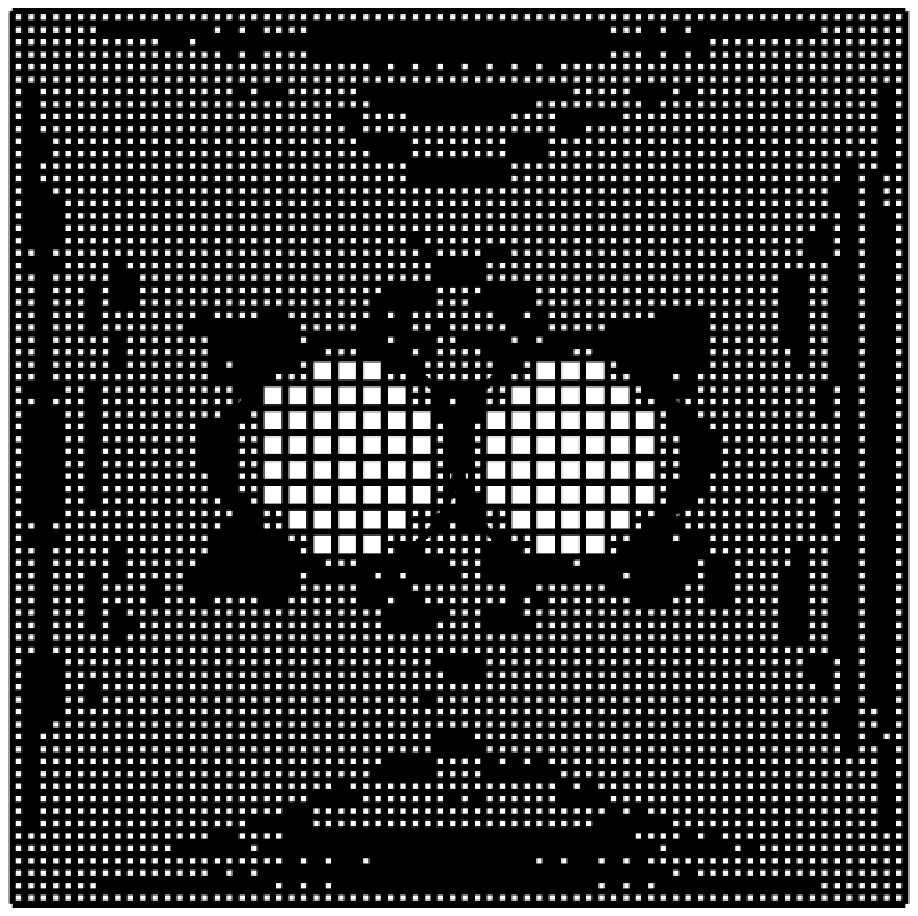} &
			\epsfxsize=0.5\textwidth\epsffile{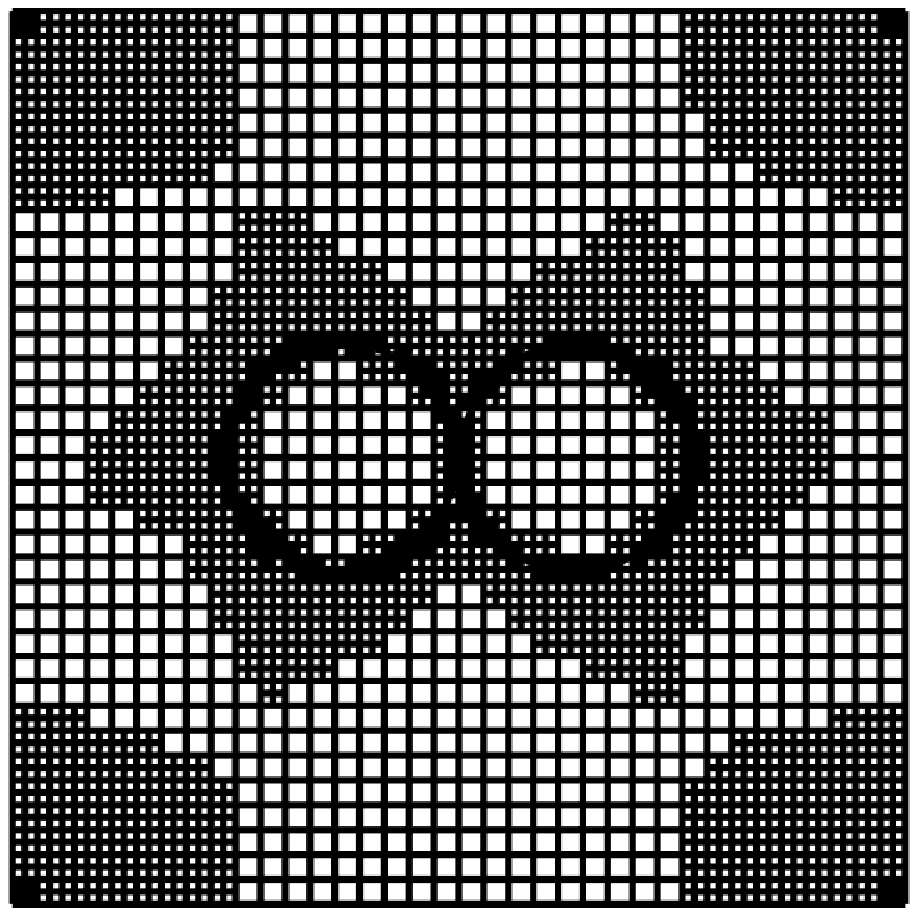} \\
			(a) & (b) \\
			\epsfxsize=0.5\textwidth\epsffile{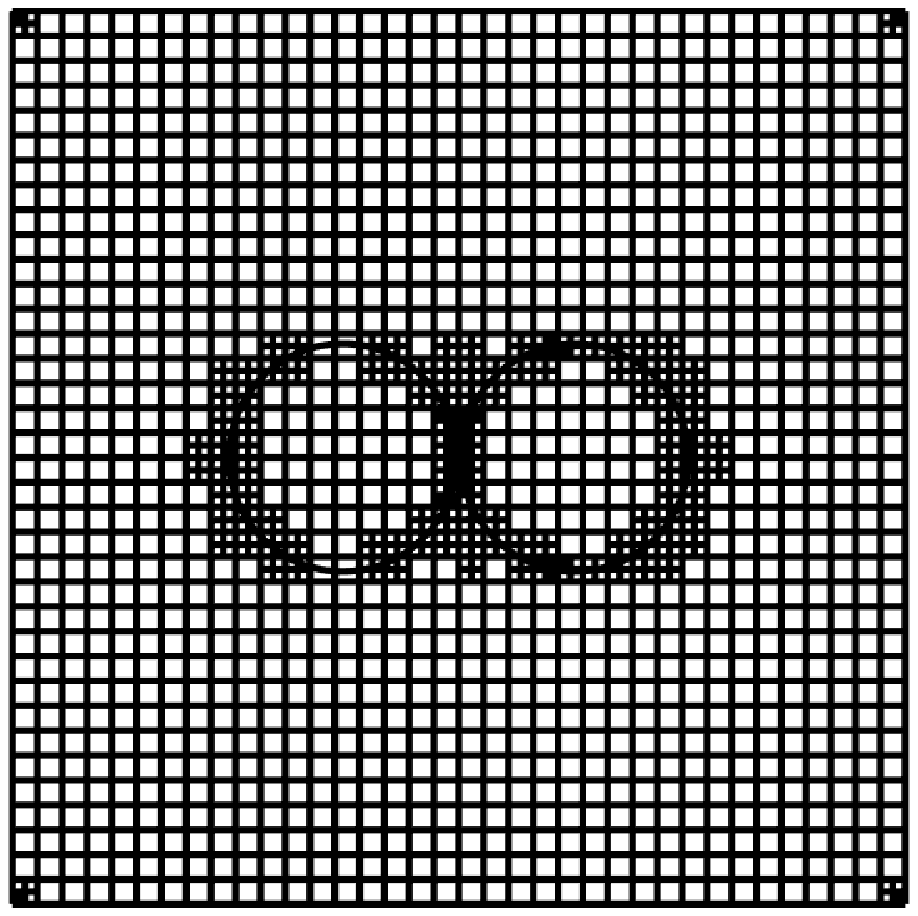}   &
			\epsfxsize=0.5\textwidth\epsffile{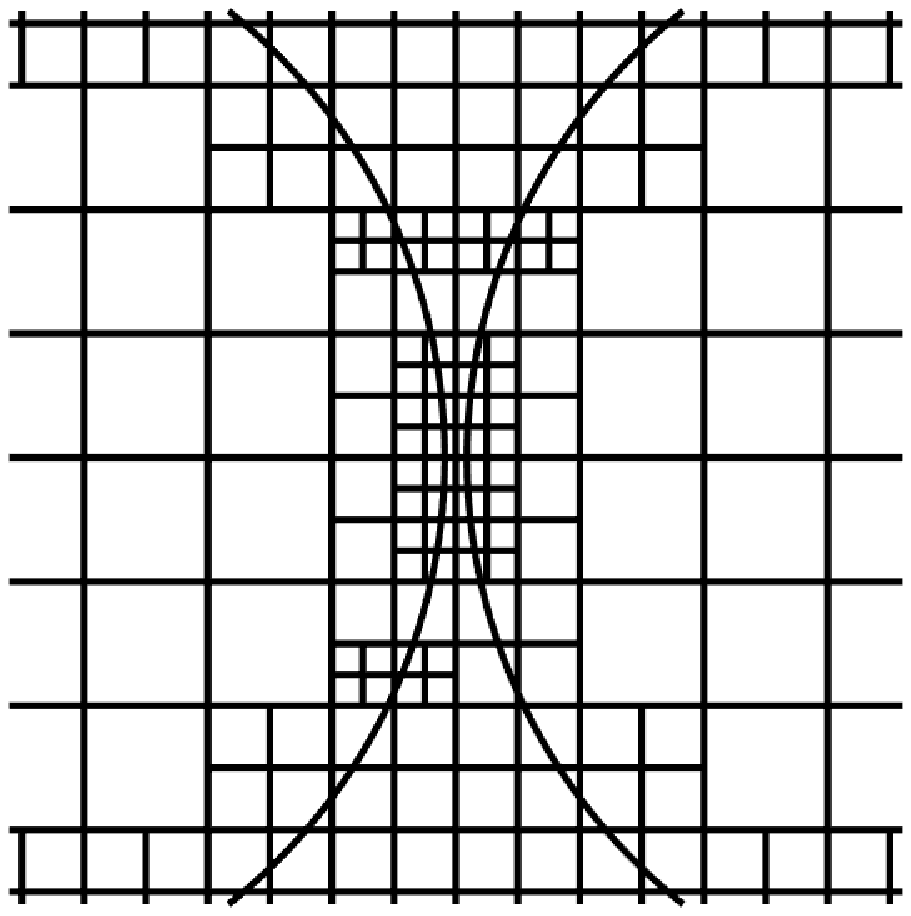}\\
			(c) & (d) \\
		\end{tabular}
	\end{center}
	\caption{Example 2: Adaptive meshes.
		(a) Mesh for $p=1$, \#DoFs=$35116$, and $\eta=2.4384e-1$.
		(b) Mesh for $p=2$, \#DoFs=$35235$, and $\eta=3.2069e-3$.
		(c) Mesh for $p=3$, \#DoFs=$30304$, and $\eta=4.6183e-4$. (d) The corresponding local mesh for $p=3$ within $(-0.4,0.4)\times(-0.4,0.4)$}\label{fig:5.6}
\end{figure}

\begin{figure}[htb]
	\begin{center}
		\begin{tabular}{cc}
			\epsfxsize=0.45\textwidth\epsffile{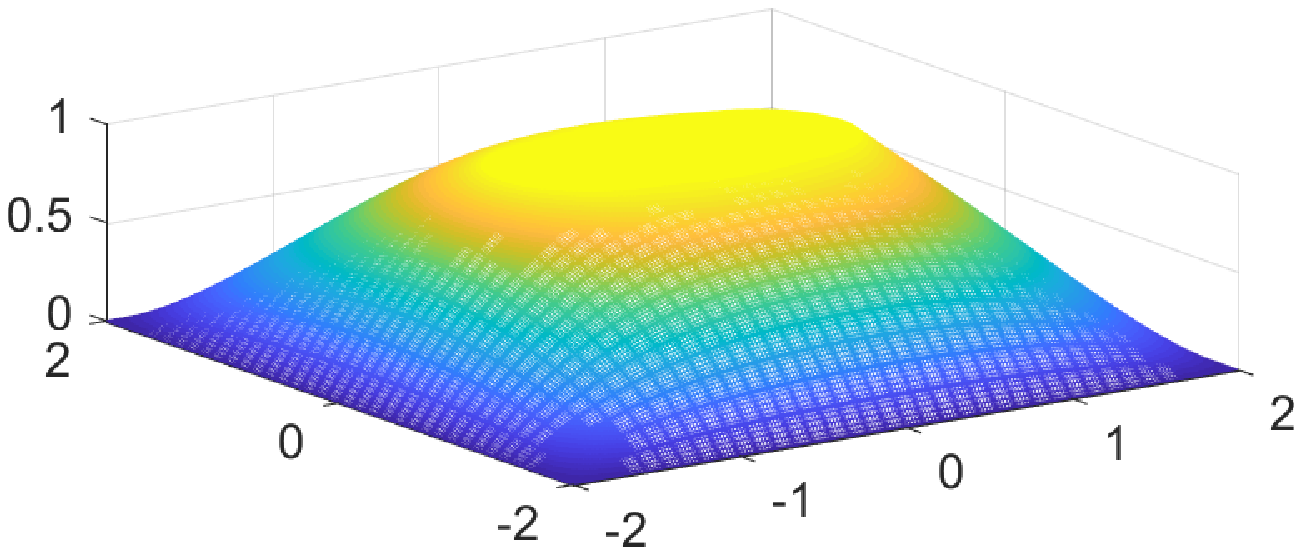} &
			\epsfxsize=0.45\textwidth\epsffile{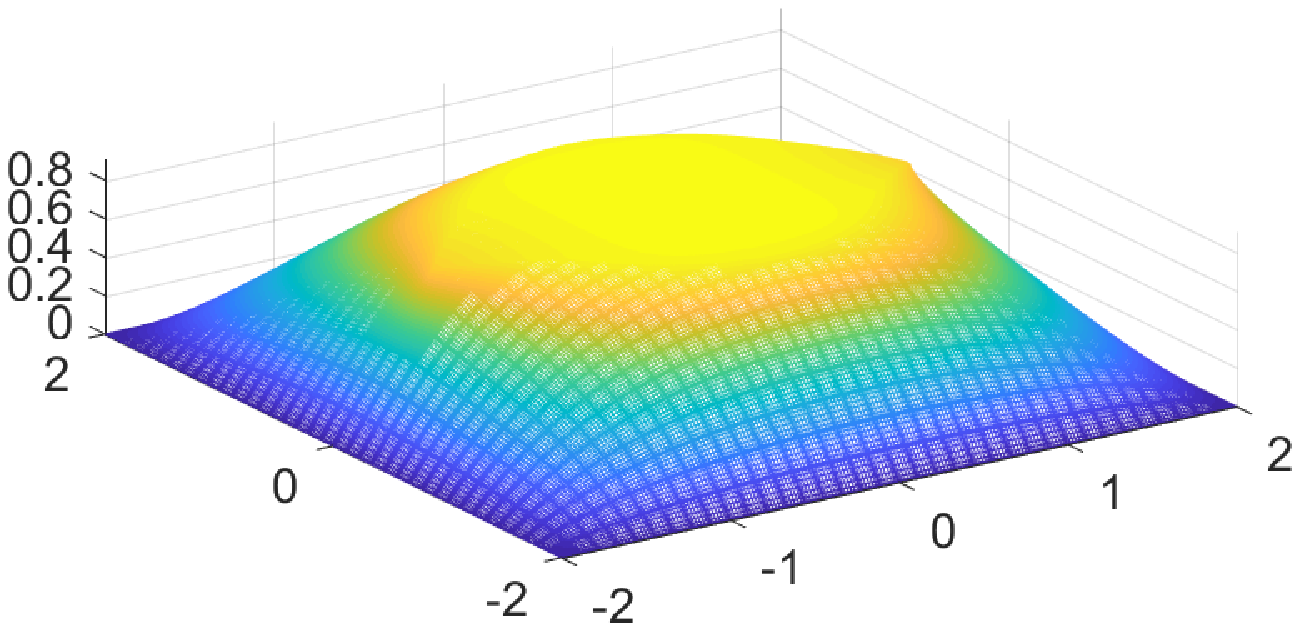}  \\
			(a) & (b) \\
		\end{tabular}
	\end{center}
	\vspace{-0.5cm}
	\caption{(a) Example 2: The discrete solution on the mesh of $2855$ elements for $p=3$. (b) Example 3: The discrete solution on the mesh of $2749$ elements for $p=3$.  }\label{fig:5.7}
\end{figure}

{\bf Example 3.} We consider a non-smooth interface defined by
\begin{equation*}
\Ga=\left\{(x,y):|y|=\frac{4\sqrt{2}}{9}\cos\left(\frac{\sqrt{2}\pi}{3}x\right)
+\frac{2\sqrt{2}}{9}\right\}.
\end{equation*}
Note that the interface is singular at the points $(\pm\sqrt{2},0)$ (see \ref{fig:5.0} (c)). We set $a_1= 10, a_2 = 1$, the right-hand side $f=1$ and boundary condition $g=0$.

The quasi-optimal decay of the a posteriori error estimate are clearly observed in Figure \ref{fig:5.8}. Figure \ref{fig:5.9} shows some examples of the adaptive meshes and parts of the zoomed meshes for $p=1,2,3$, respectively. We observe that the meshes are mainly refined around the sharp corners where the solution is singular. The discrete solution on the mesh of 2749 elements is depicted in Figure \ref{fig:5.7} (b).

\begin{figure}[htb]
	\begin{center}
		\begin{tabular}{ccc}
			\epsfxsize=0.6\textwidth\epsffile{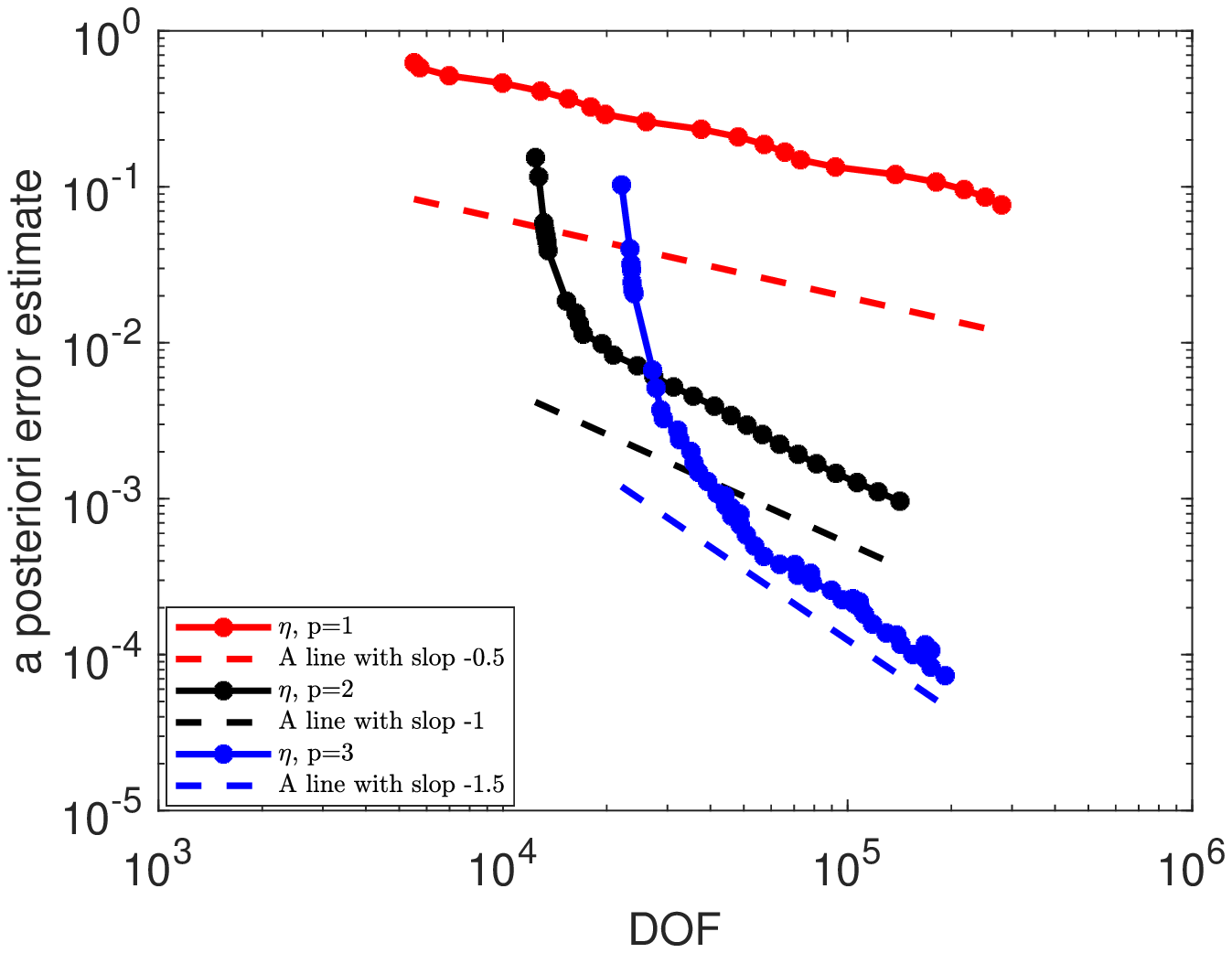}
		\end{tabular}
	\end{center}
	\vspace{-0.5cm}
	\caption{Example 3: The quasi-optimal decay of the a posteriori error estimate $\eta$ for $p=1,2,3$. }\label{fig:5.8}
\end{figure}


\begin{figure}[htb]
	\begin{center}
		\begin{tabular}{ccc}
			\epsfxsize=0.5\textwidth\epsffile{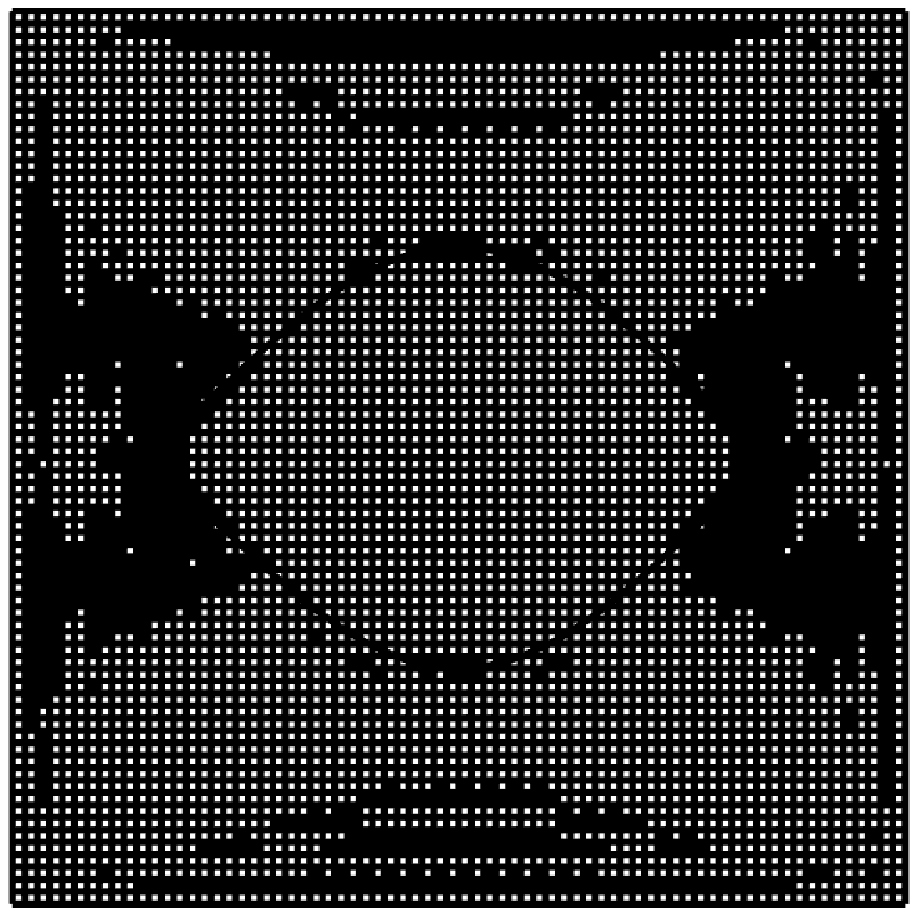} &
			\epsfxsize=0.5\textwidth\epsffile{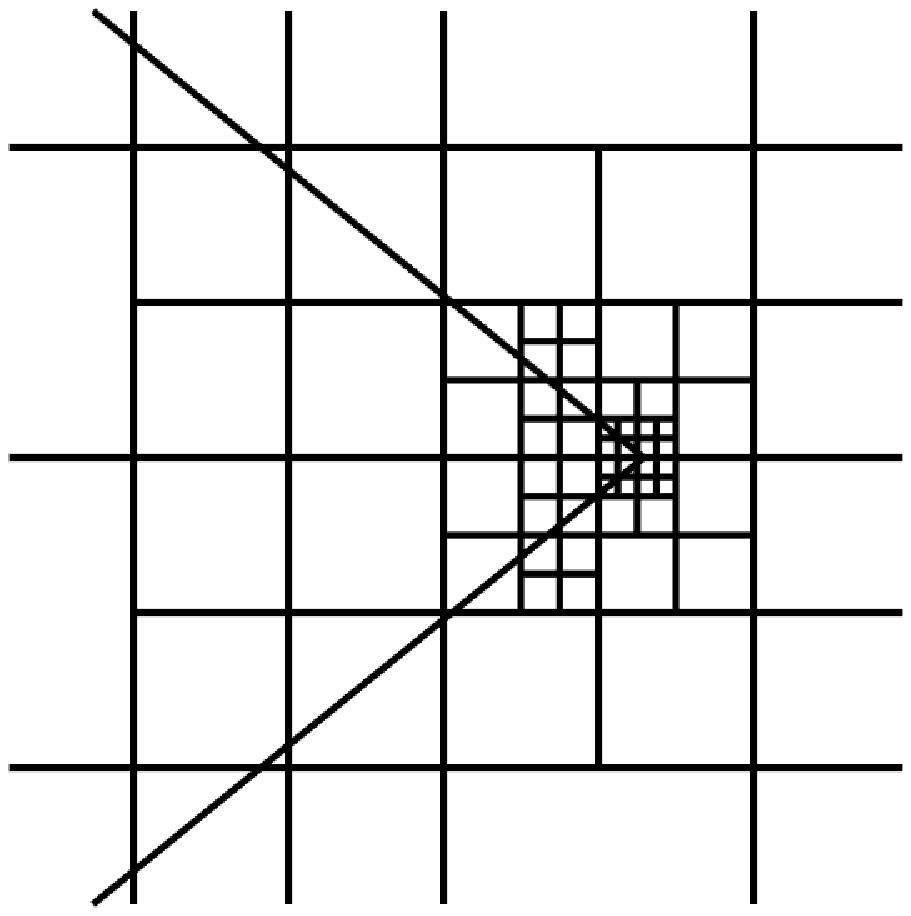}\\
			(a) & (b)\\
			\epsfxsize=0.5\textwidth\epsffile{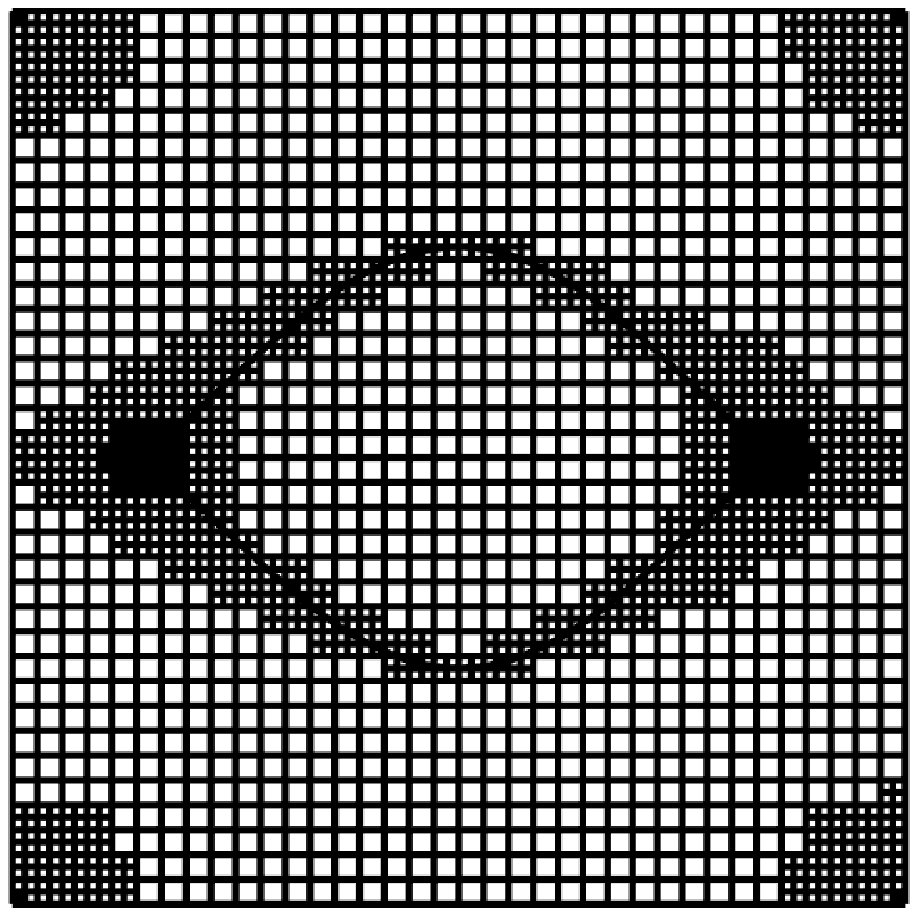} &
			\epsfxsize=0.5\textwidth\epsffile{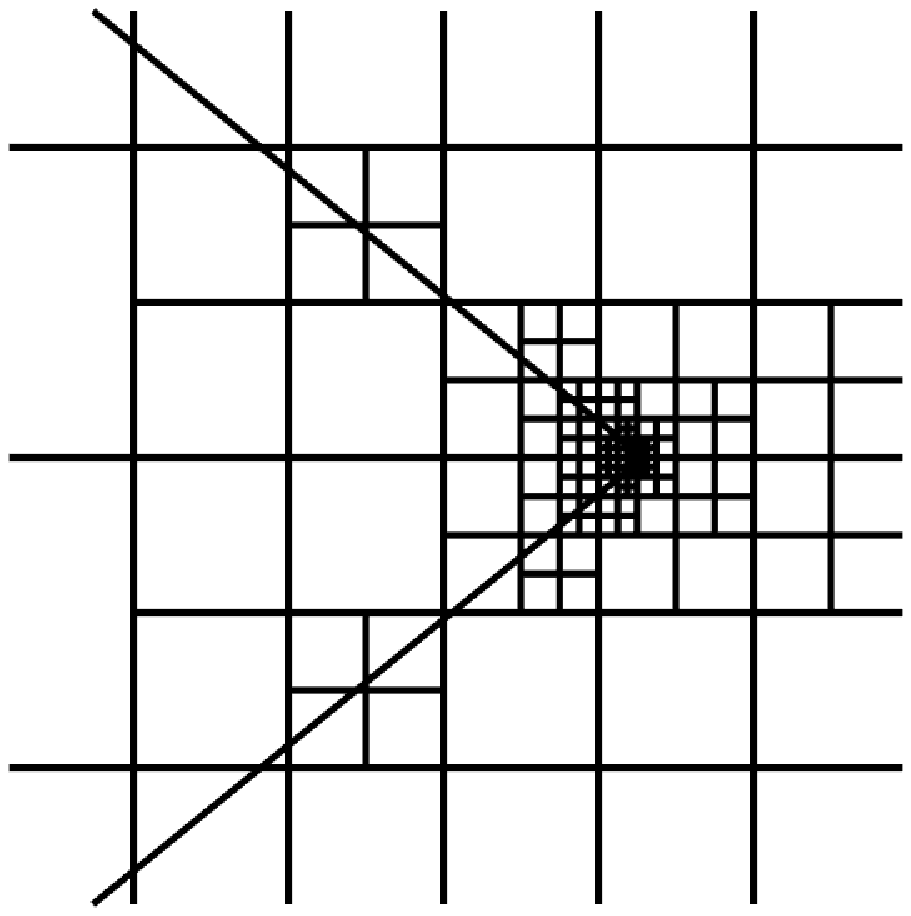}\\
			(c) & (d)\\
			\epsfxsize=0.5\textwidth\epsffile{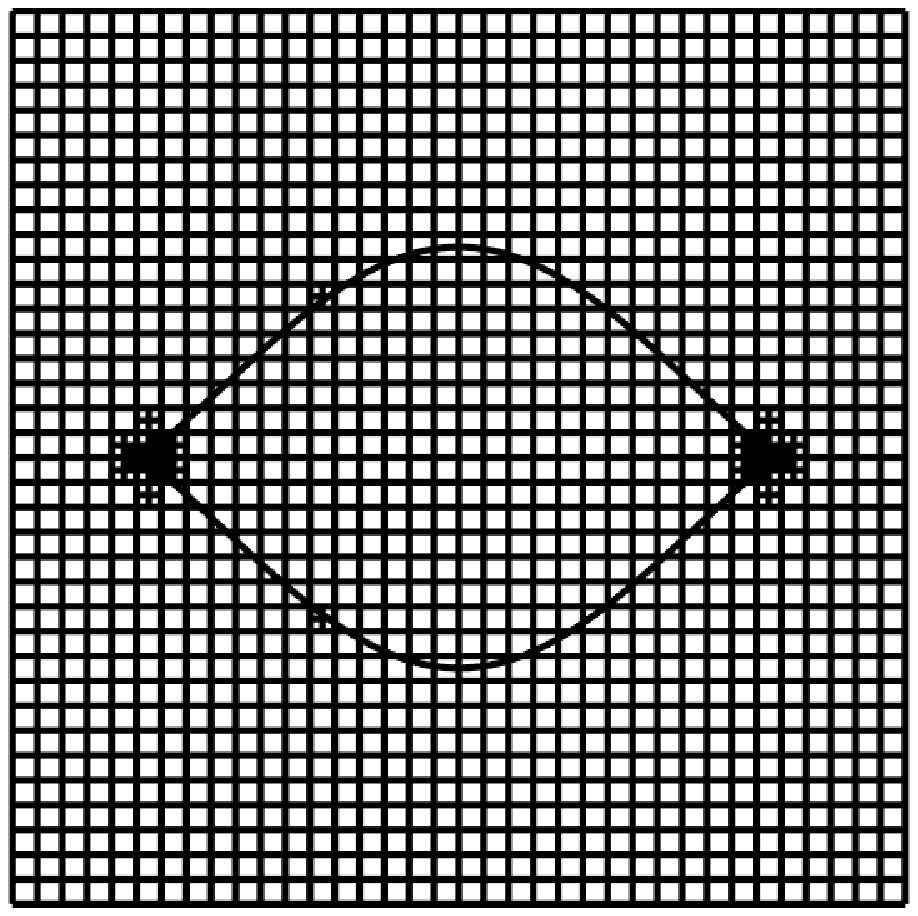}   &
			\epsfxsize=0.5\textwidth\epsffile{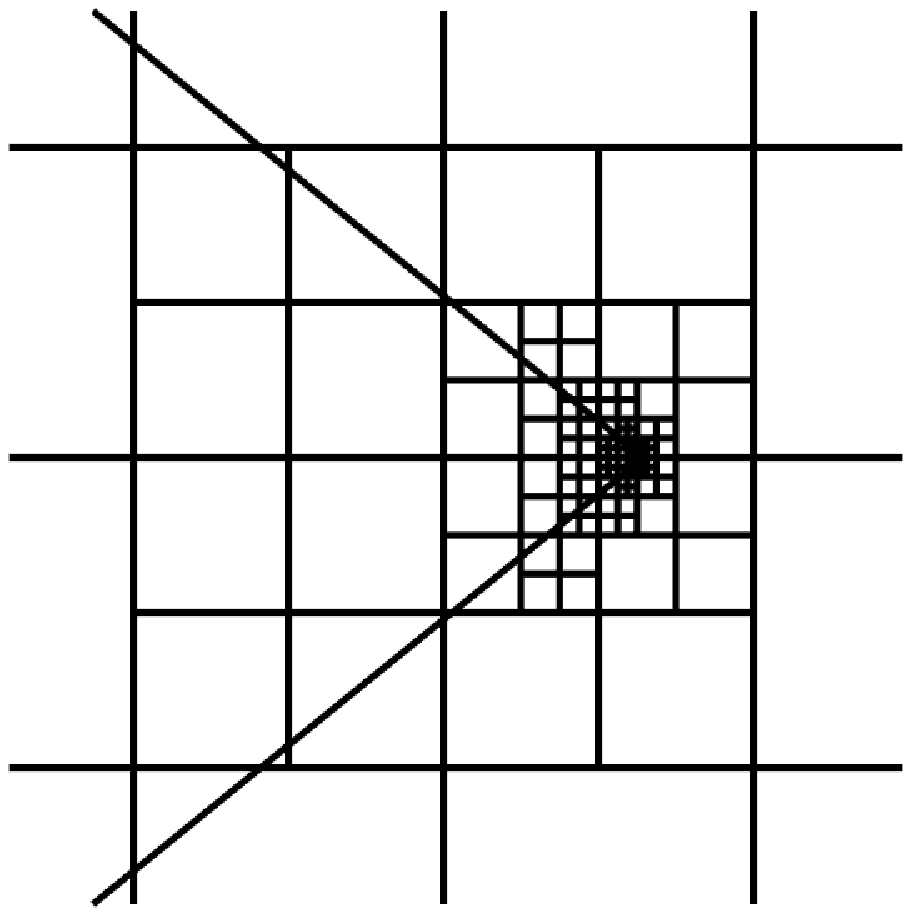}\\
			(e) & (f)
		\end{tabular}
	\end{center}
	\caption{Example 3: Adaptive meshes (left) and corresponding local meshes within $(1.40,1.42)\times(-0.01,0.01)$ (right).
		(a)\&(b) The case $p=1$, \#DoFs=$37684$, and $\eta=2.3391e-1$.
		(c)\&(d) The case $p=2$, \#DoFs=$31302$, and $\eta=5.2027e-3$.
		(e)\&(f) The case $p=3$, \#DoFs=$32128$, and $\eta=2.7504e-3$.  }\label{fig:5.9}
\end{figure}

\begin{acknowledgements}
The authors are very grateful to the referees for the constructive and helpful comments which lead to great improvement of the paper. The first author gratefully acknowledges the support and hospitality of the Program on ``Numerical Analysis
of Complex PDE Models in Sciences" in Erwin Schr\"odinger International Institute for Mathematics and Physics in Universit\"at Wien during June 25-29 and July 16-20, 2018. 
\end{acknowledgements}


\end{document}